\documentclass[a4paper, reqno]{amsart}
\usepackage{amsmath, amsthm, amssymb, tikz, tikz-cd, enumerate, stmaryrd}
\usepackage[hidelinks]{hyperref}
\usepackage[shortalphabetic]{amsrefs}
\usepackage[top=30mm, bottom=30mm, left=25mm, right=25mm]{geometry}

\tikzcdset{row sep/normal=2.7em}
\tikzcdset{column sep/normal=3.0em}

\let\hat\widehat

\newcommand{\mb}[1]{\mathbb{#1}}
\newcommand{\mc}[1]{\mathcal{#1}}
\newcommand{\mf}[1]{\mathfrak{#1}}
\newcommand{\mrm}[1]{\mathrm{#1}}
\newcommand{\tn}[1]{\textnormal{#1}}
\newcommand{\bun}{\mrm{Bun}}
\newcommand{\tbun}{\widetilde{\bun}}
\renewcommand{\sslash}{{/\mkern-6mu/}}
\newcommand{\B}{\mb{B}}
\newcommand{\Deg}{{\mf{D}\mrm{eg}}}
\newcommand{\piccat}{{\mc{P}\!\mathit{ic}}}
\newcommand{\mmid}{{\,|\,}}
\newcommand{\sym}{\mrm{Sym}}
\newcommand{\id}{\mrm{id}}
\newcommand{\ul}[1]{\underline{#1}}
\newcommand{\longhookrightarrow}{\ensuremath{\lhook\joinrel\relbar\joinrel\rightarrow}}
\newcommand{\longdashrightarrow}{{\ensuremath{\,\,\begin{tikzpicture}\path[->] (0, 0) edge[dashed] (0.6, 0);\end{tikzpicture}\,\,}}}
\newcommand{\bslash}{\mathchoice
{\reflectbox{$\displaystyle /$}}%
{\reflectbox{$\textstyle /$}}%
{\reflectbox{$\scriptstyle /$}}%
{\reflectbox{$\scriptscriptstyle /$}}}
\newcommand{\liset}{\mathit{lisse}\tn{-}\mathit{\acute{e}t}}

\renewcommand{\hom}{\mrm{Hom}}

\DeclareMathOperator{\spec}{\mrm{Spec}}
\DeclareMathOperator{\spf}{\mrm{Spf}}
\DeclareMathOperator{\rank}{\mrm{rank}}

\DeclareMathOperator{\codim}{\mrm{codim}}

\theoremstyle{plain}
\newtheorem{thm}{Theorem}[subsection]
\newtheorem{prop}[thm]{Proposition}
\newtheorem{lem}[thm]{Lemma}
\newtheorem{cor}[thm]{Corollary}
\theoremstyle{definition}
\newtheorem{defn}[thm]{Definition}
\newtheorem{rmk}[thm]{Remark}

\newtheorem{eg}[thm]{Example}

\numberwithin{equation}{subsection}

\author{Dougal Davis}
\email{dougal.davis@ed.ac.uk}
\address{School of Mathematics, The University of Edinburgh, James Clerk Maxwell Building, The King's Buildings, Peter Guthrie Tait Road, Edinburgh, EH9 3FD, UK}
\title[The elliptic Grothendieck-Springer resolution]{The elliptic Grothendieck-Springer resolution as a simultaneous log resolution of algebraic stacks}

\begin{document}
\maketitle


\begin{abstract}
Let $G$ be a simply connected simple algebraic group. The classical multiplicative and additive Grothendieck-Springer resolutions are simultaneous resolutions of singularities for the maps from $G$ and its Lie algebra to their invariant theory quotients by the conjugation action of $G$. In this paper, we construct an elliptic version of the Grothendieck-Springer resolution, which is a simultaneous log resolution of a family whose total space is the stack of principal $G$-bundles on an elliptic curve. Our construction extends a well-known simultaneous resolution of the coarse moduli space map for semistable principal bundles to the stack of all principal bundles. We also prove elliptic versions of the Chevalley isomorphism and the Kostant and Steinberg section theorems, on which our construction relies.


\end{abstract}

\tableofcontents

\section{Introduction} \label{section:introduction}

Let $G$ be a simply connected simple algebraic group over an algebraically closed field $k$. Classically, the Springer theory of $G$ is the study of various features of the additive and multiplicative adjoint quotient maps
\begin{equation} \label{eq:introadjointquotient}
\chi^{add} \colon \mf{g} \longrightarrow \mf{g}\sslash G = \spec k[\mf{g}]^G \quad \text{and} \quad \chi^{mul} \colon G \longrightarrow G \sslash G = \spec k[G]^G,
\end{equation}
where $\mf{g} = \mrm{Lie}(G)$ is the Lie algebra of $G$ and $G$ acts on $\mf{g}$ (resp., $G$) via the adjoint representation (resp., by conjugation). For example, when $G = SL_n$, $\mf{g}\sslash G$ and $G \sslash G$ are both affine spaces of dimension $n - 1$, and $\chi^{add}$ (resp., $\chi^{mul}$) sends a matrix of trace $0$ (resp., determinant $1$) to the nontrivial coefficients of its characteristic polynomial.

The key observation of Springer theory is that both $\chi^{add}$ and $\chi^{mul}$ are flat families of varieties admitting simultaneous resolutions of singularities after pulling back along finite coverings of their targets. More precisely, there are isomorphisms $\mf{g}\sslash G \cong \mf{t}\sslash W$ and $G\sslash G \cong T \sslash W$, due to Chevalley, and commutative diagrams
\begin{equation} \label{eq:introgrothendieckspringer}
\begin{tikzcd}
\tilde{\mf{g}} = G \times^B \mf{b} \ar[r, "\psi^{add}"] \ar[d, "\tilde{\chi}^{add}"'] & \mf{g} \ar[d, "\chi^{add}"] \\
\mf{t} \ar[r] & \mf{t} \sslash W
\end{tikzcd}
\quad \text{and} \quad
\begin{tikzcd}
\tilde{G} = G \times^B B \ar[r, "\psi^{mul}"] \ar[d, "\tilde{\chi}^{mul}"'] & G \ar[d, "\chi^{mul}"] \\ T \ar[r] & T \sslash W,
\end{tikzcd}
\end{equation}
where $T \subseteq B \subseteq G$ are a maximal torus and Borel subgroup respectively, $\mf{t} = \mrm{Lie}(T)$ and $\mf{b} = \mrm{Lie}(B)$ their Lie algebras, and $W = N_G(T)/T$ is the Weyl group. The diagrams \eqref{eq:introgrothendieckspringer} are called the \emph{additive} (aka rational) and \emph{multiplicative} (aka trigonometric) \emph{Grothendieck-Springer resolutions}. Assuming that $k$ has characteristic $0$ in the additive case, they are simultaneous resolutions in the sense that $\tilde{\chi}^{add}$ and $\tilde{\chi}^{mul}$ are smooth, $\psi^{add}$ and $\psi^{mul}$ are proper, and for all $t \in \mf{t}$ (resp., $T$), the morphism $(\tilde{\chi}^{add})^{-1}(t) \to (\chi^{add})^{-1}(tW)$ (resp., $(\tilde{\chi}^{mul})^{-1}(t) \to (\chi^{mul})^{-1}(tW)$) is a resolution of singularites.

The central idea of elliptic Springer theory is to replace the stack quotients $\mf{g}/G$ and $G/G$ (of which $\mf{g}$ and $G$ are charts, and $\mf{g}\sslash G$ and $G \sslash G$ are coarse moduli spaces) with the stack $\bun_G$ of principal $G$-bundles on a smooth elliptic curve $E$. Remarkably, many constructions from additive and multiplicative Springer theory have direct analogues in this context.

One of the earliest incarnations of this idea can be found in P. Slodowy's work on singularities associated with loop groups (e.g., \cite{slodowy82}), culminating in the paper \cite{helmke-slodowy04} with S. Helmke. In \cite{helmke-slodowy04}, $G$-bundles on elliptic curves appear via the isomorphism of complex analytic stacks
\begin{align}
\mc{L}G/_q\mc{L}G &\overset{\sim}\longrightarrow \bun_G^{an}(\mb{C}^\times/q^\mb{Z}) \label{eq:looijengaloopgroup} \\
\varphi &\longmapsto \frac{\mb{C}^\times \times G}{(qz, g) \sim (z, \varphi(z)g)} \nonumber
\end{align}
where $q$ is a fixed complex number with $0 < |q| < 1$, $\bun_G^{an}(\mb{C}^\times/q^\mb{Z})$ is the analytic stack of $G$-bundles on the elliptic curve $\mb{C}^\times/q^\mb{Z}$, and $\mc{L}G$ is the group of holomorphic maps $\varphi \colon \mb{C}^\times \to G$ acting on itself by \emph{$q$-twisted conjugation}
\[ (\theta \cdot \varphi)(z) = \theta(z)\varphi(z)\theta(qz)^{-1}.\]
The $\mb{C}^\times$-action on $\mc{L}G$ given by $(\lambda \cdot \varphi)(z) = \varphi(\lambda z)$ lifts to an action on the universal central extension $1 \to \mb{C}^\times \to \tilde{\mc{L}}G \to \mc{L}G \to 1$ of $\mc{L}G$. The characters of irreducible representations of the semi-direct product $\hat{\mc{L}}G = \tilde{\mc{L}}G \rtimes \mb{C}^\times$ are used to define a morphism
\begin{equation} \label{eq:helmkeslodowyadjointquotient}
\tilde{\mc{L}}G/_q \mc{L} G = (\tilde{\mc{L}}G \times \{q\})/\mc{L}G \longrightarrow (\mb{A}^{l + 1})^{an},
\end{equation}
where $l = \dim T$ is the rank of $G$. The morphism \eqref{eq:helmkeslodowyadjointquotient} can be viewed as a version of the adjoint quotient map for $\mc{L}G$. (Note that the domain of \eqref{eq:helmkeslodowyadjointquotient} is a $\mb{C}^\times$-bundle over $\bun_G^{an}(\mb{C}^\times/q^\mb{Z})$.)

Elements of elliptic Springer theory can also be found in the work of R. Friedman and J. Morgan \cite{friedman-morgan98} \cite{friedman-morgan00}. In \cite{friedman-morgan98}, Friedman and Morgan prove (for $E$ an elliptic curve over $\mb{C}$) that the coarse moduli space of semistable $G$-bundles is isomorphic to the quotient $Y \sslash W$, where $Y = \hom(\mb{X}^*(T), \mrm{Pic}^0(E)) \cong \mrm{Pic}^0(E)^l$ is the coarse moduli space of degree $0$ $T$-bundles on $E$. Friedman and Morgan's result can be viewed as a weak form of an elliptic Chevalley isomorphism; we give a more refined version, adapted for elliptic Springer theory, in Theorem \ref{thm:introbungchevalley} below. In \cite{friedman-morgan00}, Friedman and Morgan construct a map $\mb{A}^{l + 1} \to \bun_G$ sending $\mb{A}^{l + 1} \setminus \{0\}$ to the open substack $\bun_G^{ss}$ of semistable bundles, such that the composition $\mb{A}^{l + 1} \setminus \{0\} \to Y \sslash W$ with the coarse moduli space map factors through an isomorphism
\[ \mb{WP}^l = (\mb{A}^{l + 1} \setminus \{0\})\sslash \mb{G}_m \overset{\sim}\longrightarrow Y \sslash W\]
for some weighted action of $\mb{G}_m$ on $\mb{A}^{l + 1}$. This result is an elliptic version of the section theorems of B. Kostant \cite[Theorem 0.10]{kostant63} and R. Steinberg \cite[Theorem 1.4]{steinberg65}, which we also refine in this paper (Theorem \ref{thm:introfriedmanmorgansection}).

An elliptic version of the Grothendieck-Springer resolution for semistable $G$-bundles was studied by D. Ben-Zvi and D. Nadler \cite{ben-zvi-nadler15}, who considered the commutative diagram
\begin{equation} \label{eq:introssgrothendieckspringer}
\begin{tikzcd}
\tbun_G^{ss} \ar[r, "\psi^{ss}"] \ar[d, "\tilde{\chi}^{ss}"'] & \bun_G^{ss} \ar[d, "\chi^{ss}"] \\ Y \ar[r] & Y \sslash W,
\end{tikzcd}
\end{equation}
where $\chi^{ss}$ is Friedman and Morgan's coarse moduli space map and $\tbun_G^{ss} = \bun_B^0$ is the stack of degree $0$ $B$-bundles on $E$, and used it to prove some results about perverse sheaves on $\bun_G^{ss}$ of interest from the point of view of elliptic character sheaves. One can reasonably view \eqref{eq:introssgrothendieckspringer} as a semistable elliptic Grothendieck-Springer resolution, although (see Remark \ref{rmk:semistablegrothendieckspringer}) it is only a simultaneous resolution after deleting a closed set from the variety $Y \sslash W$ and its preimages in $Y$, $\bun_G^{ss}$ and $\tbun_G^{ss}$.

A small part of an elliptic Grothendieck-Springer resolution for unstable bundles was studied by I. Grojnowski and N. Shepherd-Barron \cite{grojnowski-shep19} for groups $G$ of type $D_5$, $E_6$, $E_7$ and $E_8$ only. For each of these $G$, Grojnowski and Shepherd-Barron construct a natural deformation $Z = \mb{A}^{l + 3} \to \bun_G$ of a subregular unstable bundle, and a commutative diagram
\begin{equation} \label{eq:grojnowskishepdiagram}
\begin{tikzcd}
\tilde{Z} \ar[r, "\psi_Z"] \ar[d, "\tilde{\chi}_Z"'] & Z \ar[d, "\chi_Z"] \\
\Theta_Y^{-1} \ar[r] & \hat{Y} \sslash W,
\end{tikzcd}
\end{equation}
which they prove to be a simultaneous log resolution (see \cite[Definition 1.1]{grojnowski-shep19} and Definition \ref{defn:simultaneouslogresolution} below). Here $\Theta_Y^{-1}$ is a certain anti-ample $W$-linearised line bundle on $Y$, $\hat{Y}$ is the affine cone over $Y$ obtained by contracting the zero section of $\Theta_Y^{-1}$. The resolution $\tilde{Z}$ is given by $\tilde{Z} = \tbun_G \times_{\bun_G} Z$, where $\tbun_G$ is the \emph{Kontsevich-Mori compactification} of $\tbun_G^{ss} = \bun_B^0$ over $\bun_G$, constructed using Kontsevich's theory of stable maps. The morphism $\chi_Z$ is a local version of Helmke and Slodowy's map \eqref{eq:helmkeslodowyadjointquotient}: a theorem of E. Looijenga \cite{looijenga76} gives an identification $\hat{Y}\sslash W \cong \mb{A}^{l + 1}$ of the bases of the two morphisms over $\mb{C}$. Grojnowski and Shepherd-Barron also give a detailed description of the morphism $\chi_Z$ and its resolution $\tilde{Z}$ in terms of the theory of del Pezzo surfaces \cite[Theorem 1.2]{grojnowski-shep19}.

The main result of this paper is that the diagrams \eqref{eq:introssgrothendieckspringer} and \eqref{eq:grojnowskishepdiagram} extend in a nice way to the whole of $\bun_G$, for all simply connected simple groups $G$. More precisely, we have the following.

\begin{thm}[Corollaries \ref{cor:fundiag} and \ref{cor:simultaneouslogresolution} and Proposition \ref{prop:chifibreunstable}] \label{thm:introgrothendieckspringer}
There exists an ample $W$-linearised line bundle $\Theta_Y$ on $Y$, with inverse $\Theta_Y^{-1}$, and a commutative diagram
\begin{equation} \label{eq:introbunggrothendieckspringer}
\begin{tikzcd}
\tbun_G \ar[r, "\psi"] \ar[d, "\tilde{\chi}"'] & \bun_G \ar[d, "\chi"] \\
\Theta_Y^{-1}/\mb{G}_m \ar[r] & (\hat{Y}\sslash W)/\mb{G}_m,
\end{tikzcd}
\end{equation}
which is a simultaneous log resolution with respect to the zero section of $\Theta_Y^{-1}/\mb{G}_m$ in the sense of Definition \ref{defn:simultaneouslogresolution} below, where $\hat{Y}$ is the affine cone over $Y$ obtained by contracting the zero section of $\Theta_Y^{-1}$ to a point. The preimage of the cone point under $\chi$ is precisely the locus of unstable bundles in $\bun_G$.
\end{thm}

We will call the diagram \eqref{eq:introbunggrothendieckspringer} the \emph{elliptic Grothendieck-Springer resolution}. Unlike the additive and multiplicative Grothendieck-Springer resolutions, the elliptic Grothendieck-Springer resolution is not quite a simultaneous resolution, as the morphism $\tilde{\chi}$ fails to be smooth over the zero section of $\Theta_Y^{-1}$. It does, however, satisfy the following weaker property.

\begin{defn} \label{defn:simultaneouslogresolution}
Let
\begin{equation} \label{eq:simultaneouslogresolution}
\begin{tikzcd}
\tilde{X} \ar[r, "\pi"] \ar[d, "\tilde{f}"'] & X \ar[d, "f"] \\
\tilde{S} \ar[r, "q"] & S
\end{tikzcd}
\end{equation}
be a commutative diagram of algebraic stacks and let $D \subseteq \tilde{S}$ be a divisor. We say that \eqref{eq:simultaneouslogresolution} is a \emph{simultaneous log resolution with respect to $D$} if the following conditions are satisfied.
\begin{enumerate}[(1)]
\item \label{itm:simultaneouslogres1} The morphisms $f$ and $\tilde{f}$ are flat, $q$ is representable, proper, surjective and generically finite, and $\pi$ is proper with finite diagonal.
\item \label{itm:simultaneouslogres2} For any point $s \colon \spec k \to \tilde{S}$, the morphism $\tilde{f}^{-1}(s) \to f^{-1}(q(s))$ is an isomorphism over a dense open substack of $f^{-1}(q(s))$.
\item \label{itm:simultaneouslogres3} The stack $\tilde{X}$ is regular, the morphism $\tilde{f}$ is smooth away from $D$ and $\tilde{f}^{-1}(D)$ is a (possibly non-reduced) divisor with normal crossings.
\end{enumerate}
\end{defn}

\begin{rmk}
Grojnowski and Shepherd-Barron also give a definition of simultaneous log resolution \cite[Definition 1.1]{grojnowski-shep19}, which is stronger than Definition \ref{defn:simultaneouslogresolution} in several respects. First, Grojnowski and Shepherd-Barron require that the map $\pi \colon \tilde{X} \to X$ be representable, whereas we impose only the weaker condition that it have finite diagonal. (Our condition is equivalent to requiring that the fibres of $\pi$ have only finite stabilisers, which in characteristic $0$ is equivalent to requiring that $\pi$ be relatively Deligne-Mumford.) Second, Grojnowski and Shepherd-Barron require that the singular fibres of $\tilde{f}$ be reduced with simple normal crossings, whereas we allow non-reduced irreducible components with self-intersections. Finally, Grojnowski and Shepherd-Barron require that the relative canonical bundle $K_{\tilde{X}/\tilde{S}}$ be the pullback of $K_{X/S}$, which we do not. We have chosen to make these modifications in order for Theorem \ref{thm:introgrothendieckspringer} to be true.
\end{rmk}

\begin{rmk} \label{rmk:semistablegrothendieckspringer}
Deleting the zero section from $\Theta_Y^{-1}$, Theorem \ref{thm:introgrothendieckspringer} implies that the semistable elliptic Grothendieck-Springer resolution \eqref{eq:introssgrothendieckspringer} is a simultaneous resolution as long as we replace $Y \sslash W \cong (\mb{A}^{l + 1} \setminus \{0\})\sslash\mb{G}_m$ with the complement $(\mb{A}^{l + 1} \setminus \{0\})/\mb{G}_m$ of the cone point in $(\hat{Y}\sslash W)/\mb{G}_m$. It follows that the original diagram \eqref{eq:introssgrothendieckspringer} is also a simultaneous resolution after deleting the points with stabilisers in the weighted projective space $Y \sslash W$.
\end{rmk}

\begin{rmk}
An important property of the additive and multiplicative Grothendieck-Springer resolutions is that the maps $\tilde{\mf{g}} \to \mf{g}$ and $\tilde{G} \to G$ (resp., $(\tilde{\chi}^{add})^{-1}(t) \to (\chi^{add})^{-1}(tW)$ and $(\tilde{\chi}^{mul})^{-1}(t) \to (\chi^{mul})^{-1}(tW)$) are small (resp., semi-small) in the sense of Goresky and MacPherson. While this is the case for the simultaneous resolution \eqref{eq:introssgrothendieckspringer} \cite[Theorem 2.2 (1)]{ben-zvi-nadler15}, it fails for the simultaneous log resolution \eqref{eq:introbunggrothendieckspringer}. The map $\tbun_G \to \bun_G$ is never small, since the simultaneous log resolution property implies that the codimension of the union of the images of the fibres of dimension $l$ is at most $l + 1$, and the maps $\tilde{\chi}^{-1}(y) \to \chi^{-1}(q(y))$ are never semi-small for $y$ in the zero section of $\Theta_Y^{-1}$, since the results of \cite{grojnowski-shep19} (in type $E$) and \cite{davis20} (all types) show that there is always at least one irreducible component of $\tilde{\chi}^{-1}(y)$ covered by positive-dimensional fibres.
\end{rmk}

The proof of Theorem \ref{thm:introgrothendieckspringer} is divided into two parts: the construction of the diagram \eqref{eq:introbunggrothendieckspringer} (Corollary \ref{cor:fundiag}) and the proof that it is a simultaneous log resolution (Corollary \ref{cor:simultaneouslogresolution}). The main ingredient in the construction is an elliptic version of Chevalley's isomorphisms $\mf{g}\sslash G \cong \mf{t}\sslash W$ and $G\sslash G \cong T \sslash W$, which refines Friedman and Morgan's identification of the coarse moduli space of $\bun_G^{ss}$ with $Y \sslash W$.

\begin{thm}[Theorem \ref{thm:bungchevalley}] \label{thm:introbungchevalley}
There is a certain subgroup $\mrm{Pic}^W(Y)_{good} \subseteq \mrm{Pic}^W(Y)$ of the group of $W$-linearised line bundles on $Y$ and an isomorphism
\begin{equation} \label{eq:introbungchevalley1}
 \mrm{Pic}^W(Y)_{good} \overset{\sim}\longrightarrow \mrm{Pic}(\bun_G).
\end{equation}
Moreover, if $L_Y \in \mrm{Pic}^W(Y)_{good}$ and $L_{\bun_G} \in \mrm{Pic}(\bun_G)$ is its image under \eqref{eq:introbungchevalley1}, then there is a canonical isomorphism
\begin{equation} \label{eq:introbungchevalley2}
H^0(Y, L_Y)^W \overset{\sim}\longrightarrow H^0(\bun_G, L_{\bun_G}).
\end{equation}
The isomorphisms \eqref{eq:introbungchevalley2} are compatible with tensor products of line bundles on $Y$ and $\bun_G$.
\end{thm}

We give the proof of Theorem \ref{thm:introbungchevalley} (as Theorem \ref{thm:bungchevalley}) in \S\ref{subsection:chevalley}.

The fact that \eqref{eq:introbunggrothendieckspringer} is a simultaneous log resolution is proved in \S\ref{subsection:friedmanmorganapplications} as a fairly straightforward consequence of the following analogue of the Kostant and Steinberg section theorems (\cite[Theorem 0.10]{kostant63} and \cite[Theorem 1.4]{steinberg65}).

\begin{thm}[Theorem \ref{thm:friedmanmorgansection}] \label{thm:introfriedmanmorgansection}
There exists a morphism $Z \to \bun_G$ from an affine space $Z$ such that the composition $Z \to (\hat{Y}\sslash W)/\mb{G}_m$ with $\chi$ factors through an isomorphism $Z \cong \hat{Y}\sslash W$. Moreover, writing
\[ \tilde{Z} = \tbun_G \times_{\bun_G} Z,\]
the morphism $\tilde{Z} \to \Theta_Y^{-1}/\mb{G}_m$ induced by $\tilde{\chi}$ also factors through an isomorphism $\tilde{Z} \cong \Theta_Y^{-1}$.
\end{thm}

Theorem \ref{thm:introfriedmanmorgansection} is a mild refinement of Friedman and Morgan's other result discussed above \cite[Theorem 5.1.1]{friedman-morgan00} , so we call it the Friedman-Morgan section theorem. The new observations here are that Friedman and Morgan's parabolic induction construction for the map $Z \to \bun_G$ can be made to give a natural lift $Z \to \hat{Y}\sslash W$ (Proposition \ref{prop:inducedequivariantslice}), and that the isomorphism $Z \cong \hat{Y}\sslash W$ lifts to (and can actually be deduced from) an isomorphism $\tilde{Z}\cong \Theta_Y^{-1}$.

\begin{rmk} \label{rmk:rigidification}
For technical reasons, we will often need to work with the rigidified stack $\bun_{G, rig}$ obtained from $\bun_G$ by taking the quotient of all automorphism groups by the centre of $G$. (See \S\ref{subsection:rigidification} for the precise definition.) The advantages of $\bun_{G, rig}$ over $\bun_G$ are that various automorphism groups (coming from centres of Levi subgroups) that are disconnected in $\bun_G$ become connected in $\bun_{G, rig}$, and that it is easier in practice to construct morphisms $Z \to \bun_{G, rig}$. For example, the Friedman-Morgan map $\hat{Y}\sslash W \to \bun_G$ does not factor through a section $(\hat{Y}\sslash W)/\mb{G}_m \to \bun_G$ of the coarse quotient map, but it \emph{does} factor through a section $(\hat{Y}\sslash W)/\mb{G}_m \to \bun_{G, rig}$.
\end{rmk}

\begin{rmk}
The stacks $\mf{g}/G$, $G/G$ and $\bun_G^{ss}$ can all be put on an equal footing by observing that $\mf{g}/G \cong \bun_G^{ss}(C_{cusp})$ and $G/G \cong \bun_G^{ss}(C_{node})$, where $C_{cusp}$ and $C_{node}$ are singular degenerations of an elliptic curve with a single cusp and node respectively. (Here we na\"ively define a principal bundle on a singular curve to be semistable if the pullback to the normalisation is semistable.) The stacks $\tilde{\mf{g}}/G \cong \mf{b}/B$ and $\tilde{G}/G \cong B/B$ are naturally identified with $\bun_B^0(C_{cusp})$ and $\bun_B^0(C_{node})$, and one can therefore consider Kontsevich-Mori compactifications $\tbun_G(C) \to \bun_G(C)$ for $C = C_{cusp}$ and $C_{node}$. While we believe many of the results of this paper can probably be extended in some way to this context, a number of subtleties arise in the singular case that we do not know how to resolve at present. This added difficulty is already apparent in the fact that the morphisms $\tilde{\mf{g}}/G \to \mf{t}$ and $\tilde{G}/G \to T$ do not extend to $\tbun_{G}(C_{cusp})$ and $\tbun_G(C_{node})$, so some additional compactification of the space of $T$-bundles is needed.
\end{rmk}

\begin{rmk}
Throughout the body of this paper, we will work in a somewhat more general context than in this introduction. Instead of working with a single elliptic curve $E$ defined over an algebraically closed field $k$, we will allow arbitrary families $E \to S$ of smooth curves of genus $1$ over a regular stack $S$ (and work with a split simply connected simple group scheme $G$ over $\spec \mb{Z}$), subject only sometimes to the restriction that $E \to S$ have a section. The key examples that should be kept in mind are:
\begin{enumerate}[(1)]
\item \label{itm:curveoverfield} $S = \spec k$ for $k$ a field, and $E$ an elliptic curve over $k$,
\item $S = \B E'$ and $E = \spec k$, where $\B E'$ is the classifying stack of an elliptic curve $E'$ over $k$ (this amounts to working with $G$-bundles on $E'$ up to translation), and
\item $S = M_{1, 1}$ the stack of elliptic curves over $\spec \mb{Z}$ (or over some field) and $E \to S$ the universal elliptic curve.
\end{enumerate}
It should be emphasised that very little will be lost to the reader who wishes to assume that we are in case \eqref{itm:curveoverfield} throughout.
\end{rmk}

\begin{rmk}
The results presented here can all be found in some form in the author's PhD thesis \cite{davis19a}. We will also refer to that work for the proofs of several peripheral or well-known propositions.
\end{rmk}

\subsection{Plan of the paper}

The paper is divided into this introduction (Section \ref{section:introduction}) and three other sections.

Section \ref{section:preparations} is a collection of miscellaneous preparatory material, including the definition of the Kontsevich-Mori compactification, a review of the rigidification construction, structural results for various stacks of principal bundles, and a small result on ramified Galois descent.

Section \ref{section:grothendieckspringer} is concerned with the construction of the fundamental diagram \eqref{eq:introbunggrothendieckspringer}. The main results in this section are the elliptic Chevalley isomorphism (Theorem \ref{thm:bungchevalley}), a computation of the group $\mrm{Pic}^W(Y)_{good} \cong \mrm{Pic}(\bun_G)$ (Proposition \ref{prop:picwygoodcomputation} and its Corollary \ref{cor:picbung}), and the existence of the fundamental diagram (Corollary \ref{cor:fundiag}). In constructing the morphism $\tilde{\chi}$, we also give an explicit formula (Corollary \ref{cor:chitildedivisor}) for the multiplicities in the divisor $\tilde{\chi}^{-1}(0_{\Theta_Y^{-1}})$, which may be of independent interest.

Finally, Section \ref{section:friedmanmorgan} is concerned with proving that \eqref{eq:introbunggrothendieckspringer} is indeed a simultaneous log resolution (Corollary \ref{cor:simultaneouslogresolution}). The bulk of this section is taken up by the setup and proof of the Friedman-Morgan section theorem (Theorem \ref{thm:friedmanmorgansection}), from which we deduce the simultaneous log resolution property. As part of the setup, we introduce the notion of an \emph{equivariant slice} of $\bun_{G, rig}$ (Definition \ref{defn:equivariantslice}), which we believe to be the correct analogue of a transversal slice in elliptic Springer theory.

\subsection{Notation and conventions}

Unless otherwise specified, all schemes will be locally Noetherian, and all group schemes will be flat, affine and of finite presentation.

Unless otherwise specified, by a \emph{reductive group} we will mean a split connected reductive group scheme over $\spec \mb{Z}$.

Throughout the paper, we will fix a connected regular stack $S$, a smooth curve $E \to S$ of genus $1$, and a simply connected simple reductive group $G$ (over $\spec \mb{Z}$) with maximal torus and Borel subgroup $T \subseteq B \subseteq G$. 
We will write $(\mb{X}^*(T), \Phi, \mb{X}_*(T), \Phi^\vee)$ for its root datum, where
\[ \mb{X}^*(T) = \hom(T, \mb{G}_m) \quad \text{and} \quad \mb{X}_*(T) = \hom(\mb{G}_m, T)\]
are the groups of characters and cocharacters of the split torus $T$. The set of roots $\Phi$ is by definition the set of weights of $T$ acting on the Lie algebra $\mf{g} = \mrm{Lie}(G)$; we will adopt the convention that the set $\Phi_- \subseteq \Phi$ of negative roots is the set of nonzero weights of $T$ acting on $\mrm{Lie}(B)$, and let $\Phi_+ = -\Phi_-$ be the corresponding set of positive roots. We will write $\Delta = \{\alpha_1, \ldots, \alpha_l\} \subseteq \Phi_+$ and $\Delta^\vee = \{\alpha_1^\vee, \ldots, \alpha_l^\vee\} \subseteq \Phi_+^\vee$ for the sets of positive simple roots and coroots respectively, and $\{\varpi_1, \ldots, \varpi_l\}$ and $\{\varpi_1^\vee, \ldots, \varpi_l^\vee\}$ for the bases of $(\mb{Z}\Phi^\vee)^\vee$ and $(\mb{Z}\Phi)^\vee$ dual to $\Delta$ and $\Delta^\vee$ respectively. Note that $\mb{Z}\Phi^\vee = \mb{X}_*(T)$ since $G$ is simply connected, so $\{\alpha_1^\vee, \ldots, \alpha_l^\vee\}$ is a basis for $\mb{X}_*(T)$ and $\{\varpi_1, \ldots, \varpi_l\}$ is a basis for $\mb{X}^*(T)$.

If $P \subseteq G$ is a parabolic subgroup, we will say that $P$ is \emph{standard} if $B \subseteq P$. Every parabolic subgroup is conjugate to a unique standard one. If $P$ is standard, the \emph{type of $P$} is the set
\[ t(P) = \{\alpha_i \in \Delta \mid \alpha_i \;\text{is not a root of}\; P\} \subseteq \Delta.\]
The construction $P \mapsto t(P)$ defines a bijection between (proper) parabolic subgroups of $G$ and (nonempty) subsets of $\Delta$.

If $H \to S'$ is a group scheme on a stack $S'$, a \emph{principal $H$-bundle} or \emph{$H$-torsor} on an $S'$-scheme $U$ will be a morphism $\xi \to U$ equipped with a right $H$-action on $\xi$, such that $\xi$ is locally $H$-isomorphic to $U \times_{S'} H$ in the \'etale topology on $U$. We will write $\B H$ or $\B_{S'} H$ for the classifying stack of $H$ over $S'$, whose functor of points sends an $S'$-scheme $U$ to the groupoid of $H$-torsors on $U$. If $H$ is commutative, then $\B H$ is itself a (commutative) group stack over $S'$, and a torsor under this group stack will be called an \emph{$H$-gerbe}.

If $S'$ and $S''$ are stacks, $X \to S'$ is a proper curve over $S'$, $H \to S''$ is a group scheme over $S''$ and $X \to S''$ is any morphism, we will write $\bun_{H/S'}(X)$ for the stack over $S'$ whose functor of points sends an $S'$-scheme $U$ to the groupoid of $H$-torsors on $U \times_{S'} X$. (Note that $U \times_{S'} X$ is an $S''$-scheme via its morphism to $X$, so this makes sense.) So, for example, with $S'' = X$ (resp., $S'' = \spec \mb{Z}$), we have a notion of $\bun_{H/S'}(X)$ when $H$ is a group scheme over $X$ (resp., $\spec \mb{Z}$). When $X$ is smooth over $S'$ and $H$ is reductive, we write $\bun_{H/S'}^{ss}(X) \subseteq \bun_{H/S'}(X)$ for the open substack of semistable bundles. In the special case when $S' = S$ and $X = E$, we will write $\bun_H = \bun_{H/S}(E)$.

If $X$ is any stack equipped with an injective action of the classifying stack $\B Z(G)$ of the centre of $G$, then we write $X_{rig}$ for the rigidification of $X$ with respect to $Z(G)$ obtained by taking the quotient of all automorphism groups in $X$ by $Z(G)$ (see \S \ref{subsection:rigidification} for what this means). For example, if $H$ is any group scheme with $Z(G) \subseteq Z(H)$, then $\B Z(G)$ acts injectively on $\bun_H$, so we have a rigidification $\bun_{H, rig}$.

If $H$ is any reductive group, then the abelianisation $H/[H, H]$ and the reduced identity component $Z(H)^\circ$ of the centre are split tori. If $X \to \spec k$ is a proper curve over a field $k$ and $\xi_H \to X$ is a principal $H$-bundle, the \emph{degree} (resp., \emph{slope}) of $\xi_H$ is the unique vector $\deg(\xi_H) \in \mb{X}_*(H/[H, H])$ (resp., $\mu(\xi_H) \in \mb{X}_*(Z(H)^\circ)_\mb{Q} = \mb{X}_*(Z(H)^\circ) \otimes \mb{Q}$) such that
\[ \langle \lambda, \deg(\xi_H) \rangle = \deg(\xi_H \times^H \mb{Z}_\lambda) \quad \text{( resp.,}\;\; \langle \lambda, \mu(\xi_H) \rangle = \deg (\xi_H \times^H \mb{Z}_\lambda) \; \text{)},\]
for all $\lambda \in \mb{X}^*(H/[H, H]) \subseteq \mb{X}^*(Z(H)^\circ)$, where $\mb{Z}_\lambda$ is the $1$-dimensional representation with weight $\lambda$ and $\langle\,,\,\rangle$ is the canonical pairing between characters and cocharacters. If $X \to S'$ is a proper curve over a general base $S'$ and $d \in \mb{X}_*(H/[H, H])$ (resp., $\mu \in \mb{X}_*(Z(H)^\circ)_\mb{Q}$), we write $\bun_{H/S'}^{d}(X)$ (resp., $\bun_{H/S'}^\mu(X)$) for the open and closed substack of $\bun_{H/S'}(X)$ consisting of principal bundles with degree $d$ (resp., slope $\mu$) on every fibre of $X \to S'$.

More generally, if $H$ is a group scheme over $\spec \mb{Z}$ with unipotent radical $R_u(H)$ such that $H/R_u(H)$ is split reductive, the degree and slope of an $H$-bundle $\xi_H$ on a curve $X$ are by definition the degree and slope of the induced $H/R_u(H)$-bundle $\xi_H/R_u(H)$. We will use the same notation $\bun_{H/S'}^d(X)$ and $\bun_{H/S'}^\mu(X)$ for $H$-bundles with fixed degree and slope as in the case when $H$ is reductive.

Finally, if $X \to S$ is a morphism of Artin stacks, we will write $\mb{L}_{X/S}$ for the relative cotangent complex \cite[\S 8]{olsson07} and $\mb{T}_{X/S} = (\mb{L}_{X/S})^\vee$ for the relative tangent complex.

\subsection{Acknowledgements}

The author would especially like to thank Nicholas Shepherd-Barron for introducing him to this topic, for many detailed discussions, and for his very useful advice, insights and questions. He would also like to thank Ian Grojnowski, Masoud Kamgarpour, Arun Ram, Travis Schedler and Richard Thomas for helpful conversations relating to the contents of this paper, and the anonymous referee for their comments.

The author was supported by the EPSRC grants [EP/L015234/1] (The EPSRC Centre for Doctoral Training in Geometry and Number Theory (The London School of Geometry and Number Theory), University College London), and [EP/R034826/1]. The majority of this work was completed while the author was a PhD student at King's College London. The author would also like to acknowledge the generous support of the University of Melbourne, at which he was a visitor during the early stages of this research.


\section{Preparations} \label{section:preparations}

In this section, we review some of the basic objects appearing in the elliptic Grothendieck-Springer resolution \eqref{eq:introbunggrothendieckspringer} and some of the tools going into the proof of Theorem \ref{thm:introgrothendieckspringer}.

In \S\ref{subsection:kontsevichmori}, we give the definition of the Kontsevich-Mori compactification $\tbun_G$, and in \S\ref{subsection:rigidification} we review the process of rigidification mentioned in Remark \ref{rmk:rigidification}. In \S\ref{subsection:unstable}, we recall the definitions of unstable and semistable principal bundles and the decomposition of the locus of unstable $G$-bundles into Harder-Narasimhan loci. In \S\ref{subsection:bruhat}, we use Bruhat cells in $G/B$ to define some useful locally closed substacks inside the stack of $G$-bundles equipped with reductions to both $B$ and a parabolic subgroup. Finally, in \S\ref{subsection:descent}, we recall a basic result for descending line bundles along ramified Galois coverings.

\subsection{The Kontsevich-Mori compactification} \label{subsection:kontsevichmori}

In this subsection, we review the definition and basic properties of the Kontsevich-Mori compactification $\tbun_G$ of $\bun_B^0$.

To motivate the construction, first recall that functor of points of the $S$-stack $\bun_B$ is isomorphic to the functor $\mrm{Sch}_{/S}^{op} \to \mrm{Grpd}$ sending an $S$-scheme $U$ to the groupoid of pairs $(\xi_G, \sigma)$, where $\xi_G \to U \times_S E$ is a principal $G$-bundle and $\sigma \colon U \times_S E \to \xi_G/B$ is a section of the associated bundle of flag varieties. The isomorphism sends a $B$-bundle $\xi_B \to U \times_S E$ to the pair $(\xi_G, \sigma)$, where $\xi_G = \xi_B \times^B G$ is the induced $G$-bundle and $\sigma$ is the section
\[ \sigma \colon E = \xi_B \times^B B/B \longrightarrow \xi_B \times^B G/B = \xi_G/B \]
induced by the inclusion $B/B \hookrightarrow G/B$. Given a pair $(\xi_G, \sigma)$, the degree of the associated $B$-bundle is identified with the unique vector $[\sigma] \in \mb{X}_*(T)$ such that
\[ \langle \lambda, [\sigma] \rangle = \deg \sigma^* \mc{L}_\lambda^{\xi_G} \]
for all $\lambda \in \mb{X}^*(T)$, where
\[ \mc{L}_\lambda^{\xi_G} = \xi_G \times^B \mb{Z}_\lambda,\]
is the natural line bundle on $\xi_G/B$ associated to $\lambda$. The stack $\bun_B^0$ is therefore isomorphic to the stack of pairs $(\xi_G, \sigma)$ with $[\sigma] = 0$.

The Kontsevich-Mori compactification is defined by allowing the section $\sigma$ in the above description of $\bun_B^0$ to degenerate to a map from a singular curve.

\begin{defn}
Let $U$ be a scheme and let $X \to U$ be a proper morphism. A \emph{prestable map to $X$ over $U$ of genus $g$} is a pair $(C, \sigma)$, where $C \to U$ is a proper flat family of curves whose geometric fibres have arithmetic genus $g$ and at worst nodal singularities, and $\sigma \colon C \to X$ is a morphism over $U$. We say that a prestable map $(C, \sigma)$ is \emph{stable} if, for every geometric point $u \colon \spec k \to U$, the automorphism group of $C_u$ over $X_u$ is finite.
\end{defn}

\begin{defn} \label{defn:kontsevichmori}
The \emph{Kontsevich-Mori compactification} $\tbun_G$ of $\bun_B^0$ is the $S$-stack whose functor of points sends an $S$-scheme $U$ to the groupoid of tuples $(\xi_G, C, \sigma)$ where $\xi_G \to U \times_S E$ is a $G$-bundle and $(C, \sigma)$ is a stable map to $\xi_G/B$ over $U$ of genus $1$ such that
\begin{enumerate}[(1)]
\item $C \to U \times_S E$ has degree $1$ on every fibre over $U$, and
\item $[\sigma_u] = 0$ for every geometric point $u$ of $U$, where $[\sigma_u] \in \mb{X}_*(T)$ is the unique cocharacter satisfying
\[ \langle \lambda, [\sigma_u] \rangle = \deg \sigma_u^* \mc{L}_\lambda^{\xi_G} \]
for all $\lambda \in \mb{X}^*(T)$.
\end{enumerate}
\end{defn}

\begin{rmk}
Let $s \colon \spec k \to S$ be a geometric point and $(\xi_G, C, \sigma)$ a $k$-point of $\tbun_G$ lying over $s$. Then $C = E_s \cup \bigcup_i C_i$ has a unique irreducible component mapping to a section of $\xi_G/B \to E_s$, and a number of rational components $C_i \cong \mb{P}^1_k$ mapping into fibres of the $G/B$-bundle $\xi_G/B \to E_s$.
\end{rmk}

The following proposition is a straightforward consequence of the general theory of stable maps (e.g., \cite[\S 2]{abramovich-oort98}).

\begin{prop} \label{prop:kmproper}
The stack $\tbun_G$ is an Artin stack, and the morphism $\tbun_G \to \bun_G$ given by forgetting the stable map is proper with finite relative stabilisers.
\end{prop}

\begin{prop} \label{prop:kmsurjective}
The morphism $\tbun_G \to \bun_G$ is surjective.
\end{prop}
\begin{proof}
For simplicity, we can assume without loss of generality that $S = \spec k$ for $k$ an algebraically closed field.

For a generic $T$-bundle of degree $0$, we have that $\xi_T \times^T \mf{g}/\mf{b}$ is a direct sum of nontrivial line bundles of degree $0$, and hence $H^1(E, \xi_T \times^T \mf{g}/\mf{b}) = 0$, where $\mf{g} = \mrm{Lie}(G)$ and $\mf{b} = \mrm{Lie}(B)$. So the morphism $\bun_B^0 \to \bun_G$ is smooth at the point $\xi_B = \xi_T \times^T B$ for such a $T$-bundle.

Hence, there is a nonempty open subset $U \subseteq \bun_{B}^0$ such that the morphism $U \to \bun_G$ is smooth. Since smooth morphisms are open, and $\bun_B^0 \to \bun_G$ factors through $\tbun_G \to \bun_G$, it follows that the image of $\tbun_G \to \bun_G$ contains a nonempty open substack. Since $\bun_G$ is regular and connected, hence irreducible, the image of $\tbun_G \to\bun_G$ is therefore dense. So surjectivity now follows from properness.
\end{proof}

In studying the Kontsevich-Mori compactification, it is often useful to study the domain curves for the stable maps in isolation.

\begin{defn}
Let $U$ be a scheme over $S$. A \emph{prestable degeneration of $E$ over $U$} is a prestable map $f \colon C \to U \times_S E$ over $U$ such that for every geometric point $u \colon \spec k \to U$ over $s \colon \spec k \to S$, the fibre $C_u$ has arithmetic genus $1$ and the map $f_u \colon C_u \to E_s$ has degree $1$. We write $\Deg_S(E)$ for the (Artin) stack whose functor of points sends an $S$-scheme $U$ to the groupoid of prestable degenerations of $E$ over $U$.
\end{defn}

In the following proposition, we write $f \colon \mc{C} \to \Deg_S(E) \times_S E$ for the universal prestable degeneration. We also write $\Deg_S(E)^{\leq 1} \subseteq \Deg_S(E)$ for the open substack of curves with at most one node, and $\mc{C}^{\leq 1} = f^{-1}(\Deg_S(E)^{\leq 1})$.

\begin{prop} \label{prop:basicdegproperties}
The stacks $\mc{C}$ and $\Deg_S(E)$ are smooth over $S$. The section $S \to \Deg_S(E)$ classifying the identity $\id_E$ is an open immersion and its complement $D$ is a divisor with normal crossings relative to $S$. The smooth locus $D^\circ \subseteq D$ classifies prestable maps $C \to E$ with exactly one node. The stack $\mc{C}^{\leq 1}$ is the blowup of $\Deg_S(E)^{\leq 1} \times_S E$ along a closed substack mapping isomorphically to $D^\circ$ under the projection to the first factor.
\end{prop}
\begin{proof}
Smoothness of $\Deg_S(X)$ is proved in \cite[Proposition 2.4.1]{campbell16}. The remaining assertions are all clear from the proof of \cite[Proposition 4.4.1]{campbell16}.
\end{proof}

\begin{prop} \label{prop:kmdegsmooth}
The morphism
\begin{equation} \label{eq:kmdegmap}
\tbun_G \longrightarrow \Deg_S(E)
\end{equation}
is smooth.
\end{prop}
\begin{proof}
This is proved in \cite[Proposition 2.4.1]{campbell16}.
\end{proof}

\begin{cor}
The stack $\tbun_G$ is smooth over $S$, and contains $\bun_B^0$ as a dense open substack.
\end{cor}
\begin{proof}
Smoothness follows from Proposition \ref{prop:basicdegproperties} and Proposition \ref{prop:kmdegsmooth}. The discussion at the beginning of this subsection identifies $\bun_B^0$ with the preimage of $\{\id_X\}$ under \eqref{eq:kmdegmap}, so openness and density follow from Proposition \ref{prop:basicdegproperties}.
\end{proof}

If $C$ is a curve over a field $k$ and $\xi_G \to \spec k$ is a $G$-torsor, then the \emph{degree} of a map $f \colon C \to \xi_G/B$ is the unique element $[f] \in \mb{X}_*(T)$ satisfying $\langle \lambda, [f] \rangle = \deg f^*\mc{L}^{\xi_G}_\lambda$ for all $\lambda \in \mb{X}^*(T)$. Since the line bundle $\mc{L}_{\varpi_i}^{\xi_G}$ on $\xi_G/B$ is nef for all fundamental dominant weights $\varpi_i$, it follows that
\[ [f] \in \mb{X}_*(T)_{\geq 0} = \{ \mu \in \mb{X}_*(T) \mid \langle \varpi_i, \mu \rangle \geq 0 \; \text{for}\; \alpha_i \in \Delta\} \subseteq \mb{X}_*(T).\]

For $\lambda \in \mb{X}_*(T)_+ = \mb{X}_*(T)_{\geq 0} \setminus \{0\}$, we write $D_{\lambda}^{\circ} \subseteq \tbun_G$ for the locus of stable maps $C \to \xi_G/B$ with a single rational component mapping into a fibre of $\xi_G/B \to E$ with degree $\lambda$, and $D_{\lambda}$ for its closure. We also write $M_{0, 1}^+(G/B, \lambda)$ for the stack of $1$-pointed stable maps to $G/B$ of degree $\lambda$ with smooth domain curve, sending the marked point to the natural base point $B/B \subseteq G/B$. Note that the action of $B$ on $G/B$ induces an action on $M_{0, 1}^+(G/B, \lambda)$.

\begin{prop} \label{prop:kmdivisor}
We have the following.
\begin{enumerate}[(1)]
\item \label{itm:kmdivisor1} The complement $D_B$ of $\bun_B^0$ in $\tbun_G$ is a divisor with normal crossings relative to $S$, and decomposes as
\[ D_B = \sum_{\lambda \in \mb{X}_*(T)_+} D_{\lambda}.\]
\item \label{itm:kmdivisor2} The smooth locus of $D_B$ is
\[ D_B^\circ = \sum_{\lambda \in \mb{X}^*(T)_+} D_{\lambda}^{\circ}.\]
\item \label{itm:kmdivisor3} For each $\lambda \in \mb{X}^*(T)_+$, there is an isomorphism
\[ D_{\lambda}^\circ \cong \xi_{B, \mu - \lambda}^{uni} \times^B M_{0, 1}^+(G/B, \mu),\]
where $\xi_{B, \mu - \lambda}^{uni} \to \bun_{B}^{\mu - \lambda} \times_S E$ is the universal $B$-bundle of degree $\mu - \lambda$.
\end{enumerate}
\end{prop}
\begin{proof}
The assertions \eqref{itm:kmdivisor1} and \eqref{itm:kmdivisor2} follow immediately from Proposition \ref{prop:basicdegproperties} and Proposition \ref{prop:kmdegsmooth}. To prove \eqref{itm:kmdivisor3}, note that the right hand side can be identified with the stack of tuples $(\xi_G, \sigma, x, f \colon C \to \xi_{G, x}/B, p)$, where $\xi_G$ is a $G$-bundle on (a fibre of) $E$, $\sigma \colon E \to \xi_G/B$ is a section of degree $\mu - \lambda$, $x \in E$ is a point and $(f \colon C \to \xi_{G, x}/B, p \in C)$ is a $1$-pointed stable map from a smooth rational curve $C$ to $\xi_{G, x}/B$ of degree $\lambda$ sending the marked point to $\sigma(x)$. The isomorphism sends $(\xi_G, \sigma, x, f, p)$ to the stable map $E \cup_{x, p} C \to \xi_G/B$ induced by $\sigma$ and $f$.
\end{proof}

The next proposition shows that the inclusion $\bun_B^0 \subseteq \tbun_G$ is closely related to the inclusion $\bun_G^{ss} \subseteq \bun_G$ of the open substack of semistable bundles. (We direct the reader to Definition \ref{defn:semistable} for the definition of semistable and unstable $G$-bundles.)

\begin{prop}
Let $\tbun_G^{ss} = \psi^{-1}(\bun_G^{ss}) \subseteq \tbun_G$ be the preimage of the locus of semistable bundles. Then $\tbun_G^{ss} = \bun_B^0$ as open substacks of $\tbun_G$.
\end{prop}
\begin{proof}
The claim reduces easily to the following: given a geometric point $s \colon \spec k \to S$, a $G$-bundle $\xi_G \to E_s$ and a stable map $\sigma \colon C \to \xi_G/B$ with $[\sigma] = 0$, the $G$-bundle $\xi_G$ is semistable if and only if $C \to E_s$ is an isomorphism.

Assume first that $C \to E_s$ is an isomorphism. Then $\sigma$ is a section $\sigma \colon E_s \to \xi_G/B$ of degree $0$, so we can write $\xi_G = \xi_B \times^B G$ for some degree $0$ $B$-bundle $\xi_B \to E_s$ with associated $T$-bundle $\xi_T = \xi_B \times^B T$. So for every representation $V$ of $G$, the vector bundle $\xi_B \times^B V = \xi_G \times^G V$ on $E_s$ has a filtration whose subquotients are line bundles of degree $0$. So the vector bundle $\xi_G \times^G V$ is (slope) semistable, i.e., it has no subbundles of positive degree. If $\xi_G$ were unstable, then there would exist a reduction $\xi_P$ of $\xi_G$ to some standard parabolic $P$ and a dominant weight $\lambda$ of $P$ such that $\langle \lambda, \deg \xi_P \rangle < 0$. Taking $V$ to be a representation of highest weight $\lambda$, the kernel of the projection $V \to V_\lambda$ to the $\lambda$-weight space is $P$-invariant (recall that $P$ contains the weight-\emph{lowering} Borel of negative roots according to our conventions), so defines a subbundle of $\xi_G \times^G V$ of degree $-\langle \lambda, \deg \xi_P \rangle > 0$. This is a contradiction, so $\xi_G$ must be semistable.

Conversely, assume that $C \to E_s$ is not an isomorphism. Then there exists a unique irreducible component of $C$ mapping isomorphically to $E_s$, and the restriction of $\sigma$ to this irreducible component defines a section of degree
\[ [\sigma|_{E_s}] < [\sigma] = 0.\]
If $\xi_B$ denotes the $B$-bundle corresponding to $\sigma|_{E_s} \colon E_s \to \xi_G/B$, then it follows that there exists a dominant character $\lambda$ of $B$ such that
\[ \deg \xi_B \times^B \mb{Z}_\lambda = \deg (\sigma|_{E_s})^*\mc{L}^{\xi_G}_{\lambda} = \langle \lambda, [\sigma|_{E_s}] \rangle < 0,\]
so $\xi_G$ is unstable.
\end{proof}

An important part of Theorem \ref{thm:introgrothendieckspringer} is the existence of a morphism $\tilde{\chi} \colon \tbun_G \to \Theta_Y^{-1}/\mb{G}_m$, where $\Theta_Y^{-1}$ is some line bundle over the $S$-relative coarse moduli space $Y = \hom(\mb{X}^*(T), \mrm{Pic}^0_S(E))$ of degree $0$ $T$-bundles on $E$. For now, we will construct a morphism $\mrm{Bl}_B \colon \tbun_G \to \bun_T^0$. The morphism $\tilde{\chi}$ will eventually be constructed in \S\ref{subsection:fundiag} as a lift of the composition $\mrm{Bl}_{B, Y} \colon \tbun_G \to Y$ of $\mrm{Bl}_B$ with the coarse moduli space map $\bun_T^0 \to Y$.

Let $\xi_{T}^{uni} \to \bun_{T/\Deg_S(E)}(\mc{C}) \times_{\Deg_S(E)} \mc{C}$ be the universal $T$-bundle, where $f \colon \mc{C} \to \Deg_S(E) \times_S E$ is the universal prestable degeneration of $E$. Note that since the locus of singular curves in $\Deg_S(E)$ is a divisor, and since each prestable degeneration is an isomorphism over a dense open set of a fibre of $E$, it follows that the complement of the open substack $U \subseteq \Deg_S(E) \times_S E$ of points over which $f$ is an isomorphism has codimension $2$. Since $\bun_{T/\Deg_S(E)}(\mc{C}) \times_S E$ is smooth over $S$, it is regular, so the restriction of $\xi_{T}^{uni}$ to
\begin{align*}
\bun_{T/\Deg_S(E)}(\mc{C}) \times_{\Deg_S(E)} f^{-1}(U) &\cong \bun_{T/\Deg_S(E)}(\mc{C}) \times_{\Deg_S(E)} U \\
&\subseteq \bun_{T/\Deg_S(E)}(\mc{C}) \times_S E
\end{align*}
extends uniquely to a $T$-bundle $\mrm{Bl}(\xi_{T}^{uni})$ on $\bun_{T/\Deg_S(E)}(\mc{C}) \times_S E$ since $T$ is a torus. The $T$-bundle $\mrm{Bl}(\xi_{T}^{uni})$ determines a morphism
\[
\mrm{Bl}_{T} \colon \bun_{T/\Deg_S(E)}(\mc{C}) \longrightarrow \bun_{T}.
\]

\begin{defn}
In the setup above, we call the morphism $\mrm{Bl}_{T}$ the \emph{blow down morphism for $T$}.
\end{defn}

The blow down of a $T$-bundle can also be described in terms of its associated line bundles.

\begin{lem} \label{lem:tbundleblowdowndet}
Let $\lambda \in \mb{X}^*(T)$ be a character. Then
\begin{equation} \label{eq:tbundleblowdowndet}
\lambda(\mrm{Bl}(\xi_{T}^{uni})) = \det \mb{R}f_* \lambda(\xi_{T}^{uni}),
\end{equation}
where $\det$ denotes the determinant of a perfect complex, and by abuse of notation we write
\[ f \colon \bun_{T/\Deg_S(E)}(\mc{C}) \times_{\Deg_S(E)} \mc{C} \to \bun_{T/\Deg_S(E)}(\mc{C}) \times_S E \]
for the pullback of the morphism $f \colon \mc{C} \to \Deg_S(E) \times_S E$.
\end{lem}
\begin{proof}
Since $\bun_{T/\Deg_S(E)}(\mc{C}) \times_S E$ is regular, this follows from the fact that both sides of \eqref{eq:tbundleblowdowndet} agree when restricted to $\bun_{T/\Deg_S(E)}(\mc{C}) \times_{\Deg_S(E)} U$.
\end{proof}

Returning to the Kontsevich-Mori compactification, note that given a prestable degeneration $g \colon C \to E_s$ and a $G$-bundle $\xi_G \to E_s$, the datum of a map $\sigma \colon C \to \xi_G/B$ lifting $g$ is equivalent to the datum of a reduction of the $G$-bundle $g^*\xi_G$ to $B$. So we can identify $\tbun_G$ with an open substack of
\[ \bun_{B/\Deg_S(E)}(\mc{C}) \times_{\bun_{G/\Deg_S(E)}(\mc{C})}(\bun_G \times_S \Deg_S(E)) \]
defined by degree and stability conditions. In particular, we have a morphism
\[ \tbun_G \longrightarrow \bun_{B/\Deg_S(E)}(\mc{C}) \longrightarrow \bun_{T/\Deg_S(E)}(\mc{C}),\]
where the second morphism is induced by the canonical retraction $B \to T$.

\begin{defn}
We define the \emph{blow down morphism for $\tbun_G$} to be the composition
\[ \mrm{Bl}_B \colon \tbun_G \longrightarrow \bun_{T/\Deg_S(E)}(\mc{C}) \xrightarrow{\mrm{Bl}_{T}} \bun_{T}.\]
\end{defn}

\begin{prop} \label{prop:kmblowdown}
The blow down morphism $\mrm{Bl}_B$ is smooth, extends the canonical morphism $\bun_B^0 \to \bun_T$, and factors through the connected component $\bun_T^0 \subseteq \bun_T$.
\end{prop}
\begin{proof}
This is a special case of \cite[Corollary 3.5.4]{davis19a}.
\end{proof}

\begin{prop} \label{prop:blowdowngluing}
Let $\lambda \in \mb{X}_*(T)_+$. Then the composition
\[ D_{\lambda}^{\circ} \longhookrightarrow \tbun_G \xrightarrow{\mrm{Bl}_B} \bun_{T}^0 \]
is given in terms of the isomorphism of Proposition \ref{prop:kmdivisor} \eqref{itm:kmdivisor3} by
\begin{align*}
\xi_{B, - \lambda}^{uni} \times^B M_{0, 1}^+(G/B, \lambda) \longrightarrow \bun_B^{ - \lambda} \times_S E \longrightarrow \bun_{T}^{- \lambda} \times_S E &\longrightarrow \bun_{T}^0 \\
(\xi_{T}, x) &\longmapsto \xi_{T} \otimes \lambda(\mc{O}(x)).
\end{align*}
\end{prop}
\begin{proof}
By construction, the $T$-bundles on $D_{\lambda}^{\circ} \times_S E$ classified by the two morphisms to $\bun_{T}^0$ are canonically isomorphic outside the section $D_{\lambda}^{\circ} \to D_{\lambda}^{\circ} \times_S E$ sending a stable map to the image of the node in $E$. Since both have degree $0$, this extends to an isomorphism over all of $D_{\lambda}^{\circ} \times_S E$.
\end{proof}

\subsection{Rigidifications} \label{subsection:rigidification}

In this subsection, we briefly review the rigidification construction of D. Abramovich, A. Corti and A. Vistoli \cite[\S 5.1]{abramovich-corti-vistoli03}.

Let $H$ be a commutative group scheme, so that its classifying stack $\B H$ is a (commutative) group stack. Let $X$ be an $S$-stack equipped with an action $a \colon X \times \B H \to X$. We will say that the action $a$ is \emph{injective} if for every $S$-scheme $U$ and every $U$-point $x \colon U \to X$ of $X$, the induced map
\[ H_U \longrightarrow \ker(\ul{\mrm{Aut}}_X(x) \to \ul{\mrm{Aut}}_S(s)) \]
identifies $H_U$ with a closed subgroup (necessarily normal) of the target. Here $s \colon U \to X$ is the composition of $u$ with the structure map $X \to S$, and $\ul{\mrm{Aut}}_X(x)$ and $\ul{\mrm{Aut}}_S(s)$ are the automorphism group schemes (over $U$) of the points $x$ and $s$ in the stacks $X$ and $S$ respectively.

\begin{prop} \label{prop:rigidificationexistence}
Given an injective action $a \colon X \times \B H \to X$ as above, there exists a unique Artin stack $X_{rig}$ equipped with a $\B H$-invariant morphism $X \to X_{rig}$ such that the $\B H$-action makes $X$ into an $H$-gerbe (i.e., a $\B H$-torsor) over $X_{rig}$.
\end{prop}
\begin{proof}
The proposition is just a restatement of \cite[Theorem 5.1.5]{abramovich-corti-vistoli03}.
\end{proof}

\begin{defn} \label{defn:rigidification}
The stack $X_{rig}$ of Proposition \ref{prop:rigidificationexistence} is called the \emph{rigidification of $X$ with respect to $H$}.
\end{defn}

We will be interested mainly in rigidifications where $H = Z(G)$ is the centre of $G$ and $X = \bun_G$, $\tbun_G$, or $\bun_T$. Observe that we have an action
\begin{align}
\bun_G \times \B Z(G) &\longrightarrow \bun_G \label{eq:rigidificationaction} \\
(\xi_G, \eta) &\longmapsto \xi_G \otimes \eta, \nonumber
\end{align}
where, for $U$ an $S$-scheme, $\xi_G \to U \times_S E$ a $G$-bundle and $\eta \to U$ a $Z(G)$-bundle, we set
\[ \xi_G \otimes \eta = (\xi_G \times_U \eta)/Z(G),\]
where $Z(G)$ acts on $\xi_G \times_U \eta$ by the formula $(x, y)\cdot h = (xh, yh^{-1})$. Given $U$ the induced homomorphism $H_U \to \ul{\mrm{Aut}}(\xi_G)$ corresponds to the action
\begin{align*}
H_U \times_U \xi_G &\longrightarrow \xi_G \\
(h, x) &\longmapsto xh.
\end{align*}
Since this homomorphism is the inclusion of a closed subgroup, the action \eqref{eq:rigidificationaction} is injective, so we have a rigidification $\bun_{G, rig}$ of $\bun_G$ with respect to $Z(G)$. Similarly, $\B Z(G)$ also acts injectively on $\bun_T$ via the inclusion $Z(G)\subseteq T$, giving a rigidification $\bun_{T, rig}$. Since $Z(G)$ acts trivially on the flag variety $G/B$, the action on $\bun_G$ lifts to an injective $\B Z(G)$-action on the Kontsevich-Mori compactification given by
\[ (\xi_G, C, \sigma) \cdot \eta = (\xi_G \otimes \eta, C, \sigma),\]
so we also have a rigidification $\tbun_{G, rig}$. These actions are compatible with the morphisms $\psi \colon \tbun_G \to \bun_G$ and $\mrm{Bl}_B \colon \tbun_G \to \bun_T^0$, so we get an induced diagram
\[
\begin{tikzcd}
\tbun_{G, rig} \ar[r, "\psi_{rig}"] \ar[d, "\mrm{Bl}_{B, rig}"] & \bun_{G, rig} \\ \bun_{T, rig}^0
\end{tikzcd}
\]
of rigidifications.

\subsection{Unstable $G$-bundles and Harder-Narasimhan loci} \label{subsection:unstable}

In this subsection, we review some facts about unstable principal bundles on elliptic curves. We first recall the definition of semistability.

\begin{defn} \label{defn:semistable}
Fix a reductive group $H$, a geometric point $s \colon \spec k \to S$ and a principal $H$-bundle $\xi_H \to E_s$. We say that $\xi_H$ is \emph{stable} (resp., \emph{semistable}) if for every reduction $\xi_P$ of $\xi_H$ to a parabolic subgroup $P \subseteq H$ and every dominant character $\lambda \colon P \to \mb{G}_m$ of $P$ that vanishes on the reduced identity component $Z(H)^\circ$ of the centre $Z(H)$, we have
\[ \deg (\xi_P \times^P \mb{Z}_\lambda) > 0 \qquad \text{(resp.,} \quad \deg (\xi_P \times^P \mb{Z}_\lambda) \geq 0 \; \text{).}\]
We say that $\xi_H$ is \emph{unstable} if it is not semistable.
\end{defn}

\begin{rmk}
As is well known, the locus $\bun_H^{ss} \subseteq \bun_H$ of semistable bundles is an open substack.
\end{rmk}

Let $\xi_G \to E_s$ be an unstable $G$-bundle on a geometric fibre of $X \to S$. Then by definition, there exists a parabolic subgroup $P \subseteq G$, a reduction $\xi_P$ of $\xi_G$ to $P$ and a dominant character $\lambda\colon P \to \mb{G}_m$ such that the line bundle $\xi_P \times^P \mb{Z}_\lambda$ has strictly negative degree. In fact, there is a canonical choice of such a reduction, which is in some sense as destabilising as possible. 

\begin{defn} \label{defn:hardernarasimhanvector}
Let $P \subseteq G$ be a parabolic subgroup with Levi factor $L$. We say that $\mu \in \mb{X}_*(Z(L)^\circ)_{\mb{Q}}$ is a \emph{Harder-Narasimhan vector for $P$} if
\[ \mf{p} = \mrm{Lie}(P) = \bigoplus_{\langle \lambda, \mu(\xi_L) \rangle \geq 0} \mf{g}_\lambda,\]
where 
\[ \mf{g} = \bigoplus_{\lambda \in \mb{X}^*(Z(L)^\circ)} \mf{g}_\lambda,\]
is the weight space decomposition under the action of the torus $Z(L)^\circ$.
\end{defn}

\begin{defn} \label{defn:hardernarasimhan}
Let $\xi_G \to E_s$ be a principal $G$-bundle, and $\xi_P$ a reduction of $\xi_G$ to a parabolic subgroup $P \subseteq G$ with Levi factor $L \cong P/R_u(P)$. We say that the reduction $\xi_P$ is \emph{canonical}, or \emph{Harder-Narasimhan}, if the induced $L$-bundle $\xi_L$ is semistable and $\mu(\xi_L)$ is a Harder-Narasimhan vector for $P$.
\end{defn}

\begin{thm}[{\cite[Theorem 7.3]{behrend95}}] \label{thm:behrendhnreduction}
Given a $G$-bundle $\xi_G \to E_s$, there exists a parabolic subgroup $P \subseteq G$, unique up to conjugation, and a unique Harder-Narasimhan reduction of $\xi_G$ to $P$.
\end{thm}

\begin{rmk}
When $G = GL_n$, Theorem \ref{thm:behrendhnreduction} reduces to the statement that any vector bundle $V \to X_s$ of rank $n$ has a unique filtration
\[ 0 = V_0 \subseteq V_1 \subseteq \cdots \subseteq V_m = V\]
such that $V_i/V_{i - 1}$ is semistable for each $i$ and
\[ \mu(V_1/V_0) > \mu(V_2/V_1) > \cdots > \mu(V_m/V_{m - 1}),\]
where $\mu(U) = \deg U/\rank U$ denotes the slope of a vector bundle $U$. This filtration is called the Harder-Narasimhan filtration on $V$.
\end{rmk}

On an elliptic curve, the Harder-Narasimhan reduction of a $G$-bundle can be reduced further to a Levi subgroup of $G$.

\begin{prop} \label{prop:hnlevi}
Let $\xi_G \to E_s$ be an unstable principal bundle, and let $\xi_P \to E_s$ be its Harder-Narasimhan reduction. Fix a Levi subgroup $L \subseteq P$ (so that $L \to P/R_u(P)$ is an isomorphism) and set $\xi_L = \xi_P \times^P L$. Then there is an isomorphism of $P$-bundles
\[ \xi_P \cong \xi_L \times^L P.\]
In particular, the unstable $G$-bundle $\xi_G$ has a reduction to a semistable bundle for the Levi subgroup $L$.
\end{prop}
\begin{proof}
The proposition is well known. For a proof, see for example \cite[Proposition 2.6.2]{davis19a}.
\end{proof}

\begin{prop} \label{prop:hnembedding}
Fix a parabolic subgroup $P \subseteq G$ with Levi factor $L$ and $\mu \in \mb{X}_*(Z(L)^\circ)_\mb{Q}$ a Harder-Narasimhan vector for $P$. Then the morphism
\begin{equation} \label{eq:hninclusion}
\bun_{P}^{ss, \mu} \longrightarrow \bun_G
\end{equation}
is a locally closed immersion, where $\bun_{P}^{ss, \mu}$ is the open substack of $P$-bundles such that the induced $L$-bundle $\xi_L$ is semistable with slope $\mu(\xi_L) = \mu$.
\end{prop}
\begin{proof}
This proposition is also well known. For a proof, see \cite[Proposition 2.6.5]{davis19a}.
\end{proof}

For future reference, we will use the following terminology for the locally closed substacks coming from Proposition \ref{prop:hnembedding}.

\begin{defn}
If $\xi_G \to E_s$ is an unstable $G$-bundle on a geometric fibre of $E \to S$ with a Harder-Narasimhan reduction to $P \subseteq G$ with slope $\mu$, then the \emph{Harder-Narasimhan locus of $\xi_G$} is the locally closed substack
\[ \bun_{P}^{ss, \mu} \longhookrightarrow \bun_G.\]
\end{defn}

We can also compute the codimension of the Harder-Narasimhan loci. Note that if $L \subseteq P$ is the standard Levi factor (i.e., containing $T$) of a standard parabolic subgroup $P \subseteq G$, then we have an inclusion $Z(L) \subseteq T$ and hence a homomorphism $\mb{X}_*(Z(L)^\circ)_\mb{Q} \to \mb{X}_*(T)_\mb{Q}$ and a pairing
\[ \langle -, - \rangle \colon \mb{X}^*(T) \times \mb{X}_*(Z(L)^\circ)_\mb{Q} \longrightarrow \mb{Q}.\]

\begin{prop} \label{prop:hncodimension}
In the situation of Proposition \ref{prop:hnembedding}, the locally closed immersion \eqref{eq:hninclusion} has codimension $-\langle 2\rho, \mu \rangle$, where $2\rho \in \mb{X}^*(T)$ is the sum of the positive roots.
\end{prop}
\begin{proof}
This is \cite[Proposition 2.6.7]{davis19a}.
\end{proof}

\begin{prop} \label{prop:unstablecodim}
The locus of unstable bundles has codimension $l + 1$ in $\bun_G$, where $l = \dim T$.
\end{prop}
\begin{proof}
Since the locus of unstable bundles in $\bun_G$ is the union of the images of $\bun_{P}^{ss, \mu} \to \bun_{G}$ where $P$ ranges over all standard parabolic subgroups and $\mu$ ranges over all Harder-Narasimhan vectors for $P$, by Proposition \ref{prop:hncodimension}, it suffices to prove that
\[ - \langle 2\rho, \mu \rangle \geq l + 1\]
for all such $P$ and $\mu$ for which $\bun_{P}^{ss, \mu}$ is nonempty, with equality for some such choice of $P$ and $\mu$. Note that $\bun_{P}^{ss, \mu}$ is nonempty if and only if $\langle \varpi_i, \mu \rangle \in \mb{Z}$ for all $\alpha_i \in t(P)$.

Consider the case where $P$ is a maximal parabolic of type $t(P) = \{\alpha_i\}$. Then the conditions on $\mu$ are equivalent to
\[ \mu \in \mb{Z}_{>0}\tn{-span}\left\{-\frac{\varpi_i^\vee}{\langle \varpi_i, \varpi_i^\vee\rangle}\right\}.\]
So
\[ -\langle 2\rho, \mu \rangle \geq \frac{\langle 2\rho , \varpi_i^\vee \rangle}{\langle \varpi_i, \varpi_i^\vee\rangle},\]
which by \cite[Lemma 3.3.2]{friedman-morgan00} is always $\geq l + 1$, with equality achieved for some choice of $\alpha_i$.

More generally, suppose that $P \subseteq G$ is an arbitrary parabolic, choose $\alpha_i \in t(P)$, and let $L_i \supseteq L$ be the Levi factor of the unique  maximal parabolic of type $\{\alpha_i\}$ containing $P$. Let $\tilde{\mu} \in \mb{X}_*(Z(L_i)^\circ)_\mb{Q}$ be the unique element such that $\langle \varpi_i, \tilde{\mu} \rangle = \langle \varpi_i, \mu\rangle$. Then
\[ \tilde{\mu} \in \mb{Z}_{>0}\tn{-span}\left\{-\frac{\varpi_i^\vee}{\langle \varpi_i, \varpi_i^\vee\rangle}\right\},\]
so $-\langle 2\rho, \tilde{\mu} \rangle \geq l + 1$ as shown above. But
\[ -\langle 2\rho, \tilde{\mu} \rangle = -\sum_{\substack{\alpha \in \Phi_+ \\ \langle \alpha, \varpi_i^\vee \rangle > 0}} \langle \alpha, \tilde{\mu} \rangle = -\sum_{\substack{\alpha \in \Phi_+ \\ \langle \alpha, \varpi_i^\vee\rangle > 0}} \langle \alpha, \mu \rangle \leq -\langle 2\rho, \mu \rangle
\]
since
\[ \sum_{\substack{\alpha \in \Phi_+ \\ \langle \alpha, \varpi_i^\vee\rangle > 0}} \alpha \in \mb{Z}_{\geq 0}\tn{-span}\{\varpi_i\},\]
so we are done.
\end{proof}

\subsection{Bruhat cells for $B$-bundles} \label{subsection:bruhat}

Let $P \subseteq G$ be a standard parabolic subgroup. The Bruhat decomposition 
\[ G/B = \coprod_{w \in W_P \bslash W} PwB/B \]
into $P$-orbits is an important tool in the study of the flag variety $G/B$. Here $w$ ranges over any fixed set of coset representatives for the Weyl group $W_P = W_L$ of the Levi factor $L \subseteq P$ inside the Weyl group $W = N_G(T)/T$ of $G$. The purpose of this subsection is to review the construction and structure of a closely related collection of locally closed substacks in
\[ \bun_P \times_{\bun_{G}} \bun_B. \]

The construction is as follows: the Bruhat decomposition of $G/B$ gives a decomposition
\[ \B P \times_{\B G} \B B = P \bslash G/B = \coprod_{w \in W_P \bslash W} P \bslash PwB/B \cong \coprod_{w \in W_P \bslash W} \B(P \cap wBw^{-1}) \]
into disjoint locally closed substacks, and hence a family of disjoint locally closed substacks
\begin{equation} \label{eq:bruhatinclusion}
\bun_{P \cap wBw^{-1}} \longhookrightarrow \bun_{P} \times_{\bun_{G}} \bun_B
\end{equation}
for $w \in W_P \bslash W$.

\begin{defn} \label{defn:bruhat}
If $w \in W_P \bslash W$ and $\lambda \in \mb{X}_*(T)$, the associated \emph{Bruhat cell} is
\begin{align*}
C^{w, \lambda}_P &= \bun_{P \cap wBw^{-1}} \times_{\bun_B} \bun_B^\lambda \\
&\subseteq \bun_P \times_{\bun_G} \bun_{B}^\lambda,
\end{align*}
where the inclusion is the restriction of \eqref{eq:bruhatinclusion} to $B$-bundles of degree $\lambda$.
\end{defn}

Observe that the composition
\[ T \longrightarrow P \cap wBw^{-1} \longrightarrow (P \cap wBw^{-1})/[P \cap wBw^{-1}, P \cap wBw^{-1}]\]
is an isomorphism, so that the degree of a $P \cap wBw^{-1}$-bundle is naturally an element of $\mb{X}_*(T)$.

\begin{prop} \label{prop:bruhatdegrees}
The natural projection $C^{w, \lambda}_P \to \bun_{P \cap wBw^{-1}}$ factors through an isomorphism
\[ C^{w, \lambda}_{P} \cong \bun_{P \cap wBw^{-1}}^{w\lambda}.\]
\end{prop}
\begin{proof}
This is an easy special case of \cite[Proposition 3.7.4]{davis19a}.
\end{proof}

The decomposition $P = L \ltimes R_u(P)$ gives a description of the Bruhat cell $C^{w, \lambda}_P$ in terms of $L$ and $R_u(P)$. In the following proposition, if $\mu \in \mb{X}_*(T)$, then we write $\xi_L$ and $\xi_{L \cap wBw^{-1}}$ respectively for the universal $L$-bundle and $L \cap wBw^{-1}$ bundle on
\[ \bun_{L \cap wBw^{-1}}^{\mu} \times_{\bun_{L}} \bun_{P} \times_S E,\]
$\mc{U} = \xi_L \times^L R_u(P)$ for the associated unipotent group scheme, $\mc{U}_w = \xi_{L \cap wBw^{-1}} \times^{L \cap wBw^{-1}} (R_u(P) \cap wBw^{-1}) \subseteq \mc{U}$, and $\xi_{\mc{U}} = \xi_P/L$ for the associated $\mc{U}$-bundle. If $X \to X' \to S'$ are morphisms of stacks, we also write $\Gamma_{S'}(X', X)$ for the $S'$-stack whose functor of points sends $U \to S'$ to the groupoid of sections of the morphism $X \times_{S'} U \to X' \times_{S'} U$.

\begin{prop}[{\cite[Proposition 3.7.5]{davis19a}}] \label{prop:bruhatstructure}
In the setup above, there is an isomorphism
\[ \bun_{P \cap wBw^{-1}}^\mu \cong \Gamma_M(M \times_S E, \xi_{\mc{U}}/\mc{U}_w), \]
where
\[ M = \bun_{L \cap wBw^{-1}}^{\mu} \times_{\bun_{L}} \bun_{P}.\]
\end{prop}

For future reference, we define the following set of well-behaved coset representatives for $W_P$.

\begin{prop} \label{prop:cosetreps}
The set
\[ W^0_P = \{w \in W \mid w^{-1}\alpha_i \in \Phi_+ \; \text{for}\; \alpha_i \in \Delta \setminus t(P)\} \]
is a complete set of coset representatives for $W_P$ in $W$. Moreover, if $w \in W^0_P$, then $L \cap wBw^{-1} = L \cap B$.
\end{prop}
\begin{proof}
The result is well-known and elementary. For a proof, see for example \cite[Propositions 3.7.1 and 3.7.2]{davis19a}.
\end{proof}

Unlike the Bruhat cells for the flag variety, the cells $C_{P}^{w, \lambda}$ do not cover the stack $\bun_{P} \times_{\bun_{G}}\bun_B^\lambda$. However, by giving bounds on the degrees of sections of flag variety bundles, the following proposition can often be used to show that they do cover the preimages of certain substacks of interest in $\bun_P$. In what follows, we write
\[ C^{w,\lambda}_{P, \xi_P} = \{\xi_P\} \times_{\bun_P} C^{w, \lambda}_{P}\]
for $\xi_P \in \bun_P$. For $\lambda, \lambda' \in \mb{X}_*(T)$, we also write $\lambda' < \lambda$ if $\lambda - \lambda' \in \mb{X}_*(T)_+$.

\begin{prop} \label{prop:levidegreebound}
Let $\xi_P \to E_s$ be a $P$-bundle on a geometric fibre of $E \to S$, and suppose there exists a point in $\bun_{P} \times_{\bun_{G}} \bun_{B}^\lambda$ over $\xi_P$ that does not lie in any Bruhat cell. Then there exists $w \in W^0_{P} \setminus \{1\}$ and $\lambda' < \lambda$ such that $C_{P, \xi_L \times^L P}^{w, \lambda'} \neq \emptyset$, where $\xi_L = \xi_P \times^P L$ is the associated $L$-bundle.
\end{prop}
\begin{proof}
We briefly sketch the argument here; for a more detailed proof, see \cite[Proposition 3.7.6]{davis19a}.

Assume for simplicity of notation that $S = \spec k$ for some algebraically closed field $k$. Unravelling the definitions, the assumption of the proposition is equivalent to the assumption that we have a section $\sigma \colon E \to \xi_P \times^P G/B$ of degree $\lambda$ that does not factor through any Bruhat cell $\xi_P \times^P PwB/B$. 

The strategy of the proof is to construct a degeneration of $\xi_P$ to the bundle $\xi_L \times^L P$, together with a degeneration of $\sigma$ to a stable map $\sigma' \colon C \to \xi_L \times^L G/B$ such that the restriction of $\sigma'$ to the irreducible component $E \subseteq C$ factors through some Bruhat cell. The statement of the proposition then follows since $\sigma$ and $\sigma'$ have the same degree, and the degree of $\sigma'$ restricted to the union of rational components must lie in $\mb{X}_*(T)_+$.

The degeneration of $\xi_P$ is constructed by choosing a cocharacter $\mu \in \mb{X}_*(Z(L)^\circ)$ such that the action
\begin{align*}
\mb{G}_m \times P &\longrightarrow \mb{G}_m \times P \\
(t, p) &\longmapsto (t, \mu(t)p\mu(t)^{-1})
\end{align*}
extends (uniquely) to a morphism $\mb{A}^1 \times P \to P$ of group schemes over $\mb{A}^1$ whose fibre over $0$ is the composition $P \to L \to P$. Induction of $\xi_P$ along this family of homomorphisms defines the desired degeneration $\mb{A}^1_k \to \bun_P$. Over $\mb{G}_{m, k}$, the $\mb{G}_m$-action on $\sigma$ through $\mu$ gives a lift to a family of stable maps to the associated $G/B$-bundle, so we get a family of stable maps over some finite cover of $\mb{A}^1_k$ by properness of stable maps. It is straightforward to check that this gives a degeneration with the desired properties.
\end{proof}

\subsection{Descent for ramified Galois coverings} \label{subsection:descent}

In this subsection, we state a simple result on descent for line bundles along ramified Galois coverings, which is an important technical tool in the proof of the elliptic Chevalley isomorphism in \S\ref{subsection:chevalley}.

To simplify numerous statements later on, it will be convenient to package the data of line bundles and their sections as follows. For any stack $X$ over $S$, we will write $\piccat(X)$ for the category of line bundles on $X$. This is a symmetric monoidal category under tensor product, which has an enrichment over the category $\mc{O}_{S_{\liset}}\tn{-mod}$ of sheaves of $\mc{O}$-modules on the lisse-\'etale site of $S$ defined by the formula
\[ \ul{\hom}(L, L') = {\pi_X}_*(L^\vee \otimes L')\]
for $L, L' \in \piccat(X)$, where $\pi_X \colon X \to S$ is the structure morphism. If $\Gamma$ is a finite group acting on $X$ over $S$, then we write $\piccat^\Gamma(X)$ for the category of $\Gamma$-linearised line bundles on $X$. This is again a symmetric monoidal category with enrichment over $\mc{O}_{S_{\liset}}\tn{-mod}$ defined by the formula
\[ \ul{\hom}(L, L') = {\pi_X}_*(L^\vee \otimes L')^\Gamma.\]
If $X$ is proper and representable over $S$, then $\piccat(X)$ and $\piccat^\Gamma(X)$ are actually enriched over the full subcategory $\mrm{Coh}(S) \subseteq \mc{O}_{S_{\liset}}\tn{-mod}$ of coherent sheaves on $S$.

\begin{defn} \label{defn:ramifiedgaloiscovering}
Let $f \colon X \to Z$ be a morphism of smooth stacks over $S$, with $Z$ connected. We say that $f$ is a \emph{ramified Galois covering relative to $S$ with Galois group $\Gamma$} if
\begin{enumerate}[(1)]
\item the morphism $f$ is representable and finite, and
\item there exists an open substack $U \subseteq Z$ such that $f^{-1}(U) \to U$ is an \'etale Galois covering with Galois group $\Gamma$, and $U$ is dense in every fibre of $Z \to S$.
\end{enumerate}
\end{defn}

\begin{rmk}
Note that if $f \colon X \to Z$ is a ramified Galois covering relative to $S$, then $f$ is automatically flat, since it is a finite morphism between regular stacks of the same dimension.
\end{rmk}

\begin{lem}
Let $f \colon X \to Z$ be a ramified Galois covering relative to $S$ with Galois group $\Gamma$, and let $U \subseteq Z$ be as in Definition \ref{defn:ramifiedgaloiscovering}. Then the action of $\Gamma$ on $f^{-1}(U)$ extends uniquely to an action on $X$ over $Z$.
\end{lem}
\begin{proof}
Since for any smooth (connected) chart $V \to Z$, the pullback $V \times_Z X \to V$ is a ramified Galois covering relative to $S$ with Galois group $\Gamma$, by descent for morphisms of stacks, it suffices to prove the claim in the case where $Z$ (and hence $X$) is a regular affine scheme. So we can assume $Z = \spec A$ and $X = \spec B$, with $A \to B$ a finite flat extension of regular rings, with $\spec A$ connected. By assumption, we have $\spec K \otimes_A B \to \spec K$ a Galois covering with Galois group $\Gamma$, where $K = \mrm{Frac}(A)$ is the fraction field of $A$. Since $B \subseteq K \otimes_A B$ is the subring of elements integral over $A$, it follows that $B$ is preserved by the action of $\Gamma$, which completes the proof.
\end{proof}

\begin{defn} \label{defn:ramgaloispicgood}
Let $f \colon X \to Z$ be a ramified Galois covering with Galois group $\Gamma$, and let $L$ be a $\Gamma$-linearised line bundle on $X$. We say that $L$ is \emph{good} if for every $\gamma \in \Gamma$, the morphism
\[ \gamma \colon L|_{X^\gamma_{(1)}} \longrightarrow L|_{X^\gamma_{(1)}} \]
is the identity, where $X^\gamma_{(1)} \subseteq X^\gamma$ denotes the open substack of points in the fixed locus $X^\gamma$ (relative to $Z$) at which $X^\gamma \subseteq X$ has codimension $\leq 1$. We write $\mrm{Pic}^\Gamma(X)_{good} \subseteq \mrm{Pic}^\Gamma(X)$ for the subgroup of good $\Gamma$-linearised line bundles, and $\piccat^\Gamma(X)_{good} \subseteq \piccat^\Gamma(X)$ for the corresponding full subcategory.
\end{defn}

\begin{rmk}
It is important in Definition \ref{defn:ramgaloispicgood} that we take fixed loci relative to $Z$ and not $S$. The fixed locus $X^\gamma$ relative to $Z$ is by definition the fibre product
\[
\begin{tikzcd}
X^\gamma \ar[r] \ar[d] & X \ar[d, "\Delta_{X/Z}"] \\
X \ar[r, "{(\id, \gamma)}"] & X \times_Z X,
\end{tikzcd}
\]
which is a closed substack of $X$ since $X$ is representable and separated over $Z$. Taking fixed loci relative to $S$ would amount to replacing $\Delta_{X/Z}$ with $\Delta_{X/S}$, which will not be a closed immersion if $X \to S$ is not representable. 
\end{rmk}

Our descent result for ramified Galois coverings is the following.

\begin{prop} \label{prop:ramgaloisdescent}
Let $f \colon X \to Z$ be a ramified Galois covering of smooth stacks over $S$, with Galois group $\Gamma$. Assume that for any $\gamma, \gamma' \in \Gamma \setminus \{1\}$ with $\gamma \neq \gamma'$, the intersection of the fixed loci (relative to $Z$) $X^\gamma \cap X^{\gamma'} \subseteq X$ has codimension at least $2$. Then the pullback functor $\piccat(Z) \to \piccat^\Gamma(X)$ factors through an equivalence
\[ \piccat(Z) \overset{\sim}\longrightarrow \piccat^\Gamma(X)_{good}\]
of symmetric monoidal categories enriched over $\mc{O}_{S_{\liset}}\tn{-mod}$.
\end{prop}

For a proof of Proposition \ref{prop:ramgaloisdescent}, we direct the reader to \cite[Proposition 4.2.12]{davis19a}.

\section{Constructing the fundamental diagram} \label{section:grothendieckspringer}

In this section, we give the construction of the diagram \eqref{eq:introbunggrothendieckspringer}.

In \S\ref{subsection:chevalley}, we state and prove the elliptic Chevalley isomorphism $\mrm{Pic}(\bun_G) \cong \mrm{Pic}^W(Y)_{good}$. We compute $\mrm{Pic}^W(Y)_{good}$ (and hence $\mrm{Pic}(\bun_G)$) in \S\ref{subsection:thetabundle}; the computation allows us to identify an ample generator $\Theta_Y$, which we call the theta bundle. Finally, in \S\ref{subsection:fundiag}, we explain how the elliptic Chevalley isomorphism gives a morphism $\chi \colon \bun_G \to \Theta_Y^{-1}/\mb{G}_m$, which we show fits into a diagram \eqref{eq:introbunggrothendieckspringer} as claimed. Everything we do in this section works equally well for both $\bun_G$ and its rigidification $\bun_{G, rig}$.

\subsection{The elliptic Chevalley isomorphism} \label{subsection:chevalley}

The classical Chevalley isomorphisms $\mf{g}\sslash G \cong \mf{t}\sslash W$ and $G \sslash G \cong T \sslash W$ are essential ingredients in the construction of the additive and multiplicative Grothendieck-Springer resolutions (as diagrams), as they provide the base change maps $\mf{t} \to \mf{g}\sslash G$ and $T \to G\sslash G$. 

In this subsection, we prove an elliptic analogue of these statements, which is the main ingredient in constructing the diagram \eqref{eq:fundiagbody}.

One can think of the classical (say, additive) Chevalley isomorphism as an isomorphism between the ring of regular functions on the stack $\mf{g}/G$ and the ring of $W$-invariant functions on the affine variety $\mf{t}$. So a naive elliptic analogue would be to identify regular functions on the stack $\bun_G$ with $W$-invariant regular functions on some variety. However, since the coarse moduli space of semistable $G$-bundles is projective rather than affine, there are not enough global regular functions on $\bun_G$ to make such a statement particularly useful. Instead, we will give a correspondence between line bundles on $\bun_G$ and certain $W$-linearised line bundles on the abelian variety $Y$, such that the space of global sections of a line bundle on $\bun_G$ is naturally isomorphic to the space of $W$-invariant sections of the corresponding line bundle on $Y$.

\begin{rmk}
The Weyl group $W$ acts naturally on the torus $T$, and hence on the abelian variety $Y$ over $S$. Explicitly, this action is given by
\[ s_\alpha(y) = y - \alpha^\vee(\alpha(y))\]
for $\alpha \in \Phi$, where we use the natural group structure on $Y$, and for $\lambda \in \mb{X}^*(T)$ (resp., $\mu \in \mb{X}_*(T)$), we write $\lambda \colon Y \to \mrm{Pic}^0_S(E)$ (resp., $\mu \colon \mrm{Pic}^0_S(E) \to Y$) for the morphism induced by $\lambda \colon T \to \mb{G}_m$ (resp., $\mu \colon \mb{G}_m \to T$).
\end{rmk}

\begin{defn} \label{defn:picwygood}
Let $L$ be a $W$-linearised line bundle on $Y$. In the notation of \S\ref{subsection:descent}, we say that $L$ is \emph{good} if for every root $\alpha \in \Phi_+$, the morphism
\[ s_\alpha \colon L|_{Y^{s_\alpha}} \longrightarrow L|_{Y^{s_\alpha}}\]
is the identity, where $Y^{s_\alpha} \subseteq Y$ is the fixed locus of $s_\alpha \colon Y \to Y$. We write $\mrm{Pic}^W(Y)_{good} \subseteq \mrm{Pic}^W(Y)$ for the subgroup of good $W$-linearised line bundles, and $\piccat^W(Y)_{good} \subseteq \piccat^W(Y)$ for the corresponding full subcategory.
\end{defn}

\begin{rmk}
Over the smooth locus of $Y \sslash W$, the morphism $Y \to Y \sslash W$ is a ramified Galois covering. Definition \ref{defn:picwygood} is consistent with Definition \ref{defn:ramgaloispicgood} over this locus, since for $w \in W$ we have $Y^w_{(1)} \neq \emptyset$ if and only if $w = s_\alpha$ is the reflection in some root $\alpha \in \Phi_+$.
\end{rmk}

\begin{thm}[{Elliptic Chevalley isomorphism}] \label{thm:bungchevalley}
There are equivalences
\[
\piccat(\bun_G) \simeq \piccat(\bun_{G, rig}) \simeq \piccat^W(Y)_{good}
\]
of symmetric monoidal categories enriched over $\mc{O}_{S_{\liset}}\tn{-mod}$.
\end{thm}

\begin{rmk}
In more down to earth terms, Theorem \ref{thm:bungchevalley} states that there are isomorphisms
\[ \mrm{Pic}(\bun_G) \cong \mrm{Pic}(\bun_{G, rig}) \cong \mrm{Pic}^W(Y)_{good} \]
of abelian groups, and isomorphisms
\[ {\pi_{\bun_G}}{\vphantom{p}}_*L_{\bun_G} \cong {\pi_{\bun_{G, rig}}}{\vphantom{p}}_* L_{\bun_{G, rig}} \cong ({\pi_Y}{\vphantom{p}}_*L)^W \]
of sheaves of $\mc{O}$-modules on $S_{\liset}$, compatible with tensor products, for $L \in \mrm{Pic}^W(Y)_{good}$ corresponding to $L_{\bun_G} \in \mrm{Pic}(\bun_G)$ and $L_{\bun_{G, rig}} \in \mrm{Pic}(\bun_{G, rig})$.
\end{rmk}

\begin{proof}[Proof of Theorem \ref{thm:bungchevalley}]
We give the outline of the proof here, and fill in the details in the rest of the subsection.

First, by Lemma \ref{lem:bigpicrestriction} and Proposition \ref{prop:unstablecodim}, the restriction functors
\[ \piccat(\bun_G) \longrightarrow \piccat(\bun_G^{ss}) \quad \text{and} \quad \piccat(\bun_{G, rig}) \longrightarrow \piccat(\bun_{G, rig}^{ss})\]
are equivalences of symmetric monoidal categories enriched over $\mc{O}_{S_{\liset}}\tn{-mod}$. So it suffices to prove the theorem with $\bun_G^{ss}$ and $\bun_{G, rig}^{ss}$ in place of $\bun_G$ and $\bun_{G, rig}$.

Consider the commutative diagram
\[
\begin{tikzcd}
\tbun_G^{ss, reg} \ar[r, hook] \ar[d] & \tbun_G^{ss} \ar[d] \ar[r, "\mrm{Bl}_B^{ss}"] & \bun_T^0 \ar[r] & Y \\
\bun_G^{ss, reg} \ar[r, hook] & \bun_G^{ss},
\end{tikzcd}
\]
and its rigidification
\[
\begin{tikzcd}
\tbun_{G, rig}^{ss, reg} \ar[r, hook] \ar[d] & \tbun_{G, rig}^{ss} \ar[d] \ar[r, "\mrm{Bl}_{B, rig}^{ss}"] & \bun_{T, rig}^0 \ar[r ] & Y \\
\bun_{G, rig}^{ss, reg} \ar[r, hook] & \bun_{G, rig}^{ss},
\end{tikzcd}
\]
where $\bun_G^{ss, reg} \subseteq \bun_G^{ss}$ and $\tbun_G^{ss, reg} \subseteq \tbun_G^{ss}$ are the big open substacks of regular semistable bundles (see Definition \ref{defn:regularbundle} and Proposition \ref{prop:regcodim}). By Lemma \ref{lem:bigpicrestriction} again, restriction of line bundles gives equivalences
\[ \piccat(\bun_G^{ss}) \overset{\sim}\longrightarrow \piccat(\bun_G^{ss, reg}) \quad \text{and} \quad \piccat(\bun_{G, rig}^{ss}) \overset{\sim}\longrightarrow \piccat(\bun_{G, rig}^{ss, reg}).\]
By Proposition \ref{prop:regramgalois}, the morphisms $\tbun_{G, rig}^{ss, reg} \to \bun_{G, rig}^{ss, reg}$ and $\tbun_G^{ss, reg} \to \bun_G^{ss, reg}$ are ramified Galois coverings with Galois group $W$ satisfying the hypotheses of Proposition \ref{prop:ramgaloisdescent} such that the Galois action covers the natural $W$-action on $Y$. So there are equivalences
\[ \piccat(\bun_G^{ss, reg}) \overset{\sim}\longrightarrow \piccat^W(\tbun_G^{ss, reg})_{good} \quad \text{and} \quad \piccat(\bun_{G, rig}^{ss, reg}) \overset{\sim}\longrightarrow \piccat^W(\tbun_{G, rig}^{ss, reg})_{good} \]
of symmetric monoidal categories enriched over $\mc{O}_{S_{\liset}}\tn{-mod}$. But by Proposition \ref{prop:picwpullback} the natural pullback functors give equivalences
\[ \piccat^W(Y)_{good} \overset{\sim}\longrightarrow \piccat^W(\tbun_G^{ss, reg})_{good} \]
and
\[ \piccat^W(Y)_{good} \overset{\sim}\longrightarrow \piccat^W(\tbun_{G, rig}^{ss, reg})_{good}, \]
which completes the proof.
\end{proof}

\begin{rmk}
From the proof, it is clear that the equivalence
\[ \piccat(\bun_{G, rig}) \overset{\sim}\longrightarrow \piccat(\bun_G) \]
of Theorem \ref{thm:bungchevalley} is just the obvious pullback functor.
\end{rmk}

The rest of this subsection is concerned with proving the various propositions and lemmas quoted in the proof of Theorem \ref{thm:bungchevalley}. We begin by introducing and studying the substack of regular semistable bundles.

\begin{defn} \label{defn:regularbundle}
We say that a semistable $G$-bundle $\xi_G \in \bun_G^{ss}$ is \emph{regular} if $\dim \psi^{-1}(\xi_G) = 0$. We write $\bun_G^{ss, reg} \subseteq \bun_G^{ss}$ for the open substack of regular semistable bundles. We also write
\[ \tbun_G^{ss, reg} = \psi^{-1}(\bun_G^{ss, reg}).\]
\end{defn}

\begin{rmk}
There is another notion of regular semistable bundle in use in the literature, namely that of a semistable principal bundle whose automorphism group has minimal dimension $l = \rank G$. We will see later on (Proposition \ref{prop:regularautomorphisms}) that this notion agrees with ours.
\end{rmk}

In classical Springer theory, the simplest regular elements to describe are the regular semisimple ones. The same is true in our context.
 
\begin{defn}
We say that a point $y \colon \spec k \to Y$ over $s \colon \spec k \to S$ is \emph{strictly regular} if for every root $\alpha \in \Phi_+$, we have $\alpha(y) \neq 0 \in \mrm{Pic}^0(Y_s)$. We write $Y^{sreg}$ for the open subset of strictly regular points and $\bun_T^{0, sreg}$ and $\tbun_G^{ss, sreg}$ for the preimages in $\bun_T^0$ and $\tbun_G^{ss}$ respectively. We call a $G$-bundle $\xi_G \in \bun_G^{ss}$ \emph{strictly regular}, or \emph{regular semisimple}, if it lies in the image $\bun_G^{ss, sreg} = \psi(\tbun_G^{ss, sreg})$.
\end{defn}

\begin{lem} \label{lem:sregiso}
The morphism $\tbun_G^{ss, sreg} \to \bun_T^{0, sreg}$ is an isomorphism.
\end{lem}
\begin{proof}
Since $\tbun_G^{ss, sreg} \to \bun_T^{0, sreg}$ is smooth, it suffices to show that each geometric fibre is trivial. But this is clear from Lemma \ref{lem:unipotentreduction} below, so we are done.
\end{proof}

\begin{lem} \label{lem:unipotentreduction}
Fix a geometric point $s \colon \spec k \to S$ and a degree $0$ $T$-bundle $\xi_T$ on $E_s$ corresponding to $y \in Y_s$, and let $U \subseteq R_u(B)$ be a unipotent closed subgroup scheme that is invariant under conjugation by $T$. Assume that for all $\alpha \in \Phi_-$ such that $\alpha(y) = 0$, we have $U_\alpha \subseteq U$, where $U_\alpha = \mb{G}_a$ is the root subgroup corresponding to $\alpha$. Then the induced bundle morphism
\[ \bun_{T U}(E_s)_{\xi_T} \longrightarrow \bun_{B}(E_s)_{\xi_T} \]
is an isomorphism, where the subscript denotes the fibre over $\xi_T$ of the natural morphism to $\bun_{T}(E_s)$.
\end{lem}
\begin{proof}
Since the statement only concerns individual geometric fibres of $E \to S$, we can assume for simplicity that $S = \spec k$ and $E = E_s$.

Writing $\mc{R}$ and $\mc{U}$ for the group schemes $\xi_T \times^T R_u(B)$ and $\xi_T \times^T U$, we have canonical isomorphisms
\[ \bun_{T U}(E)_{\xi_T} \cong \bun_{\mc{U}} \quad \text{and} \quad \bun_{B}(E)_{\xi_T} \cong \bun_{\mc{R}},\]
so it suffices to show that the natural morphism $\bun_{\mc{U}} \to \bun_{\mc{R}}$ is an isomorphism.

Let $R_u(B) = R_u(B)^{\geq 1} \supseteq R_u(B)^{\geq 2} \supseteq \cdots$ be the filtration on $R_u(B)$ according to root height, and $U^{\geq i} = U \cap R_u(B)^{\geq i}$ for all $i$. Then writing $\mc{R}^{\geq i} = \xi_T \times^T R_u(B)^{\geq i}$ and $\mc{U}^{\geq i} = \xi_T \times^T U^{\geq i}$, by applying the long exact sequence in nonabelian cohomology repeatedly (for example, in the form \cite[Proposition 2.4.2, (2)]{davis19a}), we reduce to proving that 
\begin{equation} \label{eq:unipotentreduction1}
\bun_{\mc{U}^{\geq i - 1}/\mc{U}^{\geq i}} \longrightarrow \bun_{\mc{R}^{\geq i - 1}/\mc{R}^{\geq i}}
\end{equation}
is an isomorphism for all $i$. The group schemes $\mc{U}^{\geq i - 1}/\mc{U}^{\geq i}$ and $\mc{R}^{\geq i - 1}/\mc{R}^{\geq i}$ are direct sums of degree $0$ line bundles such that $\mc{U}^{\geq i - 1}/\mc{U}^{\geq i}$ contains all trivial summands of $\mc{R}^{\geq i - 1}/\mc{R}^{\geq i}$. Since $\bun_{\mc{V}} = \spec k$ for $\mc{V} \to E$ a nontrivial line bundle of degree $0$, it follows that \eqref{eq:unipotentreduction1} is an isomorphism as claimed.
\end{proof}

In the following lemma, for $\xi_B \to E_s$ a $B$-bundle and $w \in W$, we write
\[ C^w = \{\xi_B\} \times_{\bun_B} C^{w, 0}_B = C^{w, 0}_{B, \xi_B}\]
in the notation of \S\ref{subsection:bruhat}.

\begin{lem} \label{lem:semistablebruhatsurjective}
Let $\xi_B \to E_s$ be a $B$-bundle of degree $0$ on a geometric fibre of $E \to S$, and let $\xi_G = \xi_B \times^B G$ be the induced $G$-bundle. Then the morphism
\[ \coprod_{w \in W} C^w \longrightarrow \{ \xi_G \} \times_{\bun_G} \bun_B^0 = \psi^{-1}(\xi_G) \]
is surjective.
\end{lem}
\begin{proof}
Let $\xi_T = \xi_B \times^B T$. Since by Proposition \ref{prop:bruhatdegrees}, $C^{w, \lambda}_{B, \xi_T \times^T B} = \emptyset$ unless $\lambda = 0$, the lemma follows from Proposition \ref{prop:levidegreebound}.
\end{proof}

\begin{prop} \label{prop:reggalois}
We have $\psi^{-1}(\bun_G^{ss, sreg}) = \tbun_G^{ss, sreg}$, and the morphism $\tbun_{G, rig}^{ss, sreg} \to \bun_{G, rig}^{ss, sreg}$ (and hence also $\tbun_G^{ss, sreg} \to \bun_G^{ss, sreg}$) is an \'etale Galois covering with Galois group $W$. In particular, every strictly regular $G$-bundle is regular.
\end{prop}
\begin{proof}
We first show that $\tbun_G^{ss, sreg} \to \bun_G^{ss}$ is \'etale. To see this, let $\xi_B \to E_s$ be the $B$-bundle classified by a geometric point of $\tbun_G^{ss, sreg}$ over $s\colon \spec k \to S$, and observe that the relative tangent complex at $\xi_B$ is given by
\[ \mb{T}_{\tbun_G/\bun_G, \xi_B} = \mb{R}\Gamma(E_s, \xi_B \times^B \mf{g}/\mf{b}).\]
The $B$-module $\mf{g}/\mf{b}$ has a filtration with subquotients isomorphic to $\mf{g}_{\alpha} = \mb{Z}_\alpha$ for $\alpha \in \Phi_+$. Since the associated $T$-bundle is strictly regular, the line bundles $\xi_B \times^B \mb{Z}_{\alpha}$ are nontrivial of degree $0$, so have vanishing cohomology. So $\mb{T}_{\tbun_G/\bun_G, \xi_B} = 0$, and $\psi$ is \'etale at $\xi_B$.

We next compute the fibre of $\psi$ over a geometric point $\xi_G \in \bun_G^{ss, sreg}$ over $s \colon \spec k \to S$. Let $\xi_B$ be a lift of $\xi_G$ to $\tbun_G^{ss, sreg}$ with associated $T$-bundle $\xi_T$. By Lemma \ref{lem:semistablebruhatsurjective}, we have a decomposition
\[ \psi^{-1}(\xi_G) = \coprod_{w \in W} C^w \]
into locally closed subsets indexed by the Weyl group $W$. By Propositions \ref{prop:bruhatdegrees} and \ref{prop:bruhatstructure}, we have
\[ C^w = \Gamma(E_s, \xi_{\mc{U}}/\mc{U}_w) = \bun_{\mc{U}_w}(E_s) \times_{\bun_{\mc{U}}(E_s)} \{\xi_\mc{U}\},\] 
where $\mc{U}$ and $\mc{U}_w$ are the unipotent group schemes $\mc{U} = \xi_T \times^T R_u(B)$ and $\mc{U}_w = \xi_T \times^T (R_u(B) \cap wBw^{-1})$ on $E_s$, and $\xi_{\mc{U}} = \xi_B/T$ is the $\mc{U}$-bundle corresponding to $\xi_B$. But since $\xi_T$ is strictly regular, Lemma \ref{lem:unipotentreduction} shows that $\bun_{\mc{U}}(E_s) = \bun_{\mc{U}_w}(E_s) = \spec k$, so $C^w \cong \spec k$ as well. Since $\psi$ is \'etale, $\psi^{-1}(\xi_G)$ is reduced so we get an isomorphism
\begin{align*}
W \times \spec k &\overset{\sim}\longrightarrow \psi^{-1}(\xi_G) \\
w &\longmapsto C^w
\end{align*}
such that the composition with $\psi^{-1}(\xi_G) \to Y$ sends $w \in W$ to $w^{-1}y$, where $y = \mrm{Bl}_{B, Y}(\xi_B)$. In particular, since $w^{-1}y \in Y^{sreg}$, we have $\psi^{-1}(\xi_G) \subseteq \tbun_G^{ss, sreg}$, so this proves $\psi^{-1}(\bun_G^{ss, sreg}) = \tbun_G^{ss, sreg}$.

Since $\psi$ is proper, the above discussion implies that $\tbun_G^{ss, sreg} \to \bun_G^{ss, sreg}$ is finite \'etale, and hence so is $\tbun_{G, rig}^{ss, sreg} \to \bun_{G, rig}^{ss, sreg}$. To prove that it is a Galois cover with Galois group $W$, we need to show that $W$ acts on $\tbun_{G, rig}^{ss, sreg}$ over $\bun_{G, rig}^{ss, sreg}$, freely and transitively on some (hence every) fibre. By Lemma \ref{lem:sregiso}, we can identify $\tbun_{G, rig}^{ss, sreg} \to \bun_{G, rig}^{ss, sreg}$ with the morphism $\bun_{T, rig}^{0, sreg} \to \bun_{G, rig}^{ss, sreg}$ given by inducing along $T \hookrightarrow G$. Since $N_G(T)$ acts on $\B (T/Z(G))$ over $\B (G/Z(G))$, it acts on $\B T$ over $\B G$ preserving the $\B Z(G)$-action. So we get an action of $N_G(T)$ on $\bun_T^{0, sreg}$ over $\bun_G^{ss, sreg}$ also preserving the $\B Z(G)$-action, and hence on $\bun_{T, rig}^{0, sreg}$ over $\bun_{G, rig}^{ss, sreg}$. Since $\bun_{T, rig}^{0, sreg} \to \bun_{G, rig}^{ss, sreg}$ is \'etale, the connected subgroup $T \subseteq N_G(T)$ must act trivially, so this factors through the desired action of $W$.

By construction, the morphism $\bun_{T, rig}^{0, sreg} \to Y^{sreg}$ is $W$-equivariant. Choose a geometric point $y \colon \spec k \to Y^{sreg}$ over $s \colon \spec k \to S$ such that the stabiliser of $y$ under $W$ is trivial, and let $\xi_T$ be the corresponding $T$-bundle. Then $\psi^{-1}(\xi_T \times^T G)$ maps isomorphically onto the $W$-orbit of $y$ in $Y_s$, so the $W$-action on $\psi^{-1}(\xi_T \times^T G)$ is free and transitive, and we are done.
\end{proof}

\begin{prop} \label{prop:regramgalois}
The morphisms $\tbun_G^{ss, reg} \to \bun_G^{ss, reg}$ and $\tbun_{G, rig}^{ss, reg} \to \bun_{G, rig}^{ss, reg}$ are ramified Galois coverings with Galois group $W$ satisfying the conditions of Proposition \ref{prop:ramgaloisdescent}.
\end{prop}
\begin{proof}
Since $\tbun_G^{ss, sreg} \subseteq \tbun_G^{ss, reg}$ is dense in every fibre over $S$, and both morphisms are proper and quasi-finite, hence finite, both morphisms are ramified Galois coverings with Galois group $W$ by Proposition \ref{prop:reggalois}. It suffices to check the conditions of Proposition \ref{prop:ramgaloisdescent} for $\tbun_G^{ss, reg} \to \bun_G^{ss, reg}$, as the claims for $\tbun_{G, rig}^{ss, reg} \to \bun_{G, rig}^{ss, reg}$ then follow by descent.

We need to show that for every $w, w' \in W \setminus \{1\}$ with $w \neq w'$, the intersection of the $\bun_G^{ss, reg}$-relative fixed loci $(\tbun_G^{ss, reg})^w \cap (\tbun_G^{ss, reg})^{w'}$ has codimension at least $2$ in every fibre over $S$. Observe that since $\tbun_G^{ss, sreg} \to Y^{sreg}$ is $W$-equivariant and $Y \to S$ is representable and separated, it follows by continuity that $\mrm{Bl}_{B, Y}^{reg} \colon \tbun_G^{ss, reg} \to Y$ is also $W$-equivariant. So $(\tbun_G^{ss, reg})^w \subseteq \mrm{Bl}_{B, Y}^{-1}(Y^w)$ for all $w \in W$, where $Y^w$ denotes the fixed locus relative to $S$. Since $\mrm{Bl}_{B, Y}$ is smooth and $Y^w \cap Y^{w'}$ has codimension at least $2$ in every fibre, the result now follows.
\end{proof}

If $U \subseteq X$ is an open set in a stack $X$ over $S$, we say that $U$ is \emph{big} relative to $S$ if the complement of $U \cap X_s$ in $X_s$ has codimension at least $2$ for every geometric point $s \colon \spec k \to S$.

\begin{lem} \label{lem:bigpicrestriction}
Let $\pi_X \colon X \to S$ be a smooth morphism of stacks, and let $U \subseteq X$ be a big open substack relative to $S$. Then the restriction functor
\begin{equation} \label{eq:bigpicrestrictionfunctor}
\piccat(X) \longrightarrow \piccat(U)
\end{equation}
is an equivalence of symmetric monoidal categories enriched over $\mc{O}_{S_{\liset}}\tn{-mod}$.
\end{lem}
\begin{proof}
Since every object of $S_{\liset}$ is a smooth over $S$, hence regular, this follows from the fact that line bundles and their sections extend uniquely over sets of codimension $2$ in regular schemes.
\end{proof}

\begin{prop} \label{prop:regcodim}
The open substacks
\[ \bun_G^{ss, reg} \subseteq \bun_G^{ss} \quad \text{and} \quad \tbun_G^{ss, reg} \subseteq \tbun_G^{ss} \]
are big relative to $S$.
\end{prop}

The proof of Proposition \ref{prop:regcodim} relies on the following construction of $G$-bundles that are regular semistable but not strictly regular.

Let $s \colon \spec k \to S$ be a geometric point, $\alpha \in \Phi_+$ a positive root, and let $y \in Y_s$ satisfy $\alpha(y) = 0$ and $\beta(y) \neq 0$ for all $\beta \in \Phi_+ \setminus \{\alpha\}$. Then by Lemma \ref{lem:unipotentreduction}, the fibre of $\tbun_G^{ss} \to Y$ over $y$ is
\[\tbun_{G , y}^{ss} \cong \bun_{B, y} \cong \bun_{T U_{-\alpha}}(E_s)_{y} \cong \bun_{U_{-\alpha}}(E_s)/T,\]
where the last isomorphism follows from the fact that a $TU_{-\alpha} = T \ltimes U_{-\alpha}$-bundle is the same thing as a $T$-bundle $\xi_T$ plus a $\xi_T \times^T U_{-\alpha} \cong U_{-\alpha}$-bundle (see, e.g., \cite[Proposition 2.4.2, (1)]{davis19a}). Since $H^1(E_s, \mc{O}) = k$, there is a unique $k$-point of $\bun_{U_{-\alpha}}(E_s)/T$ corresponding to a nontrivial $U_{-\alpha}$-bundle. Let $\xi_{T U_{-\alpha}}$ be the corresponding point of $\bun_{TU_{-\alpha}}(E_s)_{y}$ and let $\xi_G = \xi_{T U_{-\alpha}} \times^{T U_{-\alpha}} G$.

\begin{lem} \label{lem:regssexistence}
Let $\xi_G$ be the $G$-bundle defined above. Then the fibre $\psi^{-1}(\xi_G)$ has exactly $|W|/2$ $k$-points. In particular, $\xi_G$ is regular semistable.
\end{lem}
\begin{proof}
Observe that since the subgroup $T U_{-\alpha} \subseteq G$ is conjugate under $N_G(T) \subseteq G$ to $T U_{-\alpha_i}$ for some $\alpha_i \in \Delta$, we can assume without loss of generality that $\alpha = \alpha_i$ is a simple root.

Writing $\xi_B = \xi_{T U_{-\alpha_i}}\times^{TU_{-\alpha_i}} B$, we have $\xi_G = \xi_B \times^B G$, so by Lemma \ref{lem:semistablebruhatsurjective} we get a decomposition
\[ \psi^{-1}(\xi_G) = \coprod_{w \in W} C^w \]
into locally closed subschemes, where by Proposition \ref{prop:bruhatstructure} we can identify $C^w$ with the space of sections of
\[ \xi_B \times^B BwB/B = \xi_{T U_{-\alpha_i}} \times^{T U_{-\alpha_i}} R_u(B)/(R_u(B) \cap wBw^{-1}).\]
If $U_{-\alpha_i} \nsubseteq R_u(B) \cap wBw^{-1}$, i.e., if $w^{-1}\alpha_i \in \Phi_-$, then there is a $T U_{-\alpha_i}$-equivariant morphism
\[ R_u(B)/(R_u(B) \cap wBw^{-1}) \longrightarrow U_{-\alpha_i},\]
so $\xi_{T U_{-\alpha_i}} \times^{T U_{-\alpha_i}} R_u(B)/(R_u(B) \cap wBw^{-1})$ has no sections since $\xi_{T U_{-\alpha_i}} \times^{T U_{-\alpha_i}} U_{-\alpha_i} = \xi_{U_{-\alpha_i}}$ has none, and hence $C^w = \emptyset$. If $U_{-\alpha_i} \subseteq R_u(B) \cap wBw^{-1}$, i.e., if $w^{-1}\alpha_i \in \Phi_+$, then the natural morphisms
\[ \bun_{T U_{-\alpha_i}}(E_s)_{\xi_T} \longrightarrow \bun_{T R_u(B) \cap wBw^{-1}}(E_s)_{\xi_T} \longrightarrow \bun_{B}(E_s)_{\xi_T} \]
are isomorphisms by Lemma \ref{lem:unipotentreduction}, where $\xi_T \to E_s$ is the $T$-bundle corresponding to $y$, which implies that $C^w = \spec k$. Since there are exactly $|W|/2$ elements of $W$ satisfying $w^{-1}\alpha_i \in \Phi_+$, this proves the lemma.
\end{proof}

\begin{proof}[Proof of Proposition \ref{prop:regcodim}]
It suffices to prove the statement for $\tbun_G^{ss, reg} \subseteq \tbun_G^{ss}$; the statement for $\bun_G^{ss, reg} \subseteq \bun_G^{ss}$ then follows immediately. Since the property of being big is defined fibrewise, it suffices to prove the claim when $S = \spec k$ for some algebraically closed field $k$.

We need to show that the complement of $\tbun_{G}^{ss, reg}$ in $\tbun_G^{ss}$ has codimension at least $2$. Since $\tbun_G^{ss, sreg} \subseteq \tbun_G^{ss, reg}$ and $\mrm{Bl}_{B, Y}^{ss} \colon \tbun_G^{ss} \to Y$ is smooth, it suffices to show that $\tbun_G^{ss, reg} \cap (\mrm{Bl}_{B, Y}^{ss})^{-1}(D)$ is dense in $(\mrm{Bl}_{B, Y}^{ss})^{-1}(D)$ for all irreducible components $D$ of $Y\setminus Y^{sreg} = \bigcup_{\alpha \in \Phi_+} Y^{s_\alpha}$. For each such $D$, $(\mrm{Bl}_{B, Y}^{ss})^{-1}(D)$ is smooth and connected, hence irreducible, and $\tbun_G^{ss, reg} \cap (\mrm{Bl}_{B, Y}^{ss})^{-1}(D)$ is open. So it suffices to show that $\tbun_G^{ss, reg} \cap (\mrm{Bl}_{B, Y}^{ss})^{-1}(D)$ is nonempty. But there exists $\alpha \in \Phi_+$ such that the generic point $y$ of $D \subseteq Y$ satisfies $\alpha(y) = 0$ and $\beta(y) \neq 0$ for all $\beta \in \Phi_+ \setminus \{\alpha\}$, so this follows from Lemma \ref{lem:regssexistence}.
\end{proof}

\begin{lem} \label{lem:tbunreglocal}
Assume that $S$ is a scheme, let $\alpha \in \Phi_+$ and let $y \in Y^{s_\alpha}$ be the generic point of an irreducible component. Then the morphism
\begin{equation} \label{eq:tbunreglocal1}
\tbun_G^{ss, reg} \times_Y \spec \mc{O}_{Y, y} \longrightarrow \bun_G \times_S \spec \mc{O}_{Y, y}
\end{equation}
is a locally closed immersion.
\end{lem}
\begin{proof}
Since $\mrm{Stab}_W(y) = \{1, s_\alpha\}$, the proofs of Lemma \ref{lem:regssexistence} and Proposition \ref{prop:regcodim} show that \eqref{eq:tbunreglocal1} separates points, so it is enough to prove that it is formally unramified. This is equivalent to the claim that the morphism
\begin{equation} \label{eq:tbunreglocal2}
H^0(E_s, \xi_B \times^B \mf{g}/\mf{b}) \longrightarrow H^1(E_s, \mf{t} \otimes \mc{O}_{E_s})
\end{equation}
induced by the extension of $B$-modules
\[ 0 \longrightarrow \mf{t} \longrightarrow \mf{g}/R_u(\mf{b}) \longrightarrow \mf{g}/\mf{b} \longrightarrow 0\]
is injective, where $\xi_B \to E_s$ is the unique $B$-bundle in $\tbun_{G}^{ss, reg} \times_Y \{y\}$. Since $\xi_B$ is induced from the unique $TU_{-\alpha}$-bundle $\xi_{TU_{-\alpha}}$ with nontrivial associated $U_{-\alpha}$-bundle,
we can identify \eqref{eq:tbunreglocal2} with the morphism
\begin{equation} \label{eq:tbunreglocal3}
H^0(E_s, \xi_{TU_{-\alpha}} \times^{T U_{-\alpha}} \mf{g}_{\alpha}) \longrightarrow H^1(E_s, \mf{t} \otimes \mc{O}_{E_s})
\end{equation}
induced by the extension
\begin{equation} \label{eq:tbunreglocal4}
0 \longrightarrow \mf{t} \longrightarrow (\mf{t} + \rho_\alpha(\mf{sl}_2))/\mf{u}_{-\alpha}  \longrightarrow \mf{g}_\alpha \longrightarrow 0
\end{equation}
of $T U_{-\alpha}$-modules. But \eqref{eq:tbunreglocal4} is induced from the canonical non-split extension
\[ 0 \longrightarrow \alpha^\vee(\mb{Z}) \longrightarrow V \longrightarrow \mf{g}_\alpha \longrightarrow 0, \]
so injectivity of \eqref{eq:tbunreglocal3} follows.
\end{proof}

\begin{prop} \label{prop:picwpullback}
We have the following.
\begin{enumerate}[(1)]
\item \label{itm:picwpullback1} The pullback functors
\[ \piccat(Y) \longrightarrow \piccat(\tbun_{G, rig}^{ss, reg}) \quad \text{and} \quad \piccat(Y) \longrightarrow \piccat(\tbun_G^{ss, reg}) \]
are fully faithful as functors enriched over $\mc{O}_{S_{\liset}}\tn{-mod}$.
\item \label{itm:picwpullback2} The pullback functors
\begin{equation} \label{eq:picwpullbackequivalences}
\piccat^W(Y) \longrightarrow \piccat^W(\tbun_{G, rig}^{ss, reg}) \longrightarrow \piccat^W(\tbun_G^{ss, reg})
\end{equation}
are equivalences of categories enriched over $\mc{O}_{S_{\liset}}\tn{-mod}$.
\item \label{itm:picwpullback3} The equivalences \eqref{eq:picwpullbackequivalences} restrict to equivalences
\[ \piccat^W(Y)_{good} \overset{\sim}\longrightarrow \piccat^W(\tbun_{G, rig}^{ss, reg})_{good} \overset{\sim}\longrightarrow \piccat^W(\tbun_G^{ss, reg})_{good}.\]
\end{enumerate}
\end{prop}
\begin{proof}
Note that since all enriched categories in the statement satisfy flat descent on $S$, we can assume without loss of generality that $S$ is a connected regular scheme and that $E \to S$ has a section $O_E \colon S \to E$.

To prove \eqref{itm:picwpullback1}, note that by Proposition \ref{prop:regcodim} and Lemma \ref{lem:bigpicrestriction}, it suffices to prove that the enriched functors
\[ \piccat(Y) \longrightarrow \piccat(\tbun_{G, rig}^{ss}) \quad \text{and} \quad \piccat(Y) \longrightarrow \piccat(\tbun_G^{ss}) \]
are fully faithful. We will in fact show that
\begin{equation} \label{eq:picwpullback1}
\mrm{Bl}_{B, Y, rig}^{ss}{\vphantom{p}}_*\mc{O}_{\tbun_{G, rig}^{ss}} = \mrm{Bl}_{B, Y}^{ss}{\vphantom{p}}_*\mc{O}_{\tbun_G^{ss}} = \mc{O}_Y
\end{equation}
as sheaves of $\mc{O}$-modules on $Y_{\liset}$, where $\mrm{Bl}_{B, Y, rig}^{ss}$ and $\mrm{Bl}_{B, Y}^{ss}$ are the morphisms from $\tbun_{G, rig}^{ss}$ and $\tbun_G^{ss}$ to $Y$ given by composing $\mrm{Bl}_{B, rig}^{ss}$ and $\mrm{Bl}_B^{ss}$ with the coarse moduli space maps. Fully faithfulness then follows since
\[ \ul{\hom}((\mrm{Bl}_{B, Y}^{ss})^*L_1, (\mrm{Bl}_{B, Y}^{ss})^*L_2) = \pi_Y{\vphantom{p}}_* \mrm{Bl}_{B, Y}^{ss}{\vphantom{p}}_*(\mrm{Bl}_{B, Y}^{ss})^*(L_1^{-1} \otimes L_2) = \pi_Y{\vphantom{p}}_*(L_1^{-1}\otimes L_2 \otimes \mrm{Bl}_{B, Y}^{ss}{\vphantom{p}}_*\mc{O})\]
(and similarly for $\mrm{Bl}_{B, Y, rig}^{ss}$).

Note that since $Y$ is normal and the morphisms $\tbun_G^{ss} \to \tbun_{G, rig}^{ss} \to Y$ are smooth, it suffices to prove \eqref{eq:picwpullback1} on the complement of a set of codimension at least $2$. By Lemma \ref{lem:sregiso}, \eqref{eq:picwpullback1} holds on $Y^{sreg}$, so it remains to check that it holds after pulling back to $\spec \mc{O}_{Y, y}$ for every generic point $y$ of an irreducible component of $Y\setminus Y^{sreg}$.

Let $\alpha \in \Phi_+$ be the unique root such that $s_\alpha(y) = y$. Then we have canonical isomorphisms
\[ \tbun_G^{ss} \times_Y \spec \mc{O}_{Y, y} \cong \bun_{TU_{-\alpha}} \times_Y \spec \mc{O}_{Y, y} \cong \bun_{\mc{U}_{-\alpha}/\spec \mc{O}_{Y, y}}(E)/T\]
and
\[ \tbun_{G, rig}^{ss} \times_Y \spec \mc{O}_{Y, y} \cong \bun_{\mc{U}_{-\alpha}/\spec \mc{O}_{Y, y}}(E)/(T/Z(G)),\]
where $\mc{U}_{-\alpha} = \xi_{T, \mc{O}_{Y, y}} \times^T U_{-\alpha}$ for $\xi_{T, \mc{O}_{Y, y}}$ the pullback to $\spec \mc{O}_{Y, y} \times_S E$ of the universal $T$-bundle on $Y \times_S E$ trivialised along $Y \times_S \{O_E\}$.

The long exact sequence in cohomology associated to the exact sequence
\[ 0 \longrightarrow \mc{U}_{-\alpha}(-O_E) \longrightarrow \mc{U}_{-\alpha} \longrightarrow {O_E}_*O_E^*\mc{U}_{-\alpha} \longrightarrow 0 \]
shows that the morphism
\begin{equation} \label{eq:picwpullback2}
\bun_{\mc{U}_{-\alpha}(-O_E)/\spec \mc{O}_{Y, y}}(E) \longrightarrow \bun_{\mc{U}_{-\alpha}/\spec \mc{O}_{Y, y}}(E)
\end{equation}
is a torsor under the line bundle $O_E^*\mc{U}_{-\alpha}$ on $\spec \mc{O}_{Y, y}$. Since $\mb{R}^0p_E{\vphantom{p}}_*\mc{U}_{-\alpha}(-O_E) = 0$, the domain of \eqref{eq:picwpullback2} can be identified with the line bundle $\mb{R}^1p_E{\vphantom{p}}_*\mc{U}_{-\alpha}(-O_E)$ over $\spec \mc{O}_{Y, y}$ on which $T$ acts with weight $-\alpha$. (Here $p_E \colon \spec \mc{O}_{Y, y} \times_S E \to \spec \mc{O}_{Y, y}$ is the natural projection.) So the claim follows by direct computation using the \v{C}ech complex for the covering \eqref{eq:picwpullback2} to compute the pushforward of $\mc{O}$ and taking $T$ and $(T/Z(G))$-invariants.

To prove \eqref{itm:picwpullback2}, note that \eqref{itm:picwpullback1} implies that the functors
\[ \piccat^W(Y) \longrightarrow \piccat^W(\tbun_{G, rig}^{ss, reg}) \quad \text{and} \quad \piccat^W(Y) \longrightarrow \piccat^W(\tbun_G^{ss, reg}) \]
are fully faithful as enriched functors, and so it is enough to prove that they are essentially surjective.

Fix a $W$-linearised line bundle $L$ on $\tbun_G^{ss, reg}$. Restricting $L$ to $\tbun_G^{ss, sreg} = \bun_T^{0, sreg} = Y^{sreg} \times \B T$ and using the isomorphism
\[ \mrm{Pic}^W(Y^{sreg} \times \B T) \cong \mrm{Pic}^W(Y^{sreg}) \oplus \mb{X}^*(T)^W = \mrm{Pic}^W(Y^{sreg})\]
gives a $W$-linearised line bundle on $Y^{sreg}$, which can be extended (non-uniquely) to a line bundle $L_0$ on $Y$. By construction,
\[ L = ((\mrm{Bl}_{B, Y}^{reg})^*L_0)\left(\sum_i n_i \bar{D}_i\right), \]
where $n_i \in \mb{Z}$ and $\bar{D}_i \subseteq \tbun_G^{ss, reg}$ are irreducible divisors in the complement of $\tbun_G^{ss,sreg}$. From the proof of Proposition \ref{prop:regcodim}, for any irreducible divisor $D$ in $Y$ in the complement of $Y^{sreg}$, $(\mrm{Bl}_{B, Y}^{reg})^{-1}(D)$ is nonempty and irreducible. So we must have $\bar{D}_i = (\mrm{Bl}_{B, Y}^{reg})^{-1}(D_i)$ for some divisors $D_i$ on $Y$, and hence $L = (\mrm{Bl}_{B, Y}^{reg})^*(L')$, where
\[ L' = L_0\left(\sum_i n_i D_i\right). \]
Now \eqref{itm:picwpullback1} implies that the $W$-linearisation on $L$ necessarily descends to a $W$-linearisation on $L'$, so we have shown that $\piccat^W(Y) \to \piccat^W(\tbun_G^{ss, reg})$ is essentially surjective. An identical argument shows that $\piccat^W(Y) \to \piccat^W(\tbun_{G, rig}^{ss, reg})$ is essentially surjective, so this proves \eqref{itm:picwpullback2}.

To prove \eqref{itm:picwpullback3}, first note that a $W$-linearised line bundle on $\tbun_{G, rig}^{ss, reg}$ is good if and only if its pullback to $\tbun_G^{ss, reg}$ is good. So it suffices to prove that a $W$-linearised line bundle $L$ on $Y$ is good in the sense of Definition \ref{defn:picwygood} if and only if the $W$-linearised line bundle $(\mrm{Bl}_{B, Y}^{reg})^*L$ on $\tbun_G^{ss, reg}$ is good in the sense of Definition \ref{defn:ramgaloispicgood}.

If $w \in W$, then $(\tbun_G^{ss, reg})^w \subseteq (\mrm{Bl}_{B, Y}^{reg})^{-1}(Y^w)$, so either $w = s_\alpha$ for some $\alpha \in \Phi_+$ or $(\tbun_G^{ss, reg})^w \subseteq \tbun_G^{ss, reg}$ has codimension at least $2$. It is clear from this and the definitions that if $L$ is good then $(\mrm{Bl}_{B, Y}^{reg})^*L$ is good also. Conversely, suppose that $(\mrm{Bl}_{B, Y}^{reg})^*L$ is good. To show that $L$ is good, it suffices to show that for every $\alpha \in \Phi_+$ and every generic point $y$ of an irreducible component of $Y^{s_\alpha}$, the morphism
\[ s_\alpha \colon L|_{(\spec \mc{O}_{Y, y})^{s_\alpha}} \longrightarrow L|_{(\spec \mc{O}_{Y, y})^{s_\alpha}} \]
is the identity. Since $(\mrm{Bl}_{B, Y}^{reg})^*L$ is good, this holds after pulling back along the morphism
\begin{equation} \label{eq:picwpullback3}
(\tbun_G^{ss, reg} \times_Y \spec \mc{O}_{Y, y})^{s_\alpha} \longrightarrow (\spec \mc{O}_{Y, y})^{s_\alpha},
\end{equation}
where the fixed locus on the left is relative to $\bun_G$, and the fixed locus on the right is relative to $Y$. But Lemma \ref{lem:tbunreglocal} implies that
\[ (\tbun_G^{ss, reg} \times_Y \spec \mc{O}_{Y, y})^{s_\alpha} = \tbun_G^{ss, reg} \times_Y (\spec \mc{O}_{Y, y})^{s_\alpha}\]
so \eqref{eq:picwpullback3} is smooth and surjective. So the result follows by descent.
\end{proof}

\subsection{The theta bundle} \label{subsection:thetabundle}

In this subsection, we compute the group $\mrm{Pic}^W(Y)_{good} = \mrm{Pic}(\bun_G)$. The computations show that $\mrm{Pic}^W(Y)_{good}$ is generated by $\mrm{Pic}(S)$ and a single ample line bundle $\Theta_Y$. This corresponds to a line bundle $\Theta_{\bun_G}$ on $\bun_G$, which we call the \emph{theta bundle}.

\begin{defn}
Let $L$ be any line bundle on $Y$. The \emph{quadratic class of $L$} is the function
\begin{align*}
q(L)\colon \mb{X}_*(T) &\longrightarrow \mb{Z} \\
\lambda &\longmapsto \deg_{\mrm{Pic}^0(E_s)}(\lambda^*L)
\end{align*}
where the degree is taken over any geometric fibre of the relative Picard scheme $\mrm{Pic}^0_S(E) \to S$.
\end{defn}

The following lemma motivates the terminology ``quadratic class''.

\begin{lem} \label{lem:chernquadratic}
Let $L$ be a line bundle on $Y$. Then $q(L)$ is a quadratic form, i.e., $q(L)(-\lambda) = q(L)(\lambda)$ for all $\lambda \in \mb{X}_*(T)$ and the map
\begin{align*}
Q(L) \colon \mb{X}_*(T) \times \mb{X}_*(T) &\longrightarrow \mb{Q} \\
(\lambda, \mu) &\longmapsto \frac{1}{2}(q(L)(\lambda + \mu) - q(L)(\lambda) - q(L)(\mu))
\end{align*}
is symmetric and bilinear.
\end{lem}
\begin{proof}
Since $q(L)$ is computed on a geometric fibre, we can assume for simplicity that $S = \spec k$ for some algebraically closed field $k$. We have
\[ q(L)(-\lambda) = \deg_{\mrm{Pic}^0(E)} (-\lambda)^*L = \deg_{\mrm{Pic}^0(E)} [-1]^*\lambda^*L = \deg_{\mrm{Pic}^0(E)}\lambda^*L = q(L)(\lambda)\]
since $[-1]\colon \mrm{Pic}^0(E) \to \mrm{Pic}^0(E)$ has degree $1$.

For the map $Q(L)$, symmetry is obvious. Bilinearity is equivalent to $Q(L)(0, 0) = 0$, which is true by inspection, and the statement that for all $\lambda, \mu, \nu \in \mb{X}_*(T)$ the line bundle on $\mrm{Pic}^0(E)$
\[ (\lambda + \mu + \nu)^*L \otimes (\lambda + \mu)^*L^{-1} \otimes (\lambda + \nu)^*L^{-1} \otimes (\mu + \nu)^*L^{-1} \otimes \lambda^*L \otimes \mu^*L \otimes \nu^*L \]
has degree $0$. But this line bundle is trivial by the theorem of the cube \cite[\S II.6]{mumford70}, so we are done.
\end{proof}

\begin{rmk}
Let $L$ be a line bundle on $Y$. Note that Lemma \ref{lem:chernquadratic} implies that $q(L)(0) = - Q(L)(0, 0) = 0$, and hence
\[ q(L)(\lambda) = -\frac{1}{2}(q(L)(\lambda - \lambda) - q(L)(\lambda) - q(L)(-\lambda)) = -Q(L)(\lambda, -\lambda) = Q(L)(\lambda, \lambda).\]
So the datum of the function $q(L) \colon \mb{X}_*(T) \to \mb{Z}$ is equivalent to the datum of the bilinear form $Q(L) \in \hom(\sym^2(\mb{X}_*(T)), \mb{Q})$. For this reason, we will often refer to $Q(L)$ as the quadratic class of $L$.
\end{rmk}

\begin{rmk}
One might hope that the quadratic class $q(L)$ determines the first Chern class $c_1(L)$. This is not true in general: for example, it fails for the elliptic curve $E = \mb{C}/(\mb{Z} + i\mb{Z})$ over $S = \spec \mb{C}$, the group $G = SL_3$ and the line bundle $L$ constructed as follows. Identify $\mrm{Pic}^0(E)$ with $E$ and hence $Y$ with $E \times E$ via the canonical principal polarisation. Let $P$ be the line bundle on $E \times E$ defining the polarisation on $E$, and let $L = (\id, i)^*P$ be the pullback of $P$ under the automorphism $(\id, i)\colon E \times E \to E \times E$. Then $q(L) = 0$ but $c_1(L) \neq 0$.
\end{rmk}

\begin{lem}
Let $L$ be a good $W$-linearised line bundle on $Y$. Then $Q(L)$ lies in the image of
\[ \sym^2(\mb{X}^*(T))^W \subseteq \hom(\sym^2(\mb{X}_*(T)), \mb{Z}) \subseteq \hom(\sym^2(\mb{X}_*(T)), \mb{Q})\]
under the inclusion sending $\lambda \mu \in \sym^2(\mb{X}^*(T))$ to the bilinear map
\begin{align*}
\mb{X}_*(T) \times \mb{X}_*(T) &\longrightarrow \mb{Z} \\
(\lambda', \mu') &\longmapsto \langle \lambda, \lambda'\rangle \langle \mu, \mu'\rangle  + \langle \lambda, \mu' \rangle \langle \mu, \lambda' \rangle.
\end{align*}
\end{lem}
\begin{proof}
As in Lemma \ref{lem:chernquadratic}, we may assume $S = \spec k$. Since $Q(L)$ is manifestly $W$-invariant, by elementary linear algebra, it is enough to show that if $\alpha_i^\vee, \alpha_j^\vee \in \Delta^\vee$ are simple coroots, then $Q(L)(\alpha_i^\vee, \alpha_i^\vee) \in 2\mb{Z}$ and $Q(L)(\alpha_i^\vee, \alpha_j^\vee) \in \mb{Z}$.

If $\alpha_i^\vee \in \Delta^\vee$, then Lemma \ref{lem:goodcriterion} below implies that
\[ Q(L)(\alpha_i^\vee, \alpha_i^\vee) = \deg_{\mrm{Pic}^0(E)} (\alpha_i^\vee)^*L \in 2\mb{Z}.\]
If $\alpha_j^\vee \in \Delta^\vee$ is another simple coroot, then invariance of $Q(L)$ under $s_i \in W$ implies that
\[ Q(L)(\alpha_i^\vee, \alpha_j^\vee) = \frac{Q(L)(\alpha_i^\vee, \alpha_i^\vee)}{2} \langle \alpha_i, \alpha_j^\vee \rangle \in \mb{Z},\]
so we are done.
\end{proof}

In the following lemma, we write $O_{\mrm{Pic}^0_S(E)}$ for the origin in $\mrm{Pic}^0_S(E)$, in order to distinguish it from the zero divisor $0$. We will also write $\pi_Y \colon Y \to S$ for the structure morphism and $O_Y \colon S \to Y$ for the section corresponding to the trivial $T$-bundle.

\begin{lem} \label{lem:goodcriterion}
Let $L$ be a $W$-linearised line bundle on $Y$ such that the action of $W$ on $O_Y^*L$ is trivial. Then the $W$-linearisation on $L$ is good if and only if for all simple coroots $\alpha_i^\vee$, $(\alpha_i^\vee)^*L \cong \pi_{\mrm{Pic}^0_S(E)}^*L' (2d \cdot O_{\mrm{Pic}^0_S(E)})$ for some $d \in \mb{Z}$ and some line bundle $L'$ on $S$.
\end{lem}
\begin{proof}
Since every reflection in $W$ is conjugate to a simple reflection, the line bundle $L$ is good if and only if
\[ s_i \colon L|_{Y^{s_i}} \longrightarrow L|_{Y^{s_i}} \]
is the identity for all simple reflections $s_i$.

Fix a simple reflection $s_i$. Observe that since $\mb{X}_*(T)^{s_i}_\mb{Q} + \mb{Q}\alpha_i^\vee = \mb{X}_*(T)_\mb{Q}$, the morphism
\[ f\colon \mrm{Pic}^0_S(E) \otimes_\mb{Z}\mb{X}_*(T)^{s_i} \times_S \mrm{Pic}^0_S(E) = \mrm{Pic}^0_S(E) \otimes_\mb{Z}(\mb{X}_*(T)^{s_i} + \mb{Z}\alpha_i^\vee) \longrightarrow \mrm{Pic}^0_S(E) \otimes_{\mb{Z}} \mb{X}_*(T) = Y \]
is an isogeny of abelian varieties over $S$, and $f \circ s_i = r \circ f$, where
\[ r = (\id, [-1])\colon \mrm{Pic}^0_S(E) \otimes_{\mb{Z}} \mb{X}_*(T)^{s_i} \times_S \mrm{Pic}^0_S(E) \longrightarrow \mrm{Pic}^0_S(E) \otimes_{\mb{Z}} \mb{X}_*(T)^{s_i} \times_S \mrm{Pic}^0_S(E).\]
Since $r$ acts trivially on $\ker(f)$, it follows that
\[ f^{-1}(Y^{s_i}) = \mrm{Pic}^0_S(E) \otimes_{\mb{Z}} \mb{X}_*(T)^{s_i} \times_S \mrm{Pic}^0_S(E)[2] \]
and so the action of $s_i$ on $L|_{Y^{s_i}}$ is trivial if and only if the action of $r$ on
\[ f^*L|_{\mrm{Pic}^0_S(E) \otimes \mb{X}_*(T)^{s_i} \times_S \mrm{Pic}^0_S(E)[2]}\]
is trivial. Since the action of $r$ is given by a global regular function on $\mrm{Pic}^0_S(E) \otimes \mb{X}_*(T)^{s_i} \times_S \mrm{Pic}^0_S(E)[2]$, which is necessarily pulled back from a regular function on $\mrm{Pic}^0_S(E)[2]$, it suffices to check that $r$ acts as the identity on the fibre of
\[ \mrm{Pic}^0_S(E) \otimes_\mb{Z} \mb{X}_*(T)^{s_i} \times_S \mrm{Pic}^0_S(E)[2] \longrightarrow \mrm{Pic}^0_S(E) \otimes_\mb{Z} \mb{X}_*(T)^{s_i} \]
over any section of the structure map to $S$. Taking the fibre over the natural origin, the restriction of $f$ here is $\alpha_i^\vee \colon \mrm{Pic}^0_S(E) \to Y$. So by Lemma \ref{lem:gooda1} below, $r$ acts as the identity if and only if $(\alpha_i^\vee)^*L = \pi_{\mrm{Pic}^0_S(E)}^*L'(2d \cdot O_{\mrm{Pic}^0_S(E)})$, which proves the lemma.
\end{proof}

\begin{lem} \label{lem:gooda1}
Let $L$ be a line bundle on $\mrm{Pic}_S^0(E)$ with $[-1]^*L \cong L$, and let $\sigma \colon [-1]^*L \to L$ be the unique isomorphism acting as the identity on the pullback of $L$ along the section $O_{\mrm{Pic}_S^0(E)} \colon S \to \mrm{Pic}_S^0(E)$ corresponding to the trivial line bundle on $E$. Then $\sigma$ acts as the identity on $L|_{\mrm{Pic}_S^0(E)[2]}$ if and only if $L = \pi^*L' \otimes \mc{O}(2d\cdot O_{\mrm{Pic}_S^0(E)})$ for some $d \in \mb{Z}$ and some line bundle $L'$ on $S$, where $\pi\colon \mrm{Pic}^0_S(E) \to S$ is the structure morphism.
\end{lem}
\begin{proof}
The morphism $\sigma$ acts as the identity on $L|_{\mrm{Pic}^0_S(E)[2]}$ if and only if $L = f^*L''$ for some line bundle $L''$ on the $\mb{P}^1$-bundle $\mb{P} = \mb{P}_S(\pi_*\mc{O}(2 \cdot O_{\mrm{Pic}^0_S(E)}))^\vee$ over $S$, where $f$ is the canonical morphism. But every line bundle $L''$ on $\mb{P}$ is of the form $L'' = \pi_\mb{P}^*L'(d \cdot f(O_{\mrm{Pic}^0_S(E)}))$ for some $d \in \mb{Z}$ and some line bundle $L'$ on $S$, so $f^*L'' = \pi^*L' (2d \cdot O_{\mrm{Pic}_S^0(E)})$ and we are done.
\end{proof}

\begin{lem} \label{lem:symdualgensandrels}
Let $M$ be a finitely generated free abelian group. Then the abelian group $\sym^2(M)^\vee = \hom_\mb{Z}(\sym^2(M), \mb{Z})$ is generated by
\begin{align*}
\lambda^2\colon \sym^2(M) &\longrightarrow \mb{Z} \\
mn &\longmapsto \lambda(m)\lambda(n)
\end{align*}
for $\lambda \in M^\vee$, with relations
\[ \lambda^2 = (-\lambda)^2 \quad \text{and} \quad (\lambda + \mu + \nu)^2 - (\lambda + \mu)^2 -(\lambda + \nu)^2 - (\mu + \nu)^2 + \lambda^2 + \mu^2 + \nu^2 = 0 \]
for $\lambda, \mu, \nu \in M^\vee$.
\end{lem}
\begin{proof}
The proof is an exercise in elementary linear algebra. The details can be found in \cite[Lemma 4.4.8]{davis19a}.
\end{proof}

\begin{prop} \label{prop:picwygoodcomputation}
The homomorphism
\begin{equation} \label{eq:goodc1}
 (Q, O_Y^*) \colon \mrm{Pic}^W(Y)_{good} \longrightarrow \sym^2(\mb{X}^*(T))^W \oplus \mrm{Pic}(S)
\end{equation}
is an isomorphism of abelian groups.
\end{prop}
\begin{proof}
We show that \eqref{eq:goodc1} is both injective and surjective.

For injectivity, suppose that $L$ is a good $W$-linearised line bundle on $Y$ such that $Q(L) = 0$ and $O_Y^*L \cong \mc{O}_S$. To show that $L$ is trivial, it is enough to show that $L$ is trivial on every geometric fibre of $Y \to S$, since this implies by Grauert's Theorem that $L$ is pulled back from a line bundle on $S$. So we can assume for this part of the proof that $S = \spec k$ for some algebraically closed field $k$.

We first claim that for any two simple coroots $\alpha_i^\vee$ and $\alpha_j^\vee$, the pullback $L' = (\alpha_i^\vee, \alpha_j^\vee)^*L$ of $L$ under
\[ (\alpha_i^\vee, \alpha_j^\vee)\colon \mrm{Pic}^0(E) \times \mrm{Pic}^0(E) \longrightarrow Y \]
is trivial. To see this, it suffices to show that $(\alpha_j^\vee)^*L = \mc{O}$, and that $L'$ is trivial restricted to every $k$-fibre of the second projection $\mrm{Pic}^0(E) \times \mrm{Pic}^0(E) \to \mrm{Pic}^0(E)$. To see the first condition, apply Lemma \ref{lem:gooda1} to the morphism
\[ [-1]^*(\alpha_j^\vee)^*L = (\alpha_j^\vee)^*s_j^*L \overset{\sim}\longrightarrow (\alpha_j^\vee)^*L,\]
and use the fact that $(\alpha_j^\vee)^*L$ has degree $q(L)(\alpha_j^\vee) = 0$. For the second, let $x_2 \in \mrm{Pic}^0(E)$ be a $k$-point, and consider the restriction $L'_{x_2}$ of $L'$ to $\mrm{Pic}^0(E) \times \{x_2\}$. Define $\sigma\colon \mrm{Pic}^0(E) \to \mrm{Pic}^0(E)$ by
\[ \sigma(x_1) = -x_1 - \langle \alpha_i, \alpha_j^\vee \rangle x_2.\]
Then the diagram
\[
\begin{tikzcd}
\mrm{Pic}^0(E) \ar[r, "\sigma"] \ar[d, "i_{x_2}"] & \mrm{Pic}^0(E) \ar[d, "i_{x_2}"] \\
Y \ar[r, "s_i"] & Y
\end{tikzcd}
\]
commutes, where $i_{x_2}$ is given by $i_{x_2}(x_1) = \alpha_i^\vee(x_1) + \alpha_j^\vee(x_2)$. So the isomorphism $s_i^*L \to L$ gives an isomorphism $\sigma^*L'_{x_2} \to L'_{x_2}$ acting as the identity on $\mrm{Pic}^0(E)^\sigma$. But since $k$ is algebraically closed, $\langle \alpha_i, \alpha_j^\vee \rangle x_2$ has a square root in $E(k)$, so $\sigma$ is a conjugate of $[-1]$ by a translation. So we can apply Lemma \ref{lem:gooda1} to conclude that $L'_{x_2} = \mc{O}$. So $L'$ is trivial as claimed.

To complete the proof of injectivity, we need to show that in fact $L = \mc{O}_Y$. We prove by induction on $n \in \mb{Z}_{>0}$ that for all $i_1, \ldots, i_n \in \{1, \ldots, l\}$, the line bundle
\[ L' = (\alpha_{i_1}^\vee, \ldots, \alpha_{i_n}^\vee)^*L \]
on $\mrm{Pic}^0(E)^n$ is trivial. We have shown this for $n = 1$ or $2$. For $n > 2$, we write $\mrm{Pic}^0(E)^n = \mrm{Pic}^0(E)^{n - 2} \times \mrm{Pic}^0(E) \times \mrm{Pic}^0(E)$, and observe that by induction, the restrictions of $L'$ to $\mrm{Pic}^0(E)^{n - 2} \times \mrm{Pic}^0(E) \times \{O_{\mrm{Pic}^0(E)}\}$, $\mrm{Pic}^0(E)^{n - 2} \times \{O_{\mrm{Pic}^0(E)}\} \times \mrm{Pic}^0(E)$ and $\{O_{\mrm{Pic}^0(E)}\}^n \times \mrm{Pic}^0(E) \times \mrm{Pic}^0(E)$ are all trivial. So $L'$ is trivial as claimed by the theorem of the cube. Setting $n = l$ and $\{i_1, \ldots, i_n\} = \{1, \ldots, l\}$, we conclude that $L$ is trivial, and hence that \eqref{eq:goodc1} is injective.

To prove surjectivity, we first claim that there is a homomorphism
\[ \phi \colon \sym^2(\mb{X}_*(T))^\vee \longrightarrow \mrm{Pic}(Y) \]
sending $\lambda^2 \in \sym^2(\mb{X}_*(T))^\vee$ to the line bundle
\[ \phi(\lambda^2) = \lambda^*\mc{O}(O_{\mrm{Pic}^0_S(E)}) \otimes \pi_Y^*O_{\mrm{Pic}^0_S(E)}^*\mc{O}(-O_{\mrm{Pic}^0_S(E)})\]
for $\lambda \in \mb{X}^*(T)$. By Lemma \ref{lem:symdualgensandrels}, it suffices to check that for all $\lambda, \mu, \nu \in \mb{X}^*(T)$, we have
\[ \phi(\lambda^2) = \phi((-\lambda)^2)\]
and
\[ \phi((\lambda + \mu + \nu)^2) \otimes \phi((\lambda + \mu)^2)^{-1} \otimes \phi((\lambda + \nu)^2)^{-1} \otimes \phi((\mu + \nu)^2)^{-1}\otimes \phi(\lambda^2) \otimes \phi(\mu^2) \otimes \phi(\nu^2) = \mc{O}_Y. \]
Since it is clear that $O_Y^*\phi(\lambda^2) = \mc{O}_S$ for every $\lambda \in \mb{X}^*(T)$, it suffices to check these relations on every geometric fibre over $S$. The first holds since $[-1]^*\mc{O}(O_{\mrm{Pic}^0_S(E)}) = \mc{O}(O_{\mrm{Pic}^0_S(E)})$ and the second holds by the theorem of the cube, so the homomorphism $\phi$ is indeed well-defined. By construction, it is also clear that $Q(\phi(P)) = P$ for all $P \in \sym^2(\mb{X}_*(T))^\vee$.

Assume now that $P \in \sym^2(\mb{X}^*(T))^W \subseteq \sym^2(\mb{X}_*(T))^\vee$ and $L \in \mrm{Pic}(S)$. We need to find a good $W$-linearised line bundle $L'$ on $Y$ such that $Q(L') = P$ and $O_Y^*L' = L$. Note that since the homomorphism $\phi$ is $W$-equivariant by construction, the line bundle $\phi(P)$ is $W$-invariant, so $w^*\phi(P) \cong \phi(P)$ for all $w \in W$. We can turn this into a $W$-linearisation by taking
\[w^*\colon w^*\phi(P) \overset{\sim} \longrightarrow \phi(P) \]
to be the unique isomorphism acting as the identity on $O_Y^*\phi(P)$. We let $L' = \phi(P) \otimes \pi_Y^*L$. It is clear that $Q(L') = P$ and $O_Y^*L' = L$, so it remains to show that $\phi(P)$, and hence $L'$, is good. By Lemma \ref{lem:goodcriterion}, it suffices to show that for every simple coroot $\alpha_i^\vee$ and every geometric point $s \colon \spec k \to S$, we have
\[ (\alpha_i^\vee)^*\phi(P)|_{\mrm{Pic}^0(E_s)} = \mc{O}(2d \cdot O_{\mrm{Pic}^0(E_s)}) \]
for some $d \in \mb{Z}$. But by construction, it is clear that
\[ (\alpha_i^\vee)^*\phi(P)|_{\mrm{Pic}^0(E_s)} = \mc{O}(P(\alpha_i^\vee, \alpha_i^\vee)\cdot O_{\mrm{Pic}^0(E_s)}) \]
and $P(\alpha_i^\vee, \alpha_i^\vee) \in 2\mb{Z}$ since $P \in \sym^2(\mb{X}^*(T)) \subseteq \sym^2(\mb{X}_*(T))^\vee$. So $\phi(P)$ is good, and hence \eqref{eq:goodc1} is surjective as claimed.
\end{proof}

Proposition \ref{prop:picwygoodcomputation} allows us to compute the Picard group $\mrm{Pic}(\bun_G)$. In the statement below, we write $(\, |\,) \in \sym^2(\mb{X}^*(T))^W$ for the normalised Killing form. This is the unique $W$-invariant symmetric bilinear form on $\mb{X}_*(T)$ satisfying $(\alpha^\vee | \alpha^\vee) = 2$ for $\alpha^\vee$ a short coroot.

\begin{cor} \label{cor:picbung}
The Picard group of $\bun_G$ is
\[ \mrm{Pic}(\bun_G) = \mb{Z}[\Theta_{\bun_G}] \oplus \mrm{Pic}(S), \]
where $\Theta_{\bun_G}$ is the unique line bundle satisfying
\begin{equation} \label{eq:picbungthetacomparison}
\psi^*(\Theta_{\bun_G})|_{\tbun_G^{ss}} = \mrm{Bl}_{B, Y}^*\Theta_Y|_{\tbun_G^{ss}} 
\end{equation}
where $\Theta_Y$ is the unique good $W$-linearised line bundle on $Y$ with $Q(\Theta_Y) = (\, |\, )$ and $O_Y^*\Theta_Y = \mc{O}_S$. Moreover, there is an isomorphism of graded $\mc{O}_S$-algebras
\[ \bigoplus_{d \geq 0} {\pi_{\bun_G}}_* \Theta_{\bun_G}^{\otimes d} = \bigoplus_{d \geq 0} ({\pi_Y}_*\Theta_Y^{\otimes d})^W.\]
\end{cor}
\begin{proof}
It is an elementary and well-known observation that since $G$ is simply connected and simple, we have $\sym^2(\mb{X}^*(T))^W = \mb{Z}(\,|\,)$. So Proposition \ref{prop:picwygoodcomputation} gives
\[ \mrm{Pic}^W(Y)_{good} = \mb{Z}[\Theta_Y] \oplus \mrm{Pic}(S),\]
and hence Theorem \ref{thm:bungchevalley} gives
\[ \mrm{Pic}(\bun_G) = \mb{Z}[\Theta_{\bun_G}] \oplus \mrm{Pic}(S),\]
where $\Theta_{\bun_G}$ is the image of $\Theta_Y$ under the Chevalley isomorphism. The construction of the Chevalley isomorphism shows that \eqref{eq:picbungthetacomparison} is satisfied. Moreover, if $L$ is any other line bundle on $\bun_G$ satisfying $\psi^*L|_{\tbun_G^{ss}} = \mrm{Bl}_{B, Y}^*\Theta_Y|_{\tbun_G^{ss}}$, then writing $L$ as the image of a good $W$-linearised line bundle $L_Y$ under the Chevalley isomorphism, we must have $L_Y \cong \Theta_Y$ as line bundles on $Y$ by Proposition \ref{prop:picwpullback}, and hence as $W$-linearised ones. So \eqref{eq:picbungthetacomparison} indeed characterises $\Theta_{\bun_G}$. Finally, the isomorphism of graded algebras follows immediately from the fact that the Chevalley isomorphism is an equivalence of enriched symmetric monoidal categories.
\end{proof}

\begin{rmk}
Of course, Corollary \ref{cor:picbung} also holds with $\bun_{G, rig}$ in place of $\bun_G$, with the same proof. We will denote the corresponding generator of $\mrm{Pic}(\bun_{G, rig})$ by $\Theta_{\bun_{G, rig}}$.
\end{rmk}

\begin{rmk}
When $S = \spec k$, Corollary \ref{cor:picbung} shows that $\mrm{Pic}(\bun_G) \cong \mb{Z}$, recovering a special case of a theorem of Y. Laszlo and C. Sorger \cite{laszlo-sorger97}. We remark that Laszlo and Sorger's proof is very different to ours: it uses the uniformisation of $\bun_G$ by an affine Grassmannian, rather than our method using the relation to the abelian variety $Y$.
\end{rmk}

We conclude by remarking on the ampleness of $\Theta_Y$.

\begin{prop}
The line bundle $\Theta_Y$ is ample relative to $S$.
\end{prop}
\begin{proof}
Since $Y \to S$ is proper and flat, it suffices to prove that $\Theta_Y$ is ample on every geometric fibre. So we can assume for the proof that $S = \spec k$ for $k$ an algebraically closed field.

Since $Y$ is an abelian variety, $Y$ is projective over $k$, so there exists some ample line bundle $L$ on $Y$. The ample line bundle $L' = \bigotimes_{w \in W} w^*L$ is canonically $W$-linearised, and it is easy to see that the $W$-linearisation on $L'' = (L')^{\otimes 2}$ is therefore good. So $L''$ is a good $W$-linearised ample line bundle on $Y$, and therefore a positive multiple of $\Theta_Y$ by Proposition \ref{prop:picwygoodcomputation}. Hence $\Theta_Y$ is ample as claimed.
\end{proof}

\subsection{The fundamental diagram} \label{subsection:fundiag}

In this subsection, we apply Corollary \ref{cor:picbung} to construct the elliptic Grothendieck-Springer resolution as a commutative diagram
\begin{equation} \label{eq:fundiagbody}
\begin{tikzcd}
\tbun_G \ar[r, "\psi"] \ar[d, "\tilde{\chi}"'] & \bun_G \ar[d, "\chi"] \\
\Theta_Y^{-1}/\mb{G}_m \ar[r] & (\hat{Y}\sslash W)/\mb{G}_m.
\end{tikzcd}
\end{equation}
We also list some easy properties of $\chi$ and $\tilde{\chi}$, and give an explicit formula for the divisor $\tilde{\chi}^{-1}(0_{\Theta_Y^{-1}})$.

By Corollary \ref{cor:picbung}, there is a tautological $\mb{G}_m$-equivariant morphism
\begin{equation} \label{eq:precoarsequotient}
\Theta_{\bun_G}^{-1} \longrightarrow \spec_S \bigoplus_{d \geq 0} {\pi_{\bun_G}}_* \Theta_{\bun_G}^{\otimes d} \cong \spec_S \bigoplus_{d \geq 0} ({\pi_Y}_* \Theta_Y^{\otimes d})^W = \hat{Y}\sslash W,
\end{equation}
where $\hat{Y}$ is the cone over $Y$ given by contracting the zero section of $\Theta_Y^{-1}$ to $S$. Deleting the zero section of $\Theta_{\bun_G}^{-1}$ and taking the quotient of both sides of \eqref{eq:precoarsequotient} by $\mb{G}_m$ therefore gives a morphism
\[ \chi \colon \bun_G \longrightarrow (\hat{Y} \sslash W)/\mb{G}_m.\]

\begin{defn}
The morphism $\chi$ above is called the \emph{coarse quotient map} for $\bun_G$.
\end{defn}

We next construct the morphism $\tilde{\chi}$ by lifting the blow down morphism $\tbun_G \to Y$. Note that the datum of such a morphism is equivalent to the data of a line bundle $L$ on $\tbun_G$ equipped with a morphism
\begin{equation} \label{eq:chitilde1}
L \longrightarrow \mrm{Bl}_{B, Y}^*\Theta_Y^{-1}.
\end{equation}
The morphism $\tilde{\chi}$ is obtained from \eqref{eq:chitilde1} by deleting the zero section of $L$ and taking the quotient by $\mb{G}_m$ of the induced morphism $L \to \Theta_Y^{-1}$ on total spaces. The diagram \eqref{eq:fundiagbody} commute if and only if $L = \psi^*\Theta_{\bun_G}^{-1}$, and the diagram
\begin{equation} \label{eq:chitilde2}
\begin{tikzcd}
\pi_{\tbun_G}{\vphantom{p}}_*\psi^*\Theta_{\bun_G}^{\otimes d} & \pi_{\bun_G}{\vphantom{p}}_*\Theta_{\bun_G}^{\otimes d} \ar[l, "\psi^*"] \\
\pi_Y{\vphantom{p}}_*\Theta_Y^{\otimes d} \ar[u, "\eqref{eq:chitilde1}^*"] & (\pi_Y{\vphantom{p}}_*\Theta_Y^{\otimes d})^W \ar[u, "\chi^*"] \ar[l]
\end{tikzcd}
\end{equation}
commutes for all $d \geq 0$.

By construction of the elliptic Chevalley isomorphism, if $L_Y$ is any good $W$-linearised line bundle on $Y$ and $L_{\bun_G}$ is its image under the Chevalley isomorphism, then there is a functorial isomorphism
\[ \gamma^{ss} \colon \psi^*L_{\bun_G}|_{\tbun_G^{ss}} \overset{\sim}\longrightarrow \mrm{Bl}_{B, Y}^*L_Y|_{\tbun_G^{ss}}, \]
compatible with the isomorphisms ${\pi_{\bun_G}}_*L_{\bun_G} \cong ({\pi_Y}_* L_Y)^W$. We therefore have a \emph{rational} map of line bundles
\[ \gamma \colon \psi^*L_{\bun_G} \longdashrightarrow \mrm{Bl}_{B, Y}^*L_Y \]
such that the diagram analogous to \eqref{eq:chitilde2} commutes. We prove at the end of this subsection (Corollary \ref{cor:chitildedivisor}) that when $L_Y = \Theta_Y^{-1}$, the rational map $\gamma$ has divisor of zeroes and poles
\[ \sum_{\lambda \in \mb{X}_*(T)_+} \frac{1}{2}(\lambda \mmid \lambda) D_\lambda.\]
In particular, $\gamma$ is in fact a morphism of line bundles, and so defines the desired morphism
\[ \tilde{\chi} \colon \tbun_G \longrightarrow \Theta_Y^{-1}/\mb{G}_m.\]
We have proved the following.

\begin{cor} \label{cor:fundiag}
The morphisms $\chi$ and $\tilde{\chi}$ constructed above fit into a commutative diagram \eqref{eq:fundiagbody}.
\end{cor}

\begin{rmk}
Since the elliptic Chevalley isomorphism also holds for the rigidified stack $\bun_{G, rig}$, the same construction with $\bun_{G, rig}$ in place of $\bun_G$ gives a ridified version
\begin{equation} \label{eq:rigidifiedgrothendieckspringer}
\begin{tikzcd}
\tbun_{G, rig} \ar[r, "\psi"] \ar[d, "\tilde{\chi}"'] & \bun_{G, rig} \ar[d, "\chi"] \\
\Theta_Y^{-1}/\mb{G}_m \ar[r] & (\hat{Y}\sslash W)/\mb{G}_m,
\end{tikzcd}
\end{equation}
of the diagram \eqref{eq:fundiagbody}.
\end{rmk}

We remark that the coarse quotient map gives a GIT-style characterisation of semistable and unstable $G$-bundles.

\begin{prop} \label{prop:chifibreunstable}
Let $\xi_G \to E_s$ be a $G$-bundle on a geometric fibre of $E \to S$. Then $\xi_G$ is unstable if and only if $\chi(\xi_G) \in (\hat{Y}_s\sslash W)/\mb{G}_m$ is equal to the image of the cone point of $\hat{Y}_s$.
\end{prop}
\begin{proof}
Write $0_{\hat{Y}} \subseteq \hat{Y}$ for the family of cone points over $S$ and $q \colon \Theta_Y^{-1} \to \hat{Y}$ for the tautological morphism. Since $\Theta_Y$ is ample, we have $q^{-1}(\hat{Y} \setminus 0_{\hat{Y}}) = \Theta_Y^{-1} \setminus 0_{\Theta_Y^{-1}}$, and by Corollary \ref{cor:chitildedivisor}, we have $\tilde{\chi}^{-1}(\Theta_Y^{-1} \setminus 0_{\Theta_Y^{-1}}) = \tbun_G^{ss}$. So
\[ \chi^{-1}(((\hat{Y} \setminus 0_{\hat{Y}})\sslash W)/\mb{G}_m) = \psi(\tilde{\chi}^{-1}((\Theta_Y^{-1} \setminus 0_{\Theta_Y^{-1}})/\mb{G}_m)) = \psi(\tbun_G^{ss}) = \bun_G^{ss},\]
which proves the proposition.
\end{proof}

\begin{prop} \label{prop:chiflatness}
The morphism $\tilde{\chi} \colon \tbun_G \to \Theta_Y^{-1}/\mb{G}_m$ is flat, and all fibres of the morphism $\chi \colon \bun_G \to (\hat{Y}\sslash W)/\mb{G}_m$ have dimension $-l = \dim (\bun_G) - \dim ((\hat{Y}\sslash W)/\mb{G}_m)$.
\end{prop}
\begin{proof}
To prove that $\tilde{\chi}$ is flat, note that since $\gamma^{ss} \colon \psi^*\Theta_{\bun_G}^{-1}|_{\bun_G^{ss}} \to \mrm{Bl}_{B, Y}^*\Theta_Y^{-1}|_{\bun_G^{ss}}$ is an isomorphism and $\mrm{Bl}_{B, Y}$ is flat, we can apply \cite[Lemma 4.5.7]{davis19a} to conclude that the morphism on total spaces $\psi^*\Theta_{\bun_G}^{-1} \to \Theta_Y^{-1}$ is flat, and hence so is $\tilde{\chi}$.

It remains to prove that the fibres of $\chi$ have dimension $-l$. Fix a geometric point $x \colon \spec k \to (\hat{Y}\sslash W)/\mb{G}_m$ over $s \colon \spec k \to S$ and consider the fibre $\chi^{-1}(x)$. We know by straightforward comparison of dimensions that $\dim \chi^{-1}(x) \geq -l$, so it suffices to show that $\dim \chi^{-1}(x) \leq - l$. Since the morphism $p \colon \Theta_Y^{-1}/\mb{G}_m \to (\hat{Y}\sslash W)/\mb{G}_m$ is finite away from the zero section and $\tilde{\chi}$ is flat, if $x$ is not the image of the cone point $0_{\hat{Y}_s}$, then
\[ \dim \chi^{-1}(x) \leq \dim \psi^{-1}\chi^{-1}(x) = \dim \tilde{\chi}^{-1}p^{-1}(x) = -l\]
so we are done. On the other hand, if $x$ is the image of the cone point, then $\chi^{-1}(x)$ is a $\mb{G}_m$-torsor over the locus of unstable $G$-bundles on $E_s$ by Proposition \ref{prop:chifibreunstable}. So
\[ \dim \chi^{-1}(x) = - \codim_{\bun_G(E_s)} \bun_G^{\textit{unstable}}(E_s) + 1 = - l \]
by Proposition \ref{prop:unstablecodim}.
\end{proof}

\begin{rmk}
In characteristic $0$, it is a theorem of Looijenga \cite[Theorem 3.4]{looijenga76} that $\hat{Y}\sslash W$ is an affine space bundle over $S$, and in particular regular; we will deduce this in arbitrary characteristic (Corollary \ref{cor:looijengaiso}) as a consequence of the Friedman-Morgan section theorem. Together with Proposition \ref{prop:chiflatness} and \cite[Theorem 18.16]{eisenbud91}, this imples that the coarse quotient map $\chi$ is flat.
\end{rmk}

The rest of this subsection is devoted to proving the following theorem, and hence Corollary \ref{cor:chitildedivisor}, which was used above to construct the morphism $\tilde{\chi}$.

\begin{thm} \label{thm:chitildedivisor}
Let $L_{\bun_G}$ be a line bundle on $\bun_G$ and let $L_Y$ be the corresponding good $W$-linearised line bundle on $Y$. Then the rational map of line bundles $\gamma\colon \psi^*L_{\bun_G} \dashrightarrow \mrm{Bl}_{B, Y}^*L_Y$ induced by the isomorphism
\[ \gamma^{ss}\colon \psi^*(L_{\bun_G})|_{\tbun_G^{ss}} \overset{\sim}\longrightarrow \mrm{Bl}_{B, Y}^*L_Y|_{\tbun_G^{ss}} \]
has divisor of zeroes and poles
\[ - \sum_{\lambda \in \mb{X}_*(T)_+} \frac{1}{2} Q(L_Y)(\lambda, \lambda) D_{\lambda}. \]
\end{thm}

\begin{cor}\label{cor:chitildedivisor}
When $L_{\bun_G} = \Theta_{\bun_G}^{-1}$, the map $\gamma$ of Theorem \ref{thm:chitildedivisor} is regular, and has divisor of zeroes
\[ \mrm{div}(\gamma) = \tilde{\chi}^{-1}(0_{\Theta_Y^{-1}}) = \sum_{\lambda \in \mb{X}_*(T)_+} \frac{1}{2}(\lambda \mmid \lambda) D_\lambda.\]
\end{cor}

The key idea behind the proof of Theorem \ref{thm:chitildedivisor} is to reduce to the case of a specific choice of line bundle $L_{\bun_G}$ that we can describe explicitly as the determinant of a perfect complex. The following lemma plays an important role. 

\begin{lem} \label{lem:divisordet}
Let $X$ be an algebraic stack, let $D \subseteq X$ be an effective Cartier divisor on $X$, let $U = X \setminus D$, and let $i\colon D \hookrightarrow X$ denote the inclusion. If $\mc{E}$ is a perfect complex on $D$, then the rational map of line bundles
\[g\colon \mc{O}_X = \det 0 \longdashrightarrow \det \mb{R}i_*\mc{E}\]
induced by the quasi-isomorphism $0|_U \cong \mb{R}i_*\mc{E}|_U$ induces an isomorphism
\[ \mc{O}_X(\chi(\mc{E})D) \overset{\sim}\longrightarrow \det \mb{R}i_*\mc{E},\]
where $\chi$ denotes the Euler characteristic of a perfect complex.
\end{lem}
\begin{proof}
We need to show that the divisor of zeros and poles of $g$ is $\chi(\mc{E})D$. Since both are local on $X$ and additive in $\mc{E}$ under exact triangles, it suffices to check this when $\mc{E} = \mc{O}_D$. In this case, an explicit free resolution for $\mb{R}i_*\mc{E}$ is
\[ \mb{R}i_*\mc{E} = [\mc{O}_X(-D) \longrightarrow \mc{O}_X],\]
which gives an isomorphism
\[ \mc{O}_X \otimes \mc{O}_X(-D)^\vee \cong \det \mb{R}i_* \mc{E}.\]
The map $g$ is given by
\[ \mc{O}_X \overset{\sim}\longrightarrow \mc{O}_X(-D) \otimes \mc{O}_X(-D)^\vee \longrightarrow \mc{O}_X \otimes \mc{O}_X(-D)^\vee = \det \mb{R}i_*\mc{E},\]
which does indeed have divisor $D = \chi(\mc{E})D$ as claimed.
\end{proof}

\begin{lem} \label{lem:blowupdetdivisor}
Let $X$ be a smooth stack over $S$, let $Z \subseteq X \times_S E$ be a smooth substack of codimension $2$ mapping isomorphically to a divisor $D \subseteq X$ under the first projection $\mrm{pr}_X\colon X \times_S E \to X$, and let $f\colon C \to X \times_S E$ be the blowup along $Z$. If $L$ is a line bundle on $C$ such that $L$ restricted to any exceptional fibre of $f$ has degree $d$, then the canonical rational map
\begin{equation}\label{eq:blowupdetdivisor}
\det \mb{R}{\mrm{pr}_X}_* \mb{R}f_* L \longdashrightarrow \det \mb{R}{\mrm{pr}_X}_*\det\mb{R}f_* L
\end{equation}
induced by the quasi-isomorphism $\mb{R}f_*L|_{X \times_S E \setminus Z} \cong \det\mb{R}f_*L|_{X \times_S E \setminus Z}$ has divisor
\[ \frac{d(d + 1)}{2} D.\]
\end{lem}
\begin{proof}
We first observe that if $d = 0$, then the claim is true since $\mb{R}f_*L$ is a line bundle and hence \eqref{eq:blowupdetdivisor} is an isomorphism.

For a general line bundle $L$, write $\mrm{div}(L)$ for the divisor of \eqref{eq:blowupdetdivisor}. Consider the exact sequence
\[ 0 \longrightarrow L(-\mrm{Exc}) \longrightarrow L \longrightarrow L|_{\mrm{Exc}} \longrightarrow 0,\]
where $\mrm{Exc} = f^{-1}(Z)$ is the exceptional divisor. The morphism $\mb{R}f_*L(-\mrm{Exc}) \to \mb{R}f_*L$ induces a commutative diagram
\[
\begin{tikzcd}
\det \mb{R}{\mrm{pr}_X}_*\mb{R}f_*L(-\mrm{Exc}) \ar[d, dashed] \ar[r, dashed, "g"] & \det \mb{R}{\mrm{pr}_X}_*\mb{R}f_*L \ar[d, dashed] \\
\det \mb{R}{\mrm{pr}_X}_*\det \mb{R}f_*L(-\mrm{Exc}) \ar[r, "\sim"] & \det \mb{R}{\mrm{pr}_X}_*\det \mb{R}f_*L
\end{tikzcd}
\]
of rational maps of line bundles, where the bottom arrow is an isomorphism since $Z$ has codimension $2$. We deduce that
\[ \mrm{div}(L(-\mrm{Exc})) = \mrm{div}(L) + \mrm{div}(g).\]
But $\mrm{div}(g) = \mrm{div}(g')$, where
\[ g' \colon \mc{O}_X \longdashrightarrow \det (\mb{R}{\mrm{pr}_X}_*\mb{R}f_*L(-\mrm{Exc}))^{-1} \otimes \det \mb{R}{\mrm{pr}_X}_*\mb{R}f_*L = \det \mb{R}{\mrm{pr}_X}_*\mb{R}f_*(L|_\mrm{Exc})\]
is the rational map induced by the quasi-isomorphism $\mb{R}{\mrm{pr}_X}_*\mb{R}f_*(L|_\mrm{Exc})|_{X \setminus D} \simeq 0$. But $\mb{R}{\mrm{pr}_X}_*\mb{R}f_*(L|_\mrm{Exc})$ is the pushforward from $D$ of a perfect complex with Euler characteristic $d + 1$, so Lemma \ref{lem:divisordet} gives $\mrm{div}(g') = (d+ 1)D$ and hence
\[ \mrm{div}(L(-\mrm{Exc})) = \mrm{div}(L) + (d + 1)D.\]
Since $L(-\mrm{Exc})$ has degree $d + 1$ restricted to any exceptional fibre of $f$, the lemma now follows easily by induction on the absolute value of $d$.
\end{proof}

The next lemma identifies the $W$-linearised line bundle on $Y$ and quadratic form corresponding to a determinant line bundle on $\bun_G$.

\begin{lem} \label{lem:detquadraticclass}
Let $V$ be a representation of $G$, and let $L_{\bun_G} = \det \mb{R}{\mrm{pr}_{\bun_G}}{\vphantom{p}}_*(\xi_G^{uni} \times^G V)$, where $\xi_G^{uni} \to \bun_G \times_S E$ is the universal $G$-bundle, and $\mrm{pr}_{\bun_G} \colon \bun_G \times_S E \to \bun_G$ is the canonical projection. Then the corresponding good $W$-linearised line bundle on $Y$ is given by
\[ L_Y = \bigotimes_{\lambda \in \mb{X}^*(T)} \lambda^*\mc{O}(-O_{\mrm{Pic}^0_S(E)})^{\otimes \dim V_\lambda},\]
and hence
\[ q(L_Y)(\mu) = -\sum_{\lambda \in \mb{X}^*(T)} \dim V_\lambda \langle \lambda, \mu \rangle^2 \]
for $\mu \in \mb{X}_*(T)$, where
\[ V = \bigoplus_{\lambda \in \mb{X}^*(T)} V_\lambda \]
is the weight space decomposition of $V$ under the action of $T$.
\end{lem}
\begin{proof}

We have
\[ \psi^*L_{\bun_G} = \det \mb{R}{\mrm{pr}_{\tbun_G}}_* ((\psi^*\xi_G^{uni}) \times^G V) = \det \mb{R}{\mrm{pr}_{\tbun_G}}_* \mb{R}f_* f^*((\psi^*\xi_G^{uni}) \times^G V) \]
where $\mrm{pr}_{\tbun_G}\colon \tbun_G \times_S E \to \tbun_G$ is the first projection and $f \colon \tbun_G \times_{\Deg_S(E)} \mc{C} \to \tbun_G \times_S E$ is the pullback of the universal prestable degeneration of $E$ over $\tbun_G$. Since the $G$-linearised vector bundle $V \otimes \mc{O}_F$ on the flag variety $F = G/B$ has a $G$-equivariant filtration with associated quotients isomorphic to $V_\lambda \otimes \mc{L}_\lambda$ for $\lambda \in \mb{X}^*(T)$, we get an isomorphism
\begin{align*}
\psi^*L_{\bun_G} &\cong \bigotimes_{\lambda \in \mb{X}^*(T)} \det \mb{R}{\mrm{pr}_{\tbun_G}}{\vphantom{p}}_* \mb{R}f_* (V_\lambda \otimes(\xi_{T, \mc{C}} \times^T \mb{Z}_\lambda)) \\
& = \bigotimes_{\lambda \in \mb{X}^*(T)} (\det \mb{R}{\mrm{pr}_{\tbun_G}}{\vphantom{p}}_*\mb{R}f_*(\xi_{T, \mc{C}} \times^T \mb{Z}_\lambda))^{\otimes \dim V_{\lambda}},
\end{align*}
where $\xi_{T, \mc{C}}$ is the induced $T$-bundle on $\tbun_G \times_{\Deg_S(E)} \mc{C}$. So restricting to $\tbun_G^{ss}$ gives
\begin{equation} \label{eq:detquadraticclass1}
\psi^*L_{\bun_G}|_{\tbun_G^{ss}} = \bigotimes_{\lambda \in \mb{X}^*(T)} \mrm{Bl}_B^*(\det \mb{R}{\mrm{pr}_{\bun_T^0}}{\vphantom{p}}_*(\xi^{uni}_T \times^T \mb{Z}_{\lambda}))^{\otimes \dim V_\lambda}|_{\tbun_G^{ss}},
\end{equation}
where $\xi_T^{uni}$ is the universal $T$-bundle on $\bun_T^0 \times_S E$ and $\mrm{Bl}_B \colon \tbun_G \to \bun_T^0$ is the blow down morphism. Now, for all $\lambda \in \mb{X}^*(T)$, we have
\[ \det \mb{R}{\mrm{pr}_{\bun_T^0}}{\vphantom{p}}_*(\xi^{uni}_T \times^T \mb{Z}_\lambda) = \lambda^*\det \mb{R}{\mrm{pr}_{\bun_{\mb{G}_m}^0}}{\vphantom{p}}_*M,\]
where $M$ is the universal line bundle on $\bun_{\mb{G}_m}^0 \times_S E$. But
\[ \mb{R}{\mrm{pr}_{\bun_{\mb{G}_m}^0}}{\vphantom{p}}_*M = (\mb{R}^1{\mrm{pr}_{\bun_{\mb{G}_m}^0}}{\vphantom{p}}_*M)[-1] \]
is the pushforward of a perfect complex of Euler characteristic $-1$ on the Cartier divisor $O_{\bun_{\mb{G}_m}^0}$ corresponding to the trivial bundle, so by Lemma \ref{lem:divisordet}, we have
\[ \det \mb{R}{\mrm{pr}_{\bun_{\mb{G}_m}^0}}{\vphantom{p}}_*M = \mc{O}(-O_{\bun_{\mb{G}_m}^0}) = q^*\mc{O}(-O_{\mrm{Pic}_S^0(E)}) \]
where $q \colon \bun_{\mb{G}_m}^0 \to \mrm{Pic}^0_S(E)$ is the canonical quotient by $\B\mb{G}_m$. So \eqref{eq:detquadraticclass1} gives
\[ \psi^*L_{\bun_G}|_{\tbun_G^{ss}} = \mrm{Bl}_{B, Y}^*\bigotimes_{\lambda \in \mb{X}^*(T)} \lambda^*\mc{O}(-O_{\mrm{Pic}^0_S(E)})^{\otimes \dim V_\lambda}|_{\tbun_G^{ss}} \]
from which the result follows immediately.
\end{proof}

\begin{proof}[Proof of Theorem \ref{thm:chitildedivisor}]
We first remark that since the truth or falsehood of the statement is unchanged if we raise $L_{\bun_G}$ to a nonzero power or tensor with a line bundle on $S$, by Corollary \ref{cor:picbung}, it suffices to prove the claim for any single $L_{\bun_G}$ with $q(L_Y) \neq 0$. So choose any nontrivial representation $V$ of $G$ and set
\[ L_{\bun_G} = \det \mb{R}{\mrm{pr}_{\bun_G}}_*(\xi_G^{uni} \times^G V),\]
as in Lemma \ref{lem:detquadraticclass}. Keeping the notation from that proof, since ${\pi_{\tbun_G^{ss}}}{\vphantom{p}}_*\mc{O} = {\pi_Y}{\vphantom{p}}_*\mc{O}_Y = \mc{O}_S$ by Proposition \ref{prop:picwpullback} \eqref{itm:picwpullback1}, the rational map $\gamma$ must coincide up to rescaling by a nonvanishing function on $S$ with the rational map
\begin{equation} \label{eq:chitildedivisorcomp2}
\bigotimes_{\lambda \in \mb{X}^*(T)} (\det \mb{R}{\mrm{pr}_{\tbun_G}}{\vphantom{p}}_*\mb{R}f_*(\xi_{T, \mc{C}} \times^T \mb{Z}_\lambda))^{\otimes \dim V_\lambda} \dashrightarrow \bigotimes_{\lambda \in \mb{X}^*(T)} (\det \mb{R}{\mrm{pr}_{\tbun_G}}{\vphantom{p}}_*\det \mb{R}f_*(\xi_{T, \mc{C}} \times^T \mb{Z}_\lambda))^{\otimes \dim V_\lambda}
\end{equation}
given by the quasi-isomorphisms
\[\det \mb{R}f_*(\xi_{T, \mc{C}} \times^T \mb{Z}_\lambda)|_{\tbun_G^{ss}} \cong \mb{R}f_*(\xi_{T, \mc{C}} \times^T \mb{Z}_\lambda)|_{\tbun_G^{ss}},\]
for $\lambda \in \mb{X}^*(T)$, where we recall that $\mrm{Bl}_B^*(\xi_T^{uni} \times^T \mb{Z}_\lambda) = \det \mb{R}f_* (\xi_{T, \mc{C}} \times^T \mb{Z}_\lambda)$ by Lemma \ref{lem:tbundleblowdowndet} and the definition of $\mrm{Bl}_B \colon \tbun_G \to \bun_T^0$. To complete the proof, observe that 
\[ (\tbun_G)^{\leq 1} = \bigcup_{\mu \in \mb{X}_*(T)_+} (\tbun_G)^{\leq 1}_\mu \]
where
\[ (\tbun_G)^{\leq 1} = \tbun_G \times_{\Deg_S(E)} \Deg_S(E)^{\leq 1} \subseteq \tbun_G \]
is the open substack of points where the nodal domain curve $C$ has at most $1$ node, and for $\mu \in \mb{X}_*(T)_+$, $(\tbun_G)^{\leq 1}_{\mu} \subseteq (\tbun_G)^{\leq 1}$ is the (open) union of $\tbun_G^{ss} = \bun_B^0$ and the smooth locus $D_\mu^\circ \subseteq D_\mu$. Since the complement of $(\tbun_G)^{\leq 1}$ in $\tbun_G$ has codimension $2$, it suffices to show that for any $\mu \in \mb{X}_*(T)_+$, the restriction of \eqref{eq:chitildedivisorcomp2} to $(\tbun_G)^{\leq 1}_\mu$ has divisor
\[ \frac{1}{2}\sum_{\lambda \in \mb{X}^*(T)} \dim V_\lambda \langle \lambda, \mu \rangle^2 D_\mu.\]
By Lemma \ref{lem:blowupdetdivisor} and Proposition \ref{prop:basicdegproperties}, the map
\[ \det \mb{R}{\mrm{pr}_{(\tbun_G)^{\leq 1}_\mu}}{\vphantom{p}}_* \mb{R}f_*(\xi_{T, \mc{C}} \times^T \mb{Z}_{\lambda}) \longdashrightarrow \det \mb{R}{\mrm{pr}_{(\tbun_G)^{\leq 1}_\mu}}{\vphantom{p}}_* \det \mb{R}f_*(\xi_{T, \mc{C}} \times^T \mb{Z}_{\lambda}) \]
has divisor
\[ \frac{\langle \lambda, \mu \rangle(\langle \lambda, \mu \rangle + 1)}{2} D_\mu,\]
so taking the tensor product over all $\lambda \in \mb{X}^*(T)$, we find that the restriction of \eqref{eq:chitildedivisorcomp2} has divisor
\[ \left( \sum_{\lambda \in \mb{X}^*(T)} \frac{\dim V_\lambda}{2} \langle \lambda, \mu \rangle^2 + \sum_{\lambda \in \mb{X}^*(T)} \frac{\dim V_\lambda}{2} \langle \lambda, \mu \rangle\right) D_\mu =  \frac{1}{2}\sum_{\lambda \in \mb{X}^*(T)}\dim V_\lambda \langle \lambda, \mu \rangle^2 D_\mu\]
as required, since
\[ \sum_{\lambda \in \mb{X}^*(T)} (\dim V_\lambda) \lambda \in \mb{X}^*(T)^W = \{0\}.\]
\end{proof}

\section{The Friedman-Morgan section theorem} \label{section:friedmanmorgan}

We now turn to the problem of proving that the elliptic Grothendieck-Springer resolution \eqref{eq:fundiagbody} is a simultaneous log resolution. As advertised in the introduction, we will deduce this from an elliptic analogue of the Kostant and Steinberg section theorems.

The classical section theorems are proved by constructing subvarieties $Z$ in $\mf{g}$ or $G$ that are transverse to all $G$-orbits and meet the set of regular nilpotent or unipotent elements in a single point, and showing that these map isomorphically to the adjoint quotient. In stacky language, these subvarieties give smooth charts $Z \to \mf{g}/G$ or $G/G$ that are miniversal deformations of the points corresponding to the regular nilpotent or unipotent orbits.

In the elliptic context, the role of regular nilpotent or unipotent elements is played by the regular unstable bundles (Definition \ref{defn:regularunstable}). However, since translations on the elliptic curve provide an extra symmetry, we will work not with charts of $\bun_G$, but with the following slightly weaker objects.

\begin{defn} \label{defn:slice}
Assume that $E \to S$ has a section $O_E \colon S \to E$. A \emph{slice} of $\bun_G$ (resp., $\bun_{G, rig}$) is a morphism $Z \to \bun_G$ (resp., $Z \to \bun_{G, rig}$) of stacks over $S$, such that the composition $Z \to \bun_G \to \bun_G/E$ (resp., $Z \to \bun_{G, rig} \to \bun_{G, rig}/E$) is smooth.
\end{defn}

\begin{eg}
To illustrate the difference between slices and charts, consider the following example. Let $G = SL_2$, and consider the space $Z = \mrm{Ext}^1(\mc{O}(O_E), \mc{O}(-O_E)) \cong \mb{A}^2$ of extensions
\[ 0 \longrightarrow \mc{O}(-O_E) \longrightarrow V \longrightarrow \mc{O}(O_E) \longrightarrow 0.\]
There is a tautological map $Z \to \bun_{SL_2}$, which is easily seen to be a slice. It is not a chart, however, as it fails to be smooth at the origin: the unstable bundle $\mc{O}(-O_E) \oplus \mc{O}(O_E)$, which is in the image, has a $1$-parameter family of deformations $\{\mc{O}(-p) \oplus \mc{O}(p)\}_{p \in E}$, which are only in the image modulo translations.
\end{eg}

The outline of the section is as follows. In \S\ref{subsection:parabolicinduction} we review a parabolic induction construction for slices of $\bun_{G, rig}$ (the idea of which is due to Friedman and Morgan \cite{friedman-morgan00}), and show how in many cases it leads to the extra structure of an \emph{equivariant slice} (Definition \ref{defn:equivariantslice}). In \S\ref{subsection:regularunstable}, we recall the definition of regular unstable $G$-bundles, and show how to use Atiyah's classification of stable vector bundles $E$ together with parabolic induction to construct equivariant slices through them. In \S\ref{subsection:friedmanmorgan}, we state and prove the Friedman-Morgan section theorem. Finally, in \S\ref{subsection:friedmanmorganapplications}, we give the proof of the simultaneous log resolution property for the elliptic Grothendieck-Springer resolution, which completes the proof of Theorem \ref{thm:introgrothendieckspringer}.

In this section, the rigidified stack $\bun_{G, rig}$ plays a more central role than the ordinary stack of $G$-bundles $\bun_G$. The reason for this is that the theory of slices (especially of equivariant slices constructed by parabolic induction) is somewhat better behaved for $\bun_{G, rig}$ than for $\bun_G$ (see for example Remarks \ref{rmk:zlrigaction} and \ref{rmk:regularslicesmoothness}).

Throughout this section, we will assume unless otherwise specified that the elliptic curve $E \to S$ has a given section $O_E \colon S \to E$ and endow $E$ with its natural group scheme structure over $S$ for which $O_E$ is the identity.

\subsection{Parabolic induction and equivariant slices} \label{subsection:parabolicinduction}

If $L \subseteq G$ is a Levi subgroup and $Z_0 \to \bun^{ss}_{L, rig}$ is a slice (i.e., a morphism such that $Z_0 \to \bun^{ss}_{L, rig}/E$ is smooth), then there is a simple procedure due to Friedman and Morgan for constructing an induced slice $Z = \mrm{Ind}_L^G(Z_0) \to \bun_{G, rig}$. In this subsection, we review the definition and properties of this construction. We also show how in many cases, it produces slices $Z \to \bun_{G, rig}$ such that the morphism $Z \to (\hat{Y}\sslash W)/\mb{G}_m$ lifts to an equivariant morphism $Z \to \hat{Y}\sslash W$.

\begin{defn} \label{defn:parabolicinduction}
Let $L \subseteq G$ be a Levi subgroup, let $\mu \in \mb{X}_*(Z(L)^\circ)_\mb{Q}$ and let $P \subseteq G$ be the unique parabolic subgroup with Levi factor $L$ such that $-\mu$ is a Harder-Narasimhan vector for $P$ in the sense of Definition \ref{defn:hardernarasimhan}. If $Z_0 \to \bun_{L, rig}^{ss, \mu}$ is a morphism of stacks, the \emph{parabolic induction of $Z_0$ to $G$} is the morphism
\[ \mrm{Ind}_L^{G}(Z) = \bun_{P, rig} \times_{\bun_{L, rig}} Z \longrightarrow \bun_{G, rig}.
\]
\end{defn}

\begin{rmk}
In the situation of Definition \ref{defn:parabolicinduction}, if the morphism $Z_0 \to \bun_{L, rig}$ factors through $\bun_L$, then we have an isomorphism $\mrm{Ind}_L^G(Z_0) \cong \bun_P \times_{\bun_L} Z_0$, and hence a factorisation of $\mrm{Ind}_L^G(Z_0) \to \bun_{G, rig}$ as $\mrm{Ind}_L^G(Z_0) \to \bun_G \to \bun_{G, rig}$.
\end{rmk}

In the following proposition, we do not assume that $E \to S$ has a section.

\begin{prop} \label{prop:parabolicinductionsmooth}
Assume that the morphism $Z_0 \to \bun^{ss, \mu}_{L, rig}$ is smooth. Then so is the morphism $\mrm{Ind}_L^G(Z_0) \to \bun_{G, rig}$.
\end{prop}
\begin{proof}
Since $Z_0 \to \bun^{ss, \mu}_{L, rig}$ is smooth, so is the morphism $\mrm{Ind}_L^G(Z) \to \bun_{P, rig}^{ss, \mu}$. So it suffices to show that the morphism $\bun_{P, rig}^{ss, \mu} \to \bun_{G, rig}$ is smooth. By flat descent for smoothness, this is equivalent to showing that $\bun_P^{ss, \mu} \to \bun_G$ is smooth. The relative tangent complex is
\[ \mb{T} = \mb{R}{\mrm{pr}_{\bun_P^{ss, \mu}}}{\vphantom{p}}_*(\xi_P^{uni} \times^P \mf{g}/\mf{p}),\]
where $\mf{g} = \mrm{Lie}(G)$, $\mf{p} = \mrm{Lie}(P)$ and $\xi_P^{uni} \to \bun_P^{ss, \mu} \times_S E$ is the universal $P$-bundle. But since $-\mu$ is a Harder-Narasimhan vector for $P$, the vector bundle $\xi_P^{uni} \times^P \mf{g}/\mf{p}$ has a filtration whose successive quotients are semistable of positive slope on every fibre of $\bun_P^{ss, \mu} \times_S E \to \bun_P^{ss, \mu}$. So $\mb{T}$ is a vector bundle concentrated in degree $0$ (e.g., by \cite[Lemma 2.6.3]{davis19a}), which proves the claim.
\end{proof}

\begin{cor} \label{cor:inducedslice}
Assume that $Z_0 \to \bun_{L, rig}^{ss, \mu}$ is a slice. Then $\mrm{Ind}_L^G(Z_0) \to \bun_{G, rig}$ is a slice.
\end{cor}
\begin{proof}
Apply Proposition \ref{prop:parabolicinductionsmooth} to the family $E' := S \to \B_S E =: S'$, observing that $\bun_{G, rig}/E = \bun_{G/S'}(E')_{rig}$.
\end{proof}

A key feature of the parabolic induction construction is the existence of a natural action of the group $Z(L)_{rig} = Z(L)/Z(G)$ on $\mrm{Ind}_L^G(Z_0)$. Note that unlike $Z(L)$, $Z(L)_{rig}$ is always a torus, with cocharacter lattice
\[ \bigoplus_{\alpha_i \in t(P)} \mb{Z}\varpi_i^\vee \]
if $L$ is the Levi factor of a standard parabolic subgroup $P$.

In general, suppose that $X$ is a stack equipped with an action $a \colon X \times \B H \to X$ of the classifying stack of some commutative group scheme $H$. Then for any morphism of stacks $\pi \colon X' \to X$, there is a canonical action of $H$ on $X'$ over $X$ fitting into a commutative diagram
\begin{equation} \label{eq:canonicalaction}
\begin{tikzcd}
X' \ar[r] \ar[d] & X'/H \ar[r] \ar[d] & X' \ar[d] \\
X \ar[r] & X \times \B H \ar[r, "a"] & X,
\end{tikzcd}
\end{equation}
in which both squares are Cartesian, where the morphism $X \to X \times \B H$ is the quotient by the trivial action of $H$ on $X$. Explicitly, this action can be realised by using the outer square of \eqref{eq:canonicalaction} to identify $X'$ with the stack of tuples $(x' \in X', x \in X, \phi \colon x \overset{\sim}\to \pi(x'))$ and setting $(x', x, \phi)\cdot h = (x', x, \phi \circ a(h))$ for $h \in H$.

Now suppose we are in the situation of Definition \ref{defn:parabolicinduction}. Applying the above construction to the action of $\B Z(L)_{rig}$ on $\bun_{L, rig}$ inherited from the action of $\B Z(L)$ on $\bun_L$ gives an action of $Z(L)_{rig}$ on $\bun_{P, rig}$ over $\bun_{L, rig}$ and a morphism $\bun_{P, rig}/Z(L)_{rig} \to \bun_{P, rig} \to \bun_{G, rig}$. Pulling back along $Z_0 \to \bun_{L, rig}$, we get an action of $Z(L)_{rig}$ on $\mrm{Ind}_L^G(Z_0)$ over $Z_0$ and a morphism $\mrm{Ind}_L^G(Z)/Z(L)_{rig} \to \bun_{G, rig}$.

\begin{rmk} \label{rmk:zlrigaction}
If the morphism $Z_0 \to \bun_{L, rig}^{ss, \mu}$ factors through $\bun_L^{ss, \mu}$, then the $\B Z(L)$-action on $\bun_L$ gives a morphism $\mrm{Ind}_L^G(Z_0)/Z(L) \to \bun_G$ and a Cartesian diagram
\[
\begin{tikzcd}
\mrm{Ind}_L^G(Z_0)/Z(L) \ar[r] \ar[d] & \bun_G \ar[d] \\
\mrm{Ind}_L^G(Z_0)/Z(L)_{rig} \ar[r] & \bun_{G, rig}.
\end{tikzcd}
\]
Note that $\mrm{Ind}_L^G(Z_0)/Z(L) \to \bun_G$ factors through $\mrm{Ind}_L^G(Z_0)/Z(L)_{rig}$ if and only if the left hand morphism above admits a section, which holds if and only if $Z(L) = Z(G) \times Z(L)_{rig}$.
\end{rmk}

The following proposition describes the structure of the natural morphism $\mrm{Ind}_L^G(Z_0) \to Z_0$ together with the $Z(L)_{rig}$-action constructed above.

\begin{prop}[{(cf.\ \cite[Corollary 4.2.4]{friedman-morgan00})}] \label{prop:inductionaffinebundle}
Assume we are in the setup of Definition \ref{defn:parabolicinduction}. If $Z_0$ is an affine scheme, then there exists a (non-canonical) $Z(L)_{rig}$-equivariant isomorphism of stacks over $Z_0$,
\[ \mrm{Ind}_L^G(Z_0) \cong \mb{R}^1{\mrm{pr}_{Z_0}}_* (\xi_{L/Z(G)} \times^{L/Z(G)} \mf{u}),\]
where $\xi_{L/Z(G)} \to Z_0 \times_S E$ is the $L/Z(G)$-bundle classified by the morphism $Z_0 \to \bun_{L, rig} \to \bun_{L/Z(G)}$, $\mf{u}$ is the Lie algebra of the unipotent radical $R_u(P)$, and $Z(L)_{rig}$ acts on $\mf{u}$ by right conjugation. Hence, for any $Z_0$ (not necessarily affine), the morphism $\mrm{Ind}_L^G(Z_0) \to Z_0$ is always an affine space bundle with fibrewise linear $Z(L)_{rig}$-action.
\end{prop}
\begin{proof}
Let $\xi_{L/Z(G)}^{uni} \to \bun_{L, rig} \times_S E$ be the $L/Z(G)$-bundle classified by the morphism $\bun_{L, rig} \to \bun_{L/Z(G)}$, and let $\mc{U}^{uni} = \xi_{L/Z(G)}^{uni} \times^{L/Z(G)} R_u(P)$. It follows directly from the construction that the $\B Z(L)_{rig}$-action induces the right conjugation action $(x, u) \cdot g = (x, g^{-1}ug)$ of $Z(L)_{rig}$ on the group scheme $\mc{U}^{uni}$. Letting $\mc{U} = \xi_{L/Z(G)} \times^{L/Z(G)} R_u(P)$, it follows that the action of $Z(L)_{rig}$ on $\mrm{Ind}_L^G(Z_0) = \bun_{\mc{U}/Z_0}$ is also given by right conjugation.

Let $0 < \mu_1 < \cdots < \mu_n$ be the possible positive values of $\langle \alpha, -\mu \rangle$ for $\alpha \in \Phi$, and let
\[ \{1\} = U_{n + 1} \subseteq U_n \subseteq \cdots \subseteq U_1 = R_u(P)\]
be the filtration defined by
\[ U_i = \prod_{\substack{\alpha \in \Phi \\ \langle \alpha,-\mu \rangle \geq \mu_i}} U_\alpha\]
for $1 \leq i \leq n + 1$.
Letting
\[ \mc{U}_i = \xi_{L/Z(G)} \times^{L/Z(G)} U_i \subseteq \mc{U},\]
we show by induction on $i$ that each $\bun_{(\mc{U}/\mc{U}_i)/Z_0}(E)$ is $Z(L)_{rig}$-equivariantly isomorphic to the vector bundle
\[ \mb{R}^1{\mrm{pr}_{Z_0}}{\vphantom{p}}_* (\xi_{L/Z(G)} \times^{L/Z(G)} \mf{u}/\mf{u}_i) = \bigoplus_{j \leq i} \mb{R}^1{\mrm{pr}_{Z_0}}{\vphantom{p}}_*(\xi_{L/Z(G)} \times^{L/Z(G)} \mf{u}_{j - 1}/\mf{u}_{j}),\]
where $\mf{u}_i$ is the Lie algebra of $R_u(P)_i$.

For $i = 1$, the claim is trivial. For $i > 1$, the long exact sequence in nonabelian cohomology (e.g, \cite[Proposition 2.4.2]{davis19a}) implies that each morphism
\[ \bun_{(\mc{U}/\mc{U}_i)/Z_0}(E) \longrightarrow \bun_{(\mc{U}/\mc{U}_{i - 1})/Z_0}(E) \]
is a $Z(L)_{rig}$-equivariant $\bun_{(\mc{U}_{i - 1}/\mc{U}_i)/Z_0}(E)$-torsor. But $U_{i - 1}/U_i \cong \mf{u}_{i - 1}/\mf{u}_i$ as $L/Z(G)$-equivariant group schemes. So
\[ \bun_{(\mc{U}_{i - 1}/\mc{U}_i)/Z_0}(E) \cong \bun_{(\xi_{L/Z(G)}\times^{L/Z(G)} \mf{u}_{i - 1}/\mf{u}_i)/Z_0}(E) = \mb{R}^1{\mrm{pr}_{Z_0}}{\vphantom{p}}_*(\xi_{L/Z(G)} \times^{L/Z(G)} \mf{u}_{i - 1}/\mf{u}_i) \]
since
\[ H^0(E_s, \xi_{L/Z(G), z} \times^{L/Z(G)} \mf{u}_{i - 1}/\mf{u}_i) = 0\]
for any geometric point $z \colon \spec k \to Z$ over $s \colon \spec k \to S$, as this is the space of global sections of an iterated extension of semistable vector bundles of negative degree. By induction, $V_i = \bun_{(\mc{U}/\mc{U}_{i - 1})/Z_0}(E)$ is a vector bundle on $Z_0$ with linear $Z(L)_{rig}$-action, so the $Z(L)_{rig}$-equivariant torsors on it are classified by the group
\begin{align*}
H^1(V_i/Z(L)_{rig}&, \,\mb{R}^1{\mrm{pr}_{Z_0}}{\vphantom{p}}_*(\xi_{L/Z(L)} \times^{L/Z(L)} \mf{u}_{i - 1}/\mf{u}_i)) \\
&= H^1(\mb{B}Z(L)_{rig}, H^0(V_i, \mc{O}_{V_i}) \otimes_{H^0(Z_0, \mc{O}_{Z_0})} H^1(E, \xi_{L/Z(L)} \times^{L/Z(L)} \mf{u}_{i - 1}/\mf{u}_i)) \\
&= 0
\end{align*}
since $Z(L)_{rig}$ is a torus and $Z_0$ is affine. So we can trivialise the given torsor $Z(L)_{rig}$-equivariantly, to give a $Z(L)_{rig}$-equivariant isomorphism
\begin{align*}
\bun_{(\mc{U}/\mc{U}_i)/Z_0}(E) &\cong \bun_{(\mc{U}/\mc{U}_{i - 1})/Z_0}(E) \times_{Z_0} \mb{R}^1{\mrm{pr}_Z}{\vphantom{p}}_*(\xi_{L/Z(G)} \times^{L/Z(G)} \mf{u}_{i - 1}/\mf{u}_i) \\
&\cong \bigoplus_{j \leq i} \mb{R}^1{\mrm{pr}_{Z_0}}_*(\xi_{L/Z(G)} \times^{L/Z(G)} \mf{u}_{j - 1}/\mf{u}_j),
\end{align*}
as claimed.
\end{proof}

The next two propositions give root-theoretic formulas for the $Z(L)_{rig}$-weights and dimension of the affine space bundle $\mrm{Ind}_L^G(Z_0) \to Z$.

\begin{prop} \label{prop:inductionweights}
If $\lambda \in \mb{X}^*(Z(L)_{rig})$, then the multiplicity of the weight $\lambda$ in a fibre of $\mrm{Ind}_L^G(Z_0) \to Z_0$ is $d_{-\lambda} \langle \lambda, \mu \rangle$, if $\langle \lambda, \mu \rangle > 0$ and $0$ otherwise, where $d_{-\lambda}$ is the number of $\alpha \in \Phi$ such that $\alpha|_{Z(L)_{rig}} = -\lambda$.
\end{prop}
\begin{proof}
By Proposition \ref{prop:inductionaffinebundle}, the multiplicity of $\lambda$ in $\mrm{Ind}_L^G(Z_0)$ is equal to the multiplicity in $H^1(E_s, \xi_L \times^L \mf{u})$, where $E_s$ is any geometric fibre of $E \to S$, $\xi_L \to E_s$ is a semistable $L$-bundle of slope $\mu$, and $Z(L)_{rig}$ acts by right conjugation on $\mf{u}$. But this is equal to the dimension of $H^1(E_s, \xi_L \times^L \mf{u}_{-\lambda})$, where
\[ \mf{u}_{-\lambda} = \begin{cases}{\displaystyle \bigoplus_{\substack{\alpha \in \Phi \\ \alpha|_{Z(L)_{rig}} = -\lambda}}} \mf{g}_\alpha, & \text{if}\;\; \langle \lambda, \mu \rangle > 0, \\[2ex] 0, &\text{otherwise}, \end{cases} \]
is the $\lambda$-weight space of $Z(L)_{rig}$ acting on $\mf{u}$ by right conjugation. But since $\xi_L \times^L \mf{u}_{-\lambda}$ is either $0$ or a semistable vector bundle of negative slope $-\langle \lambda, \mu \rangle$, it follows that
\[ \dim H^1(\xi_L \times^L \mf{u}_{-\lambda}) = -\deg(\xi_L \times^L \mf{u}_{-\lambda}) = \begin{cases} d_{-\lambda} \langle \lambda, \mu \rangle, & \text{if}\;\; \langle \lambda, \mu \rangle > 0, \\ 0, & \text{otherwise},\end{cases}\]
which proves the claim.
\end{proof}

\begin{prop} \label{prop:inductionreldim}
The morphism $\mrm{Ind}_L^G(Z_0) \to Z_0$ has relative dimension $\langle 2\rho_P, \mu \rangle$, where $-2\rho_P$ is the sum of all roots $\alpha \in \Phi$ such that $U_\alpha \subseteq R_u(P)$.
\end{prop}
\begin{proof}
This follows from Proposition \ref{prop:inductionweights} after taking the sum over all $\lambda$.
\end{proof}

Our key original observation about parabolic induction is that, in many cases, the $Z(L)_{rig}$-action on $\mrm{Ind}_L^G(Z_0)$ can be promoted to the following structure.

\begin{defn} \label{defn:equivariantslice}
Let $H$ be a torus, and let $\lambda \in \mb{X}^*(H)$ be a nonzero character. An \emph{equivariant slice of $\bun_{G, rig}$ with equivariance group $H$ and weight $\lambda$} is a commutative diagram
\[
\begin{tikzcd}
Z/H \ar[r] \ar[d] & \bun_{G,rig} \ar[d, "\Theta_{\bun_{G, rig}}^{-1}"] \\
\B H \ar[r, "\lambda"] & \B \mb{G}_m
\end{tikzcd}
\]
where $Z$ is a stack with $H$-action over $S$, such that the composition $Z \to Z/H \to \bun_{G, rig} \to \bun_{G, rig}/E$ is smooth. We will often suppress the group $H$ from the notation and refer to $Z \to \bun_{G, rig}$, or even simply $Z$, as an equivariant slice.
\end{defn}

\begin{rmk} \label{rmk:equivariantthetasection}
Unpacking the stacky formalism, the datum of an equivariant slice $Z/H \to \bun_{G, rig}$ of weight $\lambda$ is equivalent to the datum of an $H$-equivariant morphism
\begin{equation} \label{eq:equivariantthetasection}
Z \longrightarrow (\Theta_{\bun_{G, rig}}^{-1})^*,
\end{equation}
where $H$ acts on the complement $(\Theta_{\bun_{G, rig}}^{-1})^*$ of the zero section of $\Theta_{\bun_{G, rig}}^{-1}$ through the character $\lambda \colon H \to \mb{G}_m$. Together with the rigidified elliptic Grothendieck-Springer resolution \eqref{eq:rigidifiedgrothendieckspringer}, this gives an $H$-equivariant commutative diagram
\[
\begin{tikzcd}
\tilde{Z} \ar[r, "\psi_Z"] \ar[d, "\tilde{\chi}_Z"'] & Z \ar[d, "\chi_Z"] \\
\Theta_Y^{-1} \ar[r] & \hat{Y}\sslash W,
\end{tikzcd}
\]
where $\tilde{Z} = Z \times_{\bun_{G, rig}} \tbun_{G, rig}$ and $H$ acts on $\Theta_Y^{-1}$ and $\hat{Y}\sslash W$ via the character $\lambda$.
\end{rmk}

\begin{defn}
Let $L \subseteq G$ be a Levi subgroup and $\mu \in \mb{X}_*(Z(L)^\circ)_\mb{Q}$. A \emph{$\Theta$-trivial slice of $\bun_{L, rig}^{ss, \mu}$} is a slice $Z_0 \to \bun_{L, rig}^{ss, \mu}$ equipped with a trivialisation of the pullback $\Theta_{\bun_{L, rig}^{ss, \mu}}$ of the theta bundle $\Theta_{\bun_{G, rig}}$.
\end{defn}

\begin{prop} \label{prop:inducedequivariantslice}
Let $Z_0 \to \bun_{L, rig}^{ss, \mu}$ be a $\Theta$-trivial slice. Then there is a natural equivariant slice structure on $\mrm{Ind}_L^G(Z_0) \to \bun_{G, rig}$ with equivariance group $Z(L)_{rig}$ and weight $(\mu \mmid -)$.
\end{prop}

The main ingredient in the proof of Proposition \ref{prop:inducedequivariantslice} is a computation of the action of $Z(L)_{rig}$ on the pullback of the theta bundle to $Z_0$. Before we give this computation, it will be useful to introduce the following terminology.

\begin{defn}
Let $X$ be a connected stack equipped with an action $a \colon X \times \B H \to X$ of the classifying stack of a commutative group scheme $H$. If $\mc{L}$ is a line bundle on $X$, then \emph{weight of $\mc{L}$} is the image of $\mc{L} \in \mrm{Pic}(X)$ under the homomorphism
\[ \mrm{Pic}(X) \overset{a^*}\longrightarrow \mrm{Pic}(X \times \B H) \cong \mrm{Pic}(X) \oplus \mb{X}^*(H) \longrightarrow \mb{X}^*(H),\]
where the isomorphism above is given by
\begin{align*}
\mrm{Pic}(X) \oplus \mb{X}^*(H) &\longrightarrow \mrm{Pic}(X \times \B H) \\
(\mc{L}, \lambda) &\longmapsto p^*\mc{L} \otimes q^*(\eta_H \times^H \mb{Z}_\lambda)
\end{align*}
for $\eta_H \to \B H$ the universal $H$-torsor and $p \colon X \times \B H \to X$, $q \colon X \times \B H \to \B H$ the natural projections.
\end{defn}

\begin{rmk}
It follows tautologically from the definition that whenever $f\colon X \to \B \mb{G}_m$ classifies a line bundle $\mc{L}$ with weight $\lambda$, the diagram
\[
\begin{tikzcd}
X \times \B H \ar[r] \ar[d, "{(f, \lambda)}"'] & X \ar[d, "f"] \\
\B \mb{G}_m \times \B \mb{G}_m \ar[r] & \B \mb{G}_m
\end{tikzcd}
\]
commutes, where the bottom arrow is given by tensor product of line bundles.
\end{rmk}

\begin{prop} \label{prop:mucheck}
With respect to the natural $\B Z(L)_{rig}$-action, the weight of the pullback $\Theta_{\bun_{L, rig}^{ss, \mu}}$ of $\Theta_{\bun_{G, rig}}$ to $\bun_{L, rig}^{ss, \mu}$ is given by
\[ (-\mu \mmid -) \in \hom(\mb{X}_*(Z(L)_{rig}), \mb{Z}) = \mb{X}^*(Z(L)_{rig}).\]
\end{prop}
\begin{proof}
We will in fact show that for any $\mc{L} \in \mrm{Pic}(\bun_{G, rig})$, the pullback of $\mc{L}$ to $\bun_{L, rig}^{ss, \mu}$ has weight $- Q(\mc{L})(\mu, -)$, where $Q(\mc{L})$ is the quadratic class of the corresponding $W$-linearised line bundle on $Y$. Since this statement is invariant under tensoring $\mc{L}$ with a line bundle on the base stack $S$ and raising $\mc{L}$ to a nonzero power, by Corollary \ref{cor:picbung}, it suffices to show this for a single nontrivial line bundle $\mc{L}$.

Choose any nontrivial representation $V$ of $G/Z(G)$, and set
\[ \mc{L} = \det\mb{R}{\mrm{pr}_{\bun_{G, rig}}}{\vphantom{p}}_*(\xi_{G/Z(G)} \times^{G/Z(G)} V),\]
where $\xi_{G/Z(G)} \to \bun_{G, rig} \times_S E$ is the $G/Z(G)$-bundle classified by the morphism $\bun_{G, rig} \to \bun_{G/Z(G)}$. Then the pullback of $\mc{L}$ to $\bun_{L, rig}^{ss, \mu}$ is the line bundle
\begin{align*}
\mc{L}_{\bun_{L, rig}^{ss, \mu}} &= \det \mb{R}\mrm{pr}_{\bun_{L, rig}^{ss, \mu}}{\vphantom{p}}_*(\xi_{L/Z(G)} \times^{L/Z(G)} V) \\
&= \bigotimes_{\lambda \in \mb{X}^*(Z(L)_{rig})} \det \mb{R}\mrm{pr}_{\bun_{L, rig}^{ss, \mu}}{\vphantom{p}}_*(\xi_{L/Z(G)} \times^{L/Z(G)} V_\lambda),
\end{align*}
where $\xi_{L/Z(G)} \to \bun_{L, rig} \times_S E$ is the $L/Z(G)$-bundle classified by the natural morphism $\bun_{L, rig} \to \bun_{L/Z(G)}$, and $V = \bigoplus_{\lambda \in \mb{X}^*(Z(L)_{rig})} V_\lambda$ is the weight space decomposition of $V$ under the action of the torus $Z(L)_{rig} = Z(L/Z(G))$. Pulling back along the action morphism $a \colon \bun_{L, rig}^{ss, \mu} \times \B Z(L)_{rig} \to \bun_{L, rig}^{ss, \mu}$, we get
\begin{align*}
a^*\mb{R}\mrm{pr}_{\bun_{L, rig}^{ss,\mu}}{\vphantom{p}}_*(\xi_{L/Z(G)} &\times^{L/Z(G)} V_\lambda) \\
&= \mb{R}\mrm{pr}_{\bun_{L, rig}^{ss, \mu} \times \B Z(L)_{rig}}{\vphantom{p}}_*((p^*\xi_{L/Z(G)} \otimes q^*\eta_{Z(L)_{rig}}) \times^{L/Z(G)} V_\lambda) \\
&= \mb{R}\mrm{pr}_{\bun_{L, rig}^{ss, \mu} \times \B Z(L)_{rig}}{\vphantom{p}}_*(p^*(\xi_{L/Z(G)} \times^{L/Z(G)}V_\lambda) \otimes q^*\lambda(\eta_{Z(L)_{rig}})) \\
&= \bar{p}^*(\mb{R}\mrm{pr}_{\bun_{L, rig}^{ss, \mu}}{\vphantom{p}}_*(\xi_{L/Z(G)} \times^{L/Z(G)} V_\lambda)) \otimes \bar{q}^*\lambda(\eta_{Z(L)_{rig}})
\end{align*}
where $\bar{p} \colon \bun_{L, rig}^{ss, \mu} \times \B Z(L)_{rig} \to \bun_{L, rig}^{ss, \mu}$ and $\bar{q} \colon \bun_{L, rig}^{ss, \mu} \times \B Z(L)_{rig} \to \B Z(L)_{rig}$ are the natural projections, $p$ and $q$ are their respective compositions with the projection $(\bun_{L, rig}^{ss, \mu} \times \B Z(L)_{rig}) \times_S E \to \bun_{L, rig}^{ss, \mu} \times \B Z(L)_{rig}$, and $\eta_{Z(L)_{rig}}$ is the universal $Z(L)_{rig}$-bundle on $\B Z(L)_{rig}$. So the weight of the determinant $\mc{L}_{\bun_{L, rig}^{ss, \mu}}$ is therefore
\begin{align*}
\sum_{\lambda \in \mb{X}^*(Z(L)_{rig})} \chi(\mb{R}\mrm{pr}_{\bun_{L, rig}^{ss, \mu}}{\vphantom{p}}_*(\xi_{L/Z(G)} \times^{L/Z(G)} V_\lambda)) \lambda &= \sum_{\lambda \in \mb{X}^*(Z(L)_{rig})} \dim V_\lambda \langle \lambda, \mu \rangle \lambda \\
&= -Q(\mc{L})(\mu, -)
\end{align*}
as claimed, where the last equality follows from Lemma \ref{lem:detquadraticclass}.
\end{proof}

\begin{proof}[Proof of Proposition \ref{prop:inducedequivariantslice}]
Since $\mrm{Ind}_L^G(Z_0) \to \bun_{G, rig}$ is a slice by Proposition \ref{prop:parabolicinductionsmooth}, we just need to construct an isomorphism of the pullback of $\Theta_{\bun_{G, rig}}^{-1}$ to $\mrm{Ind}_L^{G}(Z_0)/Z(L)_{rig}$ with the line bundle classified by
\[ \mrm{Ind}_L^G(Z_0)/Z(L)_{rig} \longrightarrow \B Z(L)_{rig} \xrightarrow{(\mu \mmid -)} \B \mb{G}_m.\]
Since $\mrm{Ind}_L^G(Z_0)/Z(L)_{rig} \to Z_0 \times \B Z(L)_{rig}$ is an affine space bundle by Proposition \ref{prop:inductionaffinebundle}, the pullback of $\Theta_{\bun_{G, rig}}^{-1}$ to $\mrm{Ind}_L^G(Z_0)/Z(L)_{rig}$ is canonically isomorphic to the pullback of its restriction $\Theta^{-1}_{Z_0 \times \B Z(L)_{rig}}$ to the zero section $Z_0 \times \B Z(L)_{rig}$, i.e., of the pullback of $\Theta_{\bun_{G, rig}}^{-1}$ along the morphism
\[ Z_0 \times \B Z(L)_{rig} \longrightarrow \bun_{L, rig}^{ss, \mu} \longrightarrow \bun_{G, rig}.\]
Since the first morphism above is $\B Z(L)_{rig}$-equivariant, Proposition \ref{prop:mucheck} implies that $\Theta^{-1}_{Z_0 \times \B Z(L)_{rig}}$ has weight $(\mu \mmid -)$. But the trivialisation of the pullback to $Z_0$ identifies $\Theta^{-1}_{Z_0 \times \B Z(L)_{rig}}$ with the pullback of a line bundle on $\B Z(L)_{rig}$, which must be associated to the character $(\mu \mmid -)$, so we are done.
\end{proof}

\subsection{Regular unstable bundles} \label{subsection:regularunstable}

In this subsection, we briefly review the theory of regular unstable bundles, and use Atiyah's classification of stable vector bundles to describe their moduli stacks.

\begin{defn}
Let $\alpha_i \in \Delta$ be a simple root. We say that $\alpha_i$ is \emph{special} if $\alpha_i$ is a long root, the connected components of the Dynkin diagram of $\Delta \setminus \{\alpha_i\}$ are all of type $A$, and $\alpha_i$ meets each such component at an end vertex. We call a principal bundle $\xi_P$ for a parabolic subgroup $P \subseteq G$ \emph{special} if $P$ is the standard maximal parabolic of type $\{\alpha_i\}$ with $\alpha_i \in \Delta$ special, and $\xi_P$ has slope $-\varpi_i^\vee/\langle \varpi_i, \varpi_i^\vee \rangle$.
\end{defn}

\begin{rmk}
If $G$ is of type $A$, then every simple root is special. Otherwise, there is a unique special root $\alpha_i \in \Delta$.
\end{rmk}

\begin{prop} \label{prop:regularunstable}
Let $\xi_G \to E_s$ be an unstable $G$-bundle on a geometric fibre of $E \to S$. The following are equivalent.
\begin{enumerate}[(1)]
\item \label{itm:regularunstable1} The Harder-Narasimhan reduction of $\xi_G$ is special.
\item \label{itm:regularunstable2} $\dim \mrm{Aut}(\xi_G) = l + 2$.
\item \label{itm:regularunstable3} $\dim \mrm{Aut}(\xi_G) \leq l + 2$.
\item \label{itm:regularunstable4} The Harder-Narasimhan locus of $\xi_G$ has codimension $l + 1$.
\item \label{itm:regularunstable5} The Harder-Narasimhan locus of $\xi_G$ has codimension $\leq l + 1$.
\end{enumerate}
\end{prop}
\begin{proof}
This is \cite[Proposition 5.4.2]{davis19a}.
\end{proof}

\begin{defn} \label{defn:regularunstable}
Let $\xi_G \to E_s$ be an unstable $G$-bundle on a geometric fibre of $E \to S$. We say that $\xi_G$ is \emph{regular} if it satisfies the equivalent conditions of Proposition \ref{prop:regularunstable}.
\end{defn}

The Levi subgroups associated to special roots have a very simple structure. In the following proposition, we let $\alpha_i \in \Delta$ be a special root and $L$ the Levi factor of the standard parabolic subgroup $P \subseteq G$ of type $t(P) = \{\alpha_i\}$. We write $\pi_0 = \pi_0(\Delta \setminus \{\alpha_i\})$ for the set of connected components of the Dynkin diagram of $L$; since $\alpha_i$ is special, each component $c$ is of type $A_{n_c}$ for some $n_c \geq 1$.

\begin{prop} \label{prop:speciallevi}
In the setup above, there is an isomorphism
\begin{equation} \label{eq:speciallevi1}
 L \cong \left\{\left.((A_c)_{c \in \pi_0}, \lambda) \in \prod_{c \in \pi_0} GL_{n_c + 1} \times \mb{G}_m \,\right|\, \det A_c = \lambda \right\}
\end{equation}
such that the character $\varpi_i$ of $L$ is given by \eqref{eq:speciallevi1} composed with the projection to $\mb{G}_m$.
\end{prop}
\begin{proof}
This is a restatement of \cite[Lemma 3.2.1]{friedman-morgan00}.
\end{proof}

Next, we state a version of Atiyah's classification \cite{atiyah57} of stable vector bundles on an elliptic curve, adapted to our context.

\begin{thm} \label{thm:atiyahclassification}
Let $r > 0$ and $d$ be coprime integers. Then the determinant morphism
\begin{equation} \label{eq:atiyahclassificationdet}
\det \colon \bun_{GL_r}^{ss, d} \longrightarrow \mrm{Pic}_S^d(E)
\end{equation}
from the stack of semistable vector bundles on $E$ of rank $r$ and degree $d$ to the Picard variety of degree $d$ line bundles on $E$ is a $\mb{G}_m$-gerbe, where $\B \mb{G}_m$ acts on $\bun_{GL_r}^{ss, d}$ through the centre $\mb{G}_m = Z(GL_r)$ in the usual way. If $E \to S$ has a section, then the gerbe \eqref{eq:atiyahclassificationdet} is trivial.
\end{thm}

Theorem \ref{thm:atiyahclassification} is a straightforward refinement of Atiyah's classification of indecomposable vector bundles on an elliptic curve, taking automorphism groups into account. For a proof, see \cite[Theorem 5.3.2]{davis19a}.

Combining Proposition \ref{prop:speciallevi} and Theorem \ref{thm:atiyahclassification} gives the following description of the stack of semistable $L$-bundles when $L$ is the Levi subgroup attached to a special root.

\begin{prop} \label{prop:regunstablemoduli}
Assume that $E \to S$ has a section $O_E \colon S \to E$, fix a special root $\alpha_i \in \Delta$ and let $\mu = -\varpi_i^\vee/\langle \varpi_i, \varpi_i^\vee \rangle$. Then $\varpi_i \colon \bun_{L, rig}^{ss, \mu} \to \mrm{Pic}^{-1}_S(E)$ is a $Z(L)_{rig}$-gerbe, and there exists a unique section $S \to \bun_{L, rig}^{ss, \mu}$ lifting the section $O_E \colon S \to E \cong \mrm{Pic}_S^{-1}(E)$ such that the pullback of the theta bundle $\Theta_{\bun_{G, rig}}$ to $S$ is trivial.
\end{prop}
\begin{proof}
By Proposition \ref{prop:speciallevi}, we have a Cartesian diagram
\[
\begin{tikzcd}
\bun_L^{ss, \mu} \ar[r] \ar[d, "\varpi_i"] & \prod_{c \in \pi_0} \bun_{GL_{n_c + 1}}^{ss, -1} \ar[d, "\det"] \\
\bun_{\mb{G}_m}^{-1} \ar[r] & \prod_{c \in \pi_0} \bun_{\mb{G}_m}^{-1},
\end{tikzcd}
\]
where the products denote iterated fibre products over $S$, and the bottom morphism is the diagonal. Since $\bun_{GL_{n_c + 1}}^{ss, -1} \to \mrm{Pic}^{-1}_S(E)$ and $\bun_{\mb{G}_m}^{-1} \to \mrm{Pic}^{-1}_S(E)$ are a $\mb{G}_m = Z(GL_{n_c + 1})$-gerbe and a $\mb{G}_m$-gerbe respectively by Theorem \ref{thm:atiyahclassification}, it follows that $\bun_L^{ss, \mu} \to \mrm{Pic}^{-1}_S(E)$ is a $Z(L) = \mb{G}_m \times_{\prod_c \mb{G}_m} \prod_c Z(GL_{n_c + 1})$-gerbe, and hence that $\bun_{L, rig}^{ss, \mu} \to \mrm{Pic}_S^{-1}(E)$ is a $Z(L)_{rig}$-gerbe as claimed.

To construct the section $S \to \bun_{L, rig}^{ss, \mu}$, note that since $(-\mu \mmid -) \colon Z(L)_{rig} \to \mb{G}_m$ is an isomorphism, the pullback of the theta bundle defines a $\B Z(L)_{rig}$-equivariant morphism
\[ S \times_{\mrm{Pic}^{-1}_S(E)} \bun_{L, rig}^{ss, \mu} \longrightarrow \B \mb{G}_m \cong \B Z(L)_{rig}\]
by Proposition \ref{prop:mucheck}. Since the source is a $Z(L)_{rig} = \mb{G}_m$-gerbe over $S$, it follows that there is a unique section such that the pullback of $\Theta_{\bun_{G, rig}}$ is trivial as claimed.
\end{proof}

\subsection{The Friedman-Morgan section theorem} \label{subsection:friedmanmorgan}

In this subsection, we state and prove our refinement of Friedman and Morgan's section theorem \cite[Theorem 5.1.1]{friedman-morgan00}. This theorem, and certain elements of its proof, are the main ingredients in the proof that the elliptic Grothendieck-Springer resolution is a simultaneous log resolution.

Fix a special root $\alpha_i \in \Delta$, let $\mu = -\varpi_i^\vee/\langle \varpi_i, \varpi_i^\vee \rangle$, and assume that $E \to S$ has a section $O_E$. Then the composition of the section $S \to \bun_{L, rig}^{ss, \mu}$ of Proposition \ref{prop:regunstablemoduli} with $\bun_{L, rig}^{ss, \mu} \to \bun_{L, rig}^{ss, \mu}/E$ is a section of a $Z(L)_{rig} = \mb{G}_m$-gerbe, and is hence smooth. Since the pullback of the theta bundle to $S$ is trivial, $Z = \mrm{Ind}_L^G(S) \to \bun_{G, rig}$ is an equivariant slice with equivariance group $Z(L)_{rig}$ and weight $(\mu\mmid -)$ by Proposition \ref{prop:inducedequivariantslice}. Identifying $Z(L)_{rig}$ with $\mb{G}_m$ via the isomorphism
\[ (\mu \mmid -) \colon Z(L)_{rig} \overset{\sim}\longrightarrow \mb{G}_m,\]
we therefore have a $\mb{G}_m$-equivariant commutative diagram
\[
\begin{tikzcd}
\tilde{Z} \ar[r] \ar[d, "\tilde{\chi}_Z"'] & Z \ar[d, "\chi_Z"] \\
\Theta_Y^{-1} \ar[r] & \hat{Y}\sslash W
\end{tikzcd}
\]
as in Remark \ref{rmk:equivariantthetasection}, where $\tilde{Z} = Z \times_{\bun_{G, rig}} \tbun_{G, rig}$.

\begin{rmk} \label{rmk:regularslicesmoothness}
Note that if the section $S \to \bun_{L, rig}^{ss, \mu}$ factors through $\bun_L^{ss, \mu}$, then the slice $Z$ constructed above factors through a morphism $Z \to \bun_G$. However, even when this happens, this morphism will not necessarily be a slice unless $S \times Z(L) \to S$ is smooth, i.e., unless $S \to \spec \mb{Z}$ avoids all primes at which $Z(L)$ is non-reduced.
\end{rmk}

\begin{thm}[Friedman-Morgan section theorem] \label{thm:friedmanmorgansection}
In the setup above, the morphisms
\[ \chi_Z \colon Z \longrightarrow \hat{Y}\sslash W \quad \text{and} \quad \tilde{\chi}_Z \colon \tilde{Z} \longrightarrow \Theta_Y^{-1}\]
are $\mb{G}_m$-equivariant isomorphisms. In particular, the rigidified coarse quotient map
\[ \chi \colon \bun_{G, rig} \longrightarrow (\hat{Y}\sslash W)/\mb{G}_m\]
admits a section.
\end{thm}

The strategy of our proof is somewhat different to Friedman and Morgan's, and relies heavily on a computation of the pullback $\tilde{Z}$ of the elliptic Grothendieck-Springer resolution to the slice $Z$.

For future reference, we make a note of the dimension of $Z$ over $S$.

\begin{lem} \label{lem:regularslicedim}
The morphism $Z \to S$ is an affine space bundle of relative dimension $l + 1$.
\end{lem}
\begin{proof}
This follows immediately from Propositions \ref{prop:inductionaffinebundle} and \ref{prop:inductionreldim}, and \cite[Lemma 3.3.2]{friedman-morgan00}.
\end{proof}

For $\lambda \in \mb{X}_*(T)_+$, write $D_\lambda(Z) = \tilde{Z} \times_{\tbun_{G, rig}} D_\lambda$. The first step in the proof of Theorem \ref{thm:friedmanmorgansection} is to identify those $D_\lambda(Z) \subseteq \tilde{Z}$ that are nonempty.

\begin{lem} \label{lem:emptyregulardivisors}
Assume $\lambda \in \mb{X}_*(T)_+$ and $\lambda \neq \alpha_i^\vee$. Then $D_\lambda(Z) = \emptyset$.
\end{lem}
\begin{proof}
Assume for a contradiction that there exists $\lambda \in \mb{X}_*(T)_+$ with $\lambda \neq \alpha_i^\vee$ and $D_\lambda(Z) \neq \emptyset$. Since $\lambda \neq \alpha_i^\vee$, there exists $\alpha_j \in \Delta$ such that $\langle \varpi_j, \lambda \rangle > 0$ and
\[ \mu' = - \frac{\langle \varpi_j, \lambda \rangle}{\langle \varpi_j, \varpi_j^\vee \rangle} \varpi_j^\vee \neq \mu.\]
Since $D_\lambda(Z) \neq \emptyset$, it follows that $Z \times_{\bun_{G, rig}} \bun_{B, rig}^{-\lambda}$, and hence $Z \times_{\bun_{G, rig}} \bun_{P_j, rig}^{\mu'}$, is nonempty, where $P_j$ is the standard maximal parabolic of type $\{\alpha_j\}$. We claim that $-\langle 2 \rho, \mu' \rangle \leq l$. This contradicts \cite[Lemma 3.3.2]{friedman-morgan00}, and so proves the lemma.

To prove the claim, note that since $Z \to \bun_{G, rig}/E$ is smooth, the preimage $Z \times_{\bun_{G, rig}} \bun_{P_j, rig}^{ss, \mu'}$ of $\bun_{P_j, rig}^{ss, \mu'}/E$ under the morphism
\[ Z \times_{\bun_{G, rig}} \bun_{P_j, rig}^{\mu'} \longrightarrow \bun_{P_j, rig}^{\mu'}/E \]
is dense, hence nonempty. Since $(P_j, \mu') \neq (P, \mu)$, where $P$ is the standard parabolic of type $\{\alpha_i\}$, by uniqueness of Harder-Narasimhan reductions, the $Z(L)_{rig}$-invariant locally closed substack $Z \times_{\bun_{G, rig}} \bun_{P_j, rig}^{ss, \mu'} \subseteq Z$ is disjoint from the $Z(L)_{rig}$-fixed locus $S \subseteq Z$. So $Z \times_{\bun_{G, rig}}\bun_{P_j, rig}^{\mu'} \to S$ is therefore flat of relative dimension at least $1$, and hence has codimension at most
\[ \dim_S Z - 1 = -\langle 2 \rho_P, \mu \rangle - 1 = l\]
by Lemma \ref{lem:regularslicedim}. But this codimension is equal to the codimension $- \langle 2\rho, \mu' \rangle$ of $\bun_{P_j, rig}^{ss, \mu'}/E$ in $\bun_{G, rig}/E$, so we have
\[ -\langle 2 \rho, \mu' \rangle \leq l,\]
as claimed.
\end{proof}

\begin{lem} \label{lem:regularsliceetale}
The morphism $D_{\alpha_i^\vee}(Z) \to Y$ is representable and \'etale.
\end{lem}
\begin{proof}
First note that Lemma \ref{lem:emptyregulardivisors} implies that no point of $\tilde{Z}$ can have nontrivial automorphism group relative to $Z$, so $\tilde{Z} \to Z$ is projective. So $\tilde{Z}$ and its closed substack $D_{\alpha_i^\vee}(Z)$ are representable over $S$, so $D_{\alpha_i^\vee}(Z) \to Y$ is representable as claimed.

For \'etaleness, note that the morphism
\[ \tilde{Z} \longrightarrow \tbun_{G, rig}/E \longrightarrow (Y \times_S \Deg_S(E))/E \cong Y \times_S \Deg_S(E)/E \]
is smooth since $Z \to \bun_{G, rig}$ is a slice. Since the boundary divisor $D \subseteq \Deg_S(E)$ is a reduced divisor with normal crossings relative to $S$, the closed substack $D_{\alpha_i^\vee}(Z) = \tilde{Z} \times_{Y \times_S \Deg_S(E)/E} (Y \times_S (D/E))$ is a reduced divisor with normal crossings relative to $Y$. Since $\tilde{Z} \to Y$ has relative dimension $1$, the irreducible components of $D_{\alpha_i^\vee}(Z)$ are therefore disjoint and smooth of relative dimension $0$, hence \'etale, over $Y$.
\end{proof}

In the following lemmas, for $1 \leq r \leq n$, we write
\[ Q_r^n = \{(a_{p, q})_{1 \leq p, q \leq n} \in GL_n \mid a_{p, q} = 0 \;\;\text{if}\;\; p < \mrm{min}(q, r)\} \subseteq GL_n,\]
and $T_r^n = Q_r^n/[Q_r^n, Q_r^n] \cong \mb{G}_m^r$. For $1 \leq i \leq n$, we write $e_i \in \mb{X}^*(T^n_n)$ for the character sending a diagonal matrix to its $i$th entry, and we write $\{e_1^*, \ldots, e_n^*\} \subseteq \mb{X}_*(T^n_n)$ for the basis dual to $\{e_1, \ldots, e_n\}$. If $\lambda \in \mb{X}_*(T^n_n)$, we will also write $\lambda$ for its image in $\mb{X}_*(T^n_r)$. Finally, for $\lambda \in \mb{X}_*(T)$ (resp., $\lambda \in \mb{X}_*(T_r^n)$), we write $Y^\lambda \subseteq \hom(\mb{X}^*(T), \mrm{Pic}_S(E))$ (resp., $Y_{Q^n_r}^\lambda$) for the rigidification of $\bun_T^\lambda$ (resp., $\bun_{T_r^n}^\lambda$) with respect to $T$ (resp., $T^n_r$).

\begin{lem} \label{lem:regularlevisection}
There is an isomorphism
\[ \bun_{L \cap B}^{-\alpha_i^\vee} \times_{\bun_L^\mu} \bun_L^{ss, \mu} \overset{\sim}\longrightarrow Y^{-\alpha_i^\vee} \times_{\mrm{Pic}^{-1}_S(E)} \bun_L^{ss, \mu}\]
and hence an isomorphism
\[ \bun_{L \cap B, rig}^{-\alpha_i^\vee} \times_{\bun_{L, rig}^\mu} \bun_{L, rig}^{ss, \mu} \overset{\sim}\longrightarrow Y^{-\alpha_i^\vee} \times_{\mrm{Pic}^{-1}_S(E)} \bun_{L, rig}^{ss, \mu}\]
sending an $L \cap B$-bundle to its associated $T$-bundle and $L$-bundle.
\end{lem}
\begin{proof}
Using the isomorphism of Proposition \ref{prop:speciallevi}, the claim reduces easily to Lemma \ref{lem:ssvbfiltration} below.
\end{proof}

\begin{lem} \label{lem:ssvbfiltration}
Let $n > 0$ and $1 \leq r \leq n$. Then the morphism
\begin{equation} \label{eq:ssvbfiltration1}
\bun_{Q^n_r}^{-e_n^*} \times_{\bun_{GL_n}^{-1}} \bun_{GL_n}^{ss, -1} \longrightarrow Y_{Q^n_r}^{-e_n^*} \times_{\mrm{Pic}^{-1}_S(E)} \bun_{GL_n}^{ss, -1}
\end{equation}
is an isomorphism, where the morphisms $Y_{Q^n_r}^{-e_n^*} \to \mrm{Pic}^{-1}_S(E)$ and $\bun_{GL_n}^{ss, -1} \to \mrm{Pic}^{-1}_S(E)$ are both given by the determinant.
\end{lem}
\begin{proof}
For the sake of brevity, we will write
\[ (\bun_{Q^n_r}^{-e_n^*})^{ss} = \bun_{Q^n_r}^{-e_n^*} \times_{\bun_{GL_n}^{-1}} \bun_{GL_n}^{ss, -1}.\]

We prove the claim by induction on $r$. For $r = 1$, the claim is true since $Q^n_1 = GL_n$ and $Y_{Q^n_1}^{-e_n^*} = \mrm{Pic}^{-1}_S(E)$. 

Next, suppose that $r = 2$. In this case, we construct an inverse to \eqref{eq:ssvbfiltration1} as follows. Let $V \to Y_{Q^n_2}^{-e_n^*} \times_{\mrm{Pic}^{-1}_S(E)} \bun_{GL_n}^{ss, -1} \times_S E$ be the pullback of the universal vector bundle on $\bun_{GL_n}^{ss, -1} \times_S E$ and let $M_{e_1} \to Y_{Q^n_2}^{-e_n^*} \times_{\mrm{Pic}^{-1}_S(E)} \bun_{GL_n}^{ss, -1} \times_S E$ be the pullback of the degree $0$ line bundle on $Y_{Q^n_2}^{-e_n^*} \times_S E$ associated to the character $e_1$. Writing $p\colon Y_{Q^n_2}^{-e_n^*} \times_{\mrm{Pic}^{-1}_S(E)} \bun_{GL_n}^{ss, -1} \times_S E \to Y_{Q^n_2}^{-e_n^*} \times_{\mrm{Pic}^{-1}_S(E)} \bun_{GL_n}^{ss, -1}$ for the projection to the first two factors, we have that $N = p_*(M_{e_1} \otimes V^\vee)$ is a line bundle since $M_{e_1} \otimes V^\vee$ is a family of semistable vector bundles of degree $1$. By semistability of $V$, we therefore have a natural exact sequence
\begin{equation} \label{eq:ssvbfiltration2}
0 \longrightarrow V' \longrightarrow V \longrightarrow M_{e_1} \otimes p^*N^{-1} \longrightarrow 0,
\end{equation}
where $M_{e_1} \otimes p^*N^{-1}$ is a line bundle fibrewise of degree $0$ and $V'$ is a vector bundle of rank $n - 1$ and degree $-1$. The exact sequence \eqref{eq:ssvbfiltration2} defines a degree $-e_n^*$ reduction to $Q_2^n$ of the pullback of the universal $GL_n$-bundle, and hence a morphism
\[ Y_{Q_2^n}^{-e_n^*} \times_{\mrm{Pic}^{-1}_S(E)} \bun_{GL_n}^{ss, -1} \longrightarrow (\bun_{Q^n_2}^{-e_n^*})^{ss},\]
which is easily shown to be inverse to \eqref{eq:ssvbfiltration1}. This proves the claim for $r = 2$.

Finally, assume that $r > 2$ and the lemma holds for all smaller $r$. We need to show that the outer square in the commutative diagram
\begin{equation} \label{eq:ssvbfiltration3}
\begin{tikzcd}
(\bun_{Q^n_r}^{-e_n^*})^{ss} \ar[r] \ar[d] & (\bun_{Q^n_{r - 1}}^{-e_n^*})^{ss} \ar[r] \ar[d] & \bun_{GL_n}^{ss, -1} \ar[d] \\
Y_{Q^n_r}^{-e_n^*} \ar[r] & Y_{Q^n_{r - 1}}^{-e_n^*} \ar[r] & \mrm{Pic}^{-1}_S(E)
\end{tikzcd}
\end{equation}
is Cartesian. Since the rightmost square is Cartesian by induction, it suffices to show that the leftmost suqare is Cartesian. 

Observe that there is a surjective homomorphism $Q^n_{r - 1} \to GL_{n - r + 2}$ given by forgetting the first $r - 2$ rows and columns, such that $Q^n_r \subseteq Q^n_{r - 1}$ is the preimage of $Q^{n - r + 2}_2 \subseteq GL_{n - r + 2}$. Since the morphism
\[ \bun_{Q^n_{r - 1}}^{-e_n^*} \longrightarrow \bun_{GL_{n - r + 2}}^{-1} \]
sends $(\bun_{Q^n_{r - 1}}^{-e_n^*})^{ss}$ to $\bun_{GL_{n - r + 2}}^{ss, -1}$ by \cite[Lemma II.15]{atiyah57}, we therefore have a diagram
\[
\begin{tikzcd}
(\bun_{Q^n_r}^{-e_n^*})^{ss} \ar[r] \ar[d] & (\bun_{Q^{n - r + 2}_2}^{e_{n - r + 2}^*})^{ss} \ar[d] \ar[r] & Y_{Q^{n - r + 2}_2}^{-e_{n - r + 2}^*} \ar[d] \\
(\bun_{Q^n_{r - 1}}^{-e_n^*})^{ss} \ar[r] & \bun_{GL_{n - r + 2}}^{ss, -1} \ar[r] & \mrm{Pic}^{-1}_S(E)
\end{tikzcd}
\]
in which the leftmost square is Cartesian. Since the rightmost square is also Cartesian by induction, the outer square is also. This is also the outermost square in the diagram
\[
\begin{tikzcd}
(\bun_{Q^n_r}^{-e_n^*})^{ss} \ar[r] \ar[d] & Y_{Q^n_r}^{-e_n^*} \ar[d] \ar[r] & Y_{Q^{n - r + 2}_2}^{-e_{n - r + 2}^*} \ar[d] \\
(\bun_{Q^n_{r - 1}}^{-e_n^*})^{ss} \ar[r] & Y_{Q^n_{r - 1}}^{-e_n^*} \ar[r] & \mrm{Pic}^{-1}_S(E).
\end{tikzcd}
\]
It is easy to see that the rightmost square in this diagram is Cartesian and hence so is the leftmost square. But this coincides with the leftmost square of \eqref{eq:ssvbfiltration3}, so we are done.
\end{proof}

\begin{prop} \label{prop:nonemptyregulardivisor}
The morphism $D_{\alpha_i^\vee}(Z) \to Y$ is an isomorphism.
\end{prop}
\begin{proof}
Since Lemma \ref{lem:emptyregulardivisors} implies that $D_{\alpha_i^\vee}(Z) = \tilde{Z} \times_{\tbun_{G, rig}} D_{\alpha_i^\vee, rig}^\circ$, the isomorphism of Proposition \ref{prop:kmdivisor} \eqref{itm:kmdivisor3} and the well known isomorphism (cf.\! \cite[Proposition 3.4.10]{davis19a}) $M_{0, 1}^+(G/B, \alpha_i^\vee) \cong \spec \mb{Z}$ combine to give an isomorphism
\begin{equation} \label{eq:nonemptyregulardivisor0}
 D_{\alpha_i^\vee}(Z) \cong Z \times_{\bun_{G, rig}} \bun_{B, rig}^{-\alpha_i^\vee} \times_S E.
\end{equation}

We next claim that the closed immersion
\[ S \times_{\bun_{G, rig}} \bun_{B, rig}^{-\alpha_i^\vee} \longrightarrow Z \times_{\bun_{G, rig}} \bun_{B, rig}^{-\alpha_i^\vee} \]
is surjective. To prove the claim, we may assume without loss of generality that $S = \spec k$ for some algebraically closed field. Assume that $z \colon \spec k \to Z$ is a geometric point such that $\{z\} \times_{\bun_{G, rig}} \bun_{B, rig}^{-\alpha_i^\vee}$ is nonempty, let $\xi_{G, z}$ denote the corresponding $G$-bundle, and let $x$ denote its image in $\bun_{G, rig}/E$. By $\mb{G}_m$-equivariance of $Z \to \bun_{G, rig}$, we have
\[ \mb{G}_m \cdot z \subseteq Z \times_{\bun_{G, rig}/E} \B \mrm{Aut}(x).\]
So
\begin{align*}
\dim \mrm{Aut}(\xi_G) \leq \dim \mrm{Aut}(x) &= \codim_{\bun_{G, rig}/E}(\B \mrm{Aut}(x)) + 1 \\
&= \codim_Z (Z \times_{\bun_{G, rig}/E} \B \mrm{Aut}(x)) + 1 \\
&\leq l + 2 - \dim \mb{G}_m \cdot z,
\end{align*}
and hence $\dim \mb{G}_m \cdot z = 0$ by Proposition \ref{prop:regularunstable}. So $z \in Z^{\mb{G}_m} = S$, which proves the claim.

Let $P \subseteq G$ be the standard parabolic subgroup with Levi factor $L$, and consider the locally closed Bruhat cells
\[ C^w = S \times_{\bun_{P, rig}} C^{w, -\alpha_i^\vee}_{P, rig} \subseteq S \times_{\bun_{G, rig}} \bun_{B, rig}^{-\alpha_i^\vee} ,\]
for $w \in W^0_P$, where $C^{w, -\alpha_i^\vee}_{P, rig} \subseteq \bun_{P, rig} \times_{\bun_{G, rig}} \bun_{B, rig}$ is the rigidification of the locally closed substack $C^{w, -\alpha_i^\vee}_{P} \subseteq \bun_P \times_{\bun_G} \bun_B$ defined in \S\ref{subsection:bruhat}. Then Proposition \ref{prop:levidegreebound} and Lemma \ref{lem:emptyregulardivisors} imply that
\[ \bun_{B, rig}^{-\alpha_i^\vee} \times_{\bun_{G, rig}} S = \bigcup_{w \in W^0_P} C^w.\]

Assume that $w \in W^0_P$ with $C^w \neq \emptyset$. Then by Proposition \ref{prop:bruhatdegrees} there exist a geometric point $s\colon \spec k \to S$ with corresponding $L$-bundle $\xi_L \to E_s$ and a section $\sigma_L \colon E_s \to \xi_L \times^L L/(L \cap B)$ of degree $[\sigma_L] = - w \alpha_i^\vee \in \Phi^\vee$. Since $\xi_L$ has slope $\mu$, we must have $\langle \varpi_i, [\sigma_L] \rangle = -1$ and hence $[\sigma_L] \in \Phi^\vee_- \subseteq \mb{X}_*(T)_-$. Since $[\sigma_L]$ is the degree of the section
\[ \sigma' \colon E_s \overset{\sigma} \longrightarrow \xi_L \times^L L/(L \cap B) \longhookrightarrow \xi_L \times^L G/B,\]
Lemma \ref{lem:emptyregulardivisors} implies that we must have $[\sigma_L] = -\alpha_i^\vee$. So $w \in W^0_P$ and $w \alpha_i^\vee = \alpha_i^\vee$, which implies that $w^{-1}(\Phi_+^\vee) = \Phi_+^\vee$, and hence $w = 1$. So $C^w = \emptyset$ for $w \neq 1$, and hence the closed immersion
\begin{equation} \label{eq:nonemptyregulardivisor1}
C^1 \longhookrightarrow S \times_{\bun_{G, rig}} \bun_{B, rig}^{-\alpha_i^\vee} \longhookrightarrow Z \times_{\bun_{G, rig}} \bun_{B, rig}^{-\alpha_i^\vee}
\end{equation}
is surjective on geometric points. Since $D_{\alpha_i^\vee}(Z)$ is \'etale over $Y$, hence reduced, $Z \times_{\bun_{G, rig}} \bun_{B, rig}^{-\alpha_i^\vee}$ is reduced as well, so \eqref{eq:nonemptyregulardivisor1} is an isomorphism. But by Lemma \ref{lem:regularlevisection},
\[ C^1 = S \times_{\bun_{L, rig}^\mu} \bun_{L \cap B, rig}^{-\alpha_i^\vee} \overset{\sim}\longrightarrow Y^{-\alpha_i^\vee} \times_{\mrm{Pic}^{-1}_S(E)} \{\mc{O}(-O_E)\},\]
is an isomorphism, where the morphism $Y^{-\alpha_i^\vee} \to \mrm{Pic}^{-1}_S(E)$ is induced by the character $\varpi_i \colon T \to \mb{G}_m$. So by \eqref{eq:nonemptyregulardivisor0} and Proposition \ref{prop:blowdowngluing}, we can identify the morphism $D_{\alpha_i^\vee}(Z) \to Y$ with the isomorphism
\begin{align*}
E \times_S Y^{-\alpha_i^\vee} \times_{\mrm{Pic}^{-1}_S(E)} \{\mc{O}(-O_E)\} &\longrightarrow Y \\
(p, \xi_T) &\longmapsto \alpha_i^\vee(\mc{O}(p)) \otimes \xi_T,
\end{align*}
which completes the proof of the proposition.
\end{proof}

\begin{cor} \label{cor:regularsliceunstableres}
The morphism $\tilde{\chi}_Z^{-1}(0_{\Theta_Y^{-1}}) \to 0_{\Theta_Y^{-1}} = Y$ is an isomorphism, where $0_{\Theta_Y^{-1}}$ denotes the zero section of $\Theta_Y^{-1}$.
\end{cor}
\begin{proof}
Since $Z$ is a slice of $\bun_{G, rig}$, Corollary \ref{cor:chitildedivisor} implies that
\[ \tilde{\chi}_Z^{-1}(0_{\Theta_Y^{-1}}) = \sum_{\lambda \in \mb{X}_*(T)_+} \frac{1}{2}(\lambda \mmid \lambda) D_\lambda(Z).\]
Applying Lemma \ref{lem:emptyregulardivisors}, this simplifies to
\[ \tilde{\chi}_Z^{-1}(0_{\Theta_Y^{-1}}) = \frac{1}{2}(\alpha_i^\vee \mmid \alpha_i^\vee)D_{\alpha_i^\vee}(Z) = D_{\alpha_i^\vee}(Z).\]
The claim now follows from Proposition \ref{prop:nonemptyregulardivisor}.
\end{proof}

\begin{prop} \label{prop:regularsliceresolution}
The morphism $\tilde{Z} \to \Theta_Y^{-1}$ is an isomorphism.
\end{prop}

The idea behind the proof of Proposition \ref{prop:regularsliceresolution} is to use the $\mb{G}_m$-action on $\tilde{Z}$ and $\Theta_Y^{-1}$ to reduce to Corollary \ref{cor:regularsliceunstableres}. The key tool is the following technical lemma.

\begin{lem} \label{lem:conerecognition}
Suppose that $f \colon X \to X'$ is a proper representable $\mb{G}_m$-equivariant morphism of stacks and that $X'_0 \subseteq X'$ and $X_0 \subseteq f^{-1}(X_0')$ are closed substacks satisfying the following conditions.
\begin{enumerate}[(1)]
\item $X_0 = f^{-1}(X_0')$ set-theoretically.
\item $\mb{G}_m$ acts trivially on the closed substacks $X'_0$ and $X_0$.
\item There exists a $\mb{G}_m$-equivariant retraction $X' \to X'_0$ so that $X'$ is an affine space bundle over $X'_0$ on which $\mb{G}_m$ acts with positive weights.
\item The induced action of $\mb{G}_m$ on the normal cone $C_{X_0/X}$ has a single nonzero weight.
\end{enumerate}
Then there is a unique $\mb{G}_m$-equivariant isomorphism $X \cong C_{X_0/X}$ over $X_0'$ sending $X_0 \subseteq X$ to the zero section via the identity and inducing the identity on normal cones.
\end{lem}
\begin{proof}
We first remark that since $\mb{G}_m$ acts on $C_{X_0/X}$ with a single nonzero weight, every $\mb{G}_m$-equivariant automorphism of $C_{X_0/X}$ that acts as the identity on $X_0$ and the normal cone of $X_0$ in $C_{X_0/X}$ is (canonically $2$-isomorphic to) the identity. So uniqueness follows.

The idea behind the proof of existence is to show that the deformation to the normal cone is trivial. We do this by first compactifying, so that we are in a position to apply the Grothendieck existence theorem, and then showing that the deformation is trivial infinitesimally.

First note that by the uniqueness just shown, we can reduce the proof of existence by descent to the case where $X_0' = \spec A$ for some Noetherian ring $A$. Again using uniqueness and fpqc descent for morphisms of separated algebraic spaces, it suffices to show that the desired isomorphism exists after base change along the fpqc morphism $\spec A\llbracket t \rrbracket [t^{-1}] \to \spec A$.

Let $C \to \spec A[t]$ and $C' \to \spec A[t]$ denote the deformations to the normal cone of $X_0$ in $X$ and $X_0'$ in $X'$ respectively. Then there are canonical inclusions
\[ X_0 \times_{\spec A} \spec A[t] \longhookrightarrow C, \quad \text{and} \quad X_0' \times_{\spec A} \spec A[t] = \spec A[t] \longhookrightarrow C'.\]
Define compactifications
\[ \bar{C} = ((C \times \mb{A}^1) \setminus (X_0 \times_{\spec A} \spec A[t] \times \{0\}))/\mb{G}_m \quad \text{and} \quad \bar{C}' = ((C' \times \mb{A}^1) \setminus (\spec A[t] \times \{0\}))/\mb{G}_m,\]
where $\mb{G}_m$ acts on $C$ and $C'$ over $\spec A \llbracket t \rrbracket$ via the action induced from the action on $X$ and $X'$, and $\mb{G}_m$ acts on $\mb{A}^1$ via the usual weight $1$ action. Then $C' \to \spec A[t]$ is an affine space bundle on which $\mb{G}_m$ acts with positive weights, and hence $\bar{C}' \to \spec A[t]$ is a bundle of weighted projective spaces, and in particular proper. We also have that $C \to C'$ factors as
\[ C \longhookrightarrow C' \times_{X'} X \longrightarrow C',\]
where the first morphism is a closed immersion, hence proper, and the second morphism is proper by assumption on $f$. So $C \to C'$ is proper, and hence so are $\bar{C} \to \bar{C}'$ and $\bar{C} \to \spec A[t]$. We also write
\[ \bar{C}_{X_0/X} = ((C_{X_0/X} \times \mb{A}^1) \setminus (X_0 \times \{0\}))/\mb{G}_m = \bar{C} \times_{\spec A[t], t \mapsto 0} \spec A,\]
and observe that $\bar{C}_{X_0/X} \to X_0' = \spec A$ is also proper. Note also that we have divisors at infinity $(C_{X_0/X} \setminus X_0)/\mb{G}_m \subseteq \bar{C}_{X_0/X}$ and $(C \setminus X_0 \times_{\spec A} \spec A[t])/\mb{G}_m \subseteq \bar{C}$ whose complements are canonically isomorphic to $C_{X_0/X}$ and $C$ respectively.

We claim that there is a unique $\mb{G}_m$-equivariant isomorphism $C^\wedge \cong C_{X_0/X} \times_{\spec A} \spf A\llbracket t \rrbracket$ of formal stacks over $\spf A\llbracket t \rrbracket$ acting as the identity on $X_0 \times_{\spec A} \spf A\llbracket t \rrbracket$ and on the normal cone of $X_0 \times_{\spec A} \spf A\llbracket t \rrbracket$ in both sides. Given the claim, this extends to an isomorphism between proper formal stacks $\bar{C}^\wedge \cong \bar{C}_{X_0/X} \times_{\spec A} \spf A\llbracket t \rrbracket$, and hence an isomorphism $\bar{C} \times_{\spec A[t]} \spec A\llbracket t\rrbracket \cong \bar{C}_{X_0/X} \times_{\spec A} \spec A \llbracket t \rrbracket$ by the Grothendieck existence theorem. Since this isomorphism identifies the divisors at infinity and since the restricted deformation to the normal cone $C \to \spec A[t, t^{-1}]$ is canonically trivial, it restricts to give the desired isomorphism
\[ X \times_{\spec A} \spec A\llbracket t \rrbracket [t^{-1}] \cong C_{X_0/X} \times_{\spec A} \spec A\llbracket t \rrbracket [t^{-1}].\]

To prove the claim, it is enough to prove existence and uniqueness of isomorphisms $C_n \cong C_{X_0/X} \times_{\spec A} \spec A[t]/(t^n)$ for all $n \geq 0$ with the desired properties, where $C_n = C \times_{\spec A[t]} \spec A[t]/(t^n)$. Uniqueness is clear. Letting $U = \spec R_0$ be any affine \'etale chart for $X_0$ (note that $X_0$ is an algebraic space since $f$ is representable), we have an affine \'etale chart
\[ C_{X_0/X} \times_{X_0} U = \spec \bigoplus_{d \geq 0} R_d\]
for $C_{X_0/X}$, which lifts to a canonical $\mb{G}_m$-equivariant affine \'etale chart
\[ V_n = \spec \bigoplus_{d \in \mb{Z}} R_{n, d} \]
for $C_n$ since $C_n$ is a nilpotent thickening of $C_{X_0/X}$, where the gradings are induced by the $\mb{G}_m$-action. From the flatness properties of the deformation to the normal cone, we deduce that the map $U \times_{\spec A} \spec A[t]/(t^n) = \spec R_0[t]/(t^n) \to V_n$ induces an isomorphism $R_{n, 0} \cong R_0[t]/(t^n)$, that $R_{n, d} = 0$ for all $d < 0$, and that $\bigoplus_d R_{n, d}$ is generated by $R_{n, 0}$ and $R_{n, d_0}$, where $d_0 = \min \{d > 0 \mid R_d \neq 0\}$ is the single weight of $\mb{G}_m$ acting on $C_{X_0/X}$. So $V_n$ is canonically identified with the normal cone of $U \times_{\spec A} \spec A[t]/(t^n)$ in $V_n$. But this is in turn canonically isomorphic to $C_{X_0/X} \times_{X_0} U$ since the normal cone is constant in the deformation to the normal cone. By uniqueness of this identification, it glues over all \'etale affine charts of $X_0$ to give the desired isomorphism $C_n \cong C_{X_0/X} \times_{\spec A} \spec A[t]/(t^n)$.
\end{proof}

\begin{proof}[Proof of Proposition \ref{prop:regularsliceresolution}]
Applying Lemma \ref{lem:conerecognition} to the morphism $\tilde{Z} \to Z$, we deduce that $\tilde{Z}$ is $\mb{G}_m$-equivariantly isomorphic to a line bundle over $D_{\alpha_i^\vee}(Z) = Y$. So by Corollary \ref{cor:regularsliceunstableres}, $\tilde{\chi}_Z \colon \tilde{Z} \to \Theta_Y^{-1}$ is a morphism of line bundles over $Y$ such that the preimage of the zero section is the zero section, and is therefore an isomorphism.
\end{proof}

\begin{proof}[Proof of Theorem \ref{thm:friedmanmorgansection}]
Let $Z^{ss, reg} \subseteq Z$ and $\tilde{Z}^{ss, reg} \subseteq \tilde{Z}$ denote the preimages of $\bun_{G, rig}^{ss, reg}$ in $Z$ and $\tilde{Z}$ respectively. Since the morphism $Z \to \bun_{G, rig}/E$ is smooth and $Z \to S$ is surjective, Propositions \ref{prop:regramgalois} and \ref{prop:regcodim} imply that $\tilde{Z}^{ss, reg} \to Z^{ss, reg}$ is a ramified Galois cover relative to $S$ with Galois group $W$, and that $Z^{ss, reg} \subseteq Z$ and $\tilde{Z}^{ss, reg} \subseteq Z^{ss} = \tilde{\chi}_Z^{-1}((\Theta_Y^{-1})^*)$ are big relative to $S$. So, since $Z$ and $\hat{Y}\sslash W$ are affine over $S$, there is a commutative diagram
\[
\begin{tikzcd}
\tilde{Z}^{ss} \ar[r] \ar[d, "\sim" {anchor=north, rotate=90}] & \spec {\pi_{\tilde{Z}^{ss}}}{\vphantom{p}}_*\mc{O} \ar[r] \ar[d, "\sim" {anchor=north, rotate=90}]  & \spec ({\pi_{\tilde{Z}^{ss, reg}}}{\vphantom{p}}_*\mc{O})^W \ar[r, "\sim"] \ar[d, "\sim" {anchor=north, rotate=90}] & Z \ar[d, "\chi_Z"] \\
(\Theta_Y^{-1})^* \ar[r] & \spec {\pi_{(\Theta_Y^{-1})^*}}{\vphantom{p}}_*\mc{O} \ar[r] & \spec ({\pi_{(\Theta_Y^{-1})^*}}{\vphantom{p}}_*\mc{O})^W \ar[r, "\sim"] & \hat{Y}\sslash W,
\end{tikzcd}
\]
where $\pi_X \colon X \to S$ denotes the structure map for any stack $X$ over $S$, the vertical arrows are isomorphisms by Proposition \ref{prop:regularsliceresolution}, and the horizontal arrows marked are isomorphisms either by construction or by ramified Galois descent for regular functions. So $\chi_Z$ is an isomorphism, which completes the proof of the theorem.
\end{proof}

\subsection{Proof of the simultaneous log resolution property} \label{subsection:friedmanmorganapplications}

The purpose of this subsection is to complete the proof of Theorem \ref{thm:introgrothendieckspringer} by proving (Corollary \ref{cor:simultaneouslogresolution}) that the elliptic Grothendieck-Springer resolution constructed in \S\ref{subsection:fundiag} is a simultaneous log resolution in the sense of Definition \ref{defn:simultaneouslogresolution}.

We first observe that Theorem \ref{thm:friedmanmorgansection} implies that the quotient $\hat{Y}\sslash W$ is an affine space, from which it follows that the coarse quotient $\chi$ is flat.

\begin{cor}[{(cf.\ \cite[Theorem 3.4]{looijenga76})}] \label{cor:looijengaiso}
The quotient $\hat{Y}\sslash W$ is an affine space bundle over $S$, on which $\mb{G}_m$ acts linearly and with positive weights.
\end{cor}
\begin{proof}
This is immediate from Theorem \ref{thm:friedmanmorgansection} since the claim holds for the slice $Z \to S$ by Proposition \ref{prop:inductionaffinebundle}.
\end{proof}

\begin{cor} \label{cor:chiflat}
The coarse quotient map $\chi \colon \bun_G \to (\hat{Y}\sslash W)/\mb{G}_m$ is flat.
\end{cor}
\begin{proof}
Since $(\hat{Y}\sslash W)/\mb{G}_m$ is smooth over $S$, hence regular, this follows from Proposition \ref{prop:chiflatness}.
\end{proof}

The rest of this subsection is devoted to exhibiting an open substack of $\bun_G$, dense in every fibre of $\chi$, over which the map $\psi \colon \tbun_G \to \bun_G$ is simply base change along $\Theta_Y^{-1}/\mb{G}_m \to (\hat{Y}\sslash W)/\mb{G}_m$; this is the only missing ingredient in the simultaneous log resolution property for \eqref{eq:introbunggrothendieckspringer}. In close analogy with the classical case, this open substack will consist of regular $G$-bundles. We first observe that Theorem \ref{thm:friedmanmorgansection} shows that these are in good supply.

\begin{prop} \label{prop:fmsliceregular}
Let $Z \to \bun_{G, rig}$ be as in Theorem \ref{thm:friedmanmorgansection}, let $z \colon \spec k \to Z$ be a geometric point not fixed under the $\mb{G}_m$-action, and let $\xi_{G, z} \to E_s$ be the corresponding $G$-bundle. Then the $G$-bundle $\xi_{G, z}$ is regular semistable in the sense of Definition \ref{defn:regularbundle}, and $\dim \mrm{Aut}(\xi_{G, z}) = l$.
\end{prop}
\begin{proof}
We can assume for simplicity that $S = \spec k$.

Since $z$ does not map to the image $0$ of the cone point in $\hat{Y} \sslash W$, $\xi_{G, z}$ is semistable by Proposition \ref{prop:chifibreunstable}. Since the morphism $\tilde{Z} \to Z$ can be identified with $\Theta_Y^{-1} \to \hat{Y} \sslash W$, it is finite over $z$, so $\dim \psi^{-1}(\xi_{G, z}) = 0$ and $\xi_{G, z}$ is regular.

To show that $\dim \mrm{Aut}(\xi_{G, z}) = l$, let $x$ be the image of $z$ in $\bun_{G, rig}$, and let $x'$ be its image in $\bun_{G, rig}/E$. By Lemma \ref{lem:sstranslationinvariance} below, any translate of $x$ is isomorphic to $x$, so the $E$-action on $\bun_{G, rig}$ restricts to an action on $\B \mrm{Aut}(x)$ with quotient $\B \mrm{Aut}(x')$. So we have
\[ \dim \mrm{Aut}(x') = \dim \mrm{Aut}(x) + 1 = \dim \mrm{Aut}(\xi_{G, z}) + 1.\]
Since the morphism $Z \to \bun_{G, rig}/E$ is smooth and $\bun_{G, rig}/E$ has dimension $-1$, the locally closed substack
\[ \B \mrm{Aut}(x') \times_{\bun_{G, rig}/E} Z \longhookrightarrow Z \]
has codimension $\dim \mrm{Aut}(x') - 1 = \dim \mrm{Aut}(\xi_{G, z})$. But it is clear from Theorem \ref{thm:friedmanmorgansection} that this is simply the $\mb{G}_m$-orbit of $z$, which has codimension $l$, so we are done.
\end{proof}

Given a point $y \colon \spec k \to Y$ over $s \colon \spec k \to S$, we write
\[ U_y = \prod_{\substack{\alpha \in \Phi_- \\ \alpha(y) = 0}} U_\alpha \subseteq R_u(B).\]
Note that the group scheme $\xi_T \times^T U_y$ on $E_s$ is canonically isomorphic to $U_y \times E_s$ once we fix a trivialisation of the associated $T_y$-bundle, where $T_y$ is the torus with character group $\mb{Z}\Phi_y$, $\Phi_y = \{\alpha \in \Phi \mid \alpha(y) = 0\}$ and $\xi_T$ is the $T$-bundle corresponding to $y$. We also write $U_y/[U_y, U_y] = \prod_{\alpha \in \Delta_y} U_{-\alpha}$.

\begin{lem} \label{lem:sstranslationinvariance}
Fix a geometric point $s \colon \spec k \to S$ and a semistable $G$-bundle $\xi_G \to E_s$. Then any translate of $\xi_G$ is isomorphic to $\xi_G$.
\end{lem}
\begin{proof}
For ease of notation, we may as well assume that $S = \spec k$. We need to show that for any $x \colon \spec k \to E$, we have $t_x^*\xi_G \cong \xi_G$, where $t_x \colon E \to E$ is the translation by $x$. Since $\xi_G$ is semistable, there exists a $B$-reduction $\xi_B$ of $\xi_G$ of degree $0$. We will show that $t_x^*\xi_B \cong \xi_B$ as $B$-bundles.

Writing $\xi_T = \xi_B \times^B T$ for the associated $T$-bundle, and $y \in Y$ for the point classifying $\xi_T$, by Lemma \ref{lem:unipotentreduction}, we have that $\xi_B$ reduces canonically to a $TU_y$-bundle $\xi_{TU_y}$. Moreover, we have $t_x^*\xi_T \cong \xi_T$ since $\xi_T$ has degree $0$ (this follows from translation invariance for degree $0$ line bundles). Fixing such an isomorphism and a trivialisation of the associated $T_y$-bundle, and hence an isomorphism $\xi_T \times^T U_y \cong E \times U_y$ as above, we have that the $U_y$-bundle $t_x^*\xi_{TU_y}/T$ is the image of the $U_y$-bundle $\xi_{TU_y}/T$ under the homomorphism
\[ H^1(E, U_y) \overset{t_x^*}\longrightarrow H^1(E, U_y) \overset{t}\longrightarrow H^1(E, U_y),\]
where the second morphism is induced by some element $t$ of $T$ acting on $U_y$ determined by our choice of isomorphism $t_x^*\xi_T \cong \xi_T$. Since the translation action of $E$ on $H^1(E, U_y)$ is trivial, since $H^1(E, U_y)$ is an affine variety, the morphism $t_x^*$ above is the identity. It follows that $t_x^*\xi_{TU_y} \cong \xi_{TU_y}$, and hence $t_x^* \xi_G \cong \xi_G$ as claimed.
\end{proof}

The next result is the analogue of \cite[\S 3.7, Theorem 2]{steinberg74} for elliptic Springer theory.

\begin{prop} \label{prop:regularautomorphisms}
Let $\xi_G \in \bun_G^{ss}$ be a semistable $G$-bundle on a geometric fibre $E_s$ of $E \to S$. Then $\dim \mrm{Aut}(\xi_G) \geq l$, and the following are equivalent.
\begin{enumerate}[(1)]
\item \label{itm:regaut1} The bundle $\xi_G$ is regular.
\item \label{itm:regaut2} $\dim \mrm{Aut}(\xi_G) = l$.
\item \label{itm:regaut3} For any degree $0$ reduction $\xi_B$ of $\xi_G$ to $B$ with associated $T$-bundle $\xi_T$ corresponding to $y \in Y$, the associated $\xi_T \times^T R_u(B)$-bundle $\xi_B/T$ is induced from a $U_y$-bundle with nontrivial associated $U_{-\alpha}$-bundles for $\alpha \in \Delta_y$.
\item \label{itm:regaut4} For some degree $0$ reduction $\xi_B$ of $\xi_G$ to $B$ with associated $T$-bundle $\xi_T$ corresponding to $y \in Y$, the associated $\xi_T \times^T R_u(B)$-bundle $\xi_B/T$ is induced from a $U_y$-bundle with nontrivial associated $U_{-\alpha}$-bundles for $\alpha \in \Delta_y$.
\end{enumerate}
Moreover, there is a unique $G$-bundle satisfying the above equivalent conditions in every geometric fibre of $\chi^{ss} \colon \bun_G^{ss} \to Y\sslash W$.
\end{prop}
\begin{proof}
Since the statement only concerns individual $G$-bundles on geometric fibres of $E \to S$, we may assume for simplicity that $S = \spec k$ for $k$ some algebraically closed field, and that $\xi_G \to E_s = E$ is defined over $k$.

To show that $\dim \mrm{Aut}(\xi_G) \geq l$, fix any reduction $\xi_B \in \tbun_G^{ss}$ of $\xi_G$ to a $B$-bundle of degree $0$, and write $\xi_T = \xi_B \times^B T$. Note that $\mrm{Aut}(\xi_B)$ is a closed subgroup of $\mrm{Aut}(\xi_G)$, so $\dim \mrm{Aut}(\xi_G) \geq \dim \mrm{Aut}(\xi_B)$. Moerover, there is a commutative diagram
\[
\begin{tikzcd}
\B \mrm{Aut}(\xi_B) \ar[r, hook] \ar[d] & \tbun_G^{ss} \ar[d] \\
\B \mrm{Aut}(\xi_T) \ar[r, hook] & \bun_T^0
\end{tikzcd}
\]
where the top and bottom horizontal arrows are locally closed immersions of codimension $\dim \mrm{Aut}(\xi_B)$ and $\dim \mrm{Aut}(\xi_T)$ respectively (since both target stacks have dimension $0$), and the right hand vertical arrow is smooth. It follows that $\dim \mrm{Aut}(\xi_G) \geq \dim \mrm{Aut}(\xi_B) \geq \dim \mrm{Aut}(\xi_T) = l$.

To prove the equivalence of \eqref{itm:regaut1}, \eqref{itm:regaut2}, \eqref{itm:regaut3} and \eqref{itm:regaut4}, we first remark that by Lemma \ref{lem:unipotentreduction}, for every degree $0$ $B$-bundle $\xi_B$ with image $y \in Y$, the associated $\xi_T \times^T R_u(B)$-bundle $\xi_B/T$ reduces canonically to $U_y$.

It is clear that $\eqref{itm:regaut3} \Rightarrow \eqref{itm:regaut4}$. We show that $\eqref{itm:regaut4} \Rightarrow \eqref{itm:regaut2} \Rightarrow \eqref{itm:regaut3}$ and $\eqref{itm:regaut4} \Rightarrow \eqref{itm:regaut1} \Rightarrow \eqref{itm:regaut3}$.

Assume \eqref{itm:regaut4} holds, and let $\xi_{U_y}$ be the reduction of $\xi_B/T$ to $U_y$. Observe that the set of $U_y$-bundles $\eta_{U_y}$ such that the induced $G$-bundle $\eta_G$ is regular with $\dim \mrm{Aut}(\eta_G) = l$ is open, and nonempty by Proposition \ref{prop:fmsliceregular} and the existence of reductions to $U_y$ remarked above. So we can find $\eta_{U_y}$ with these properties such that all the associated $U_{-\alpha}$-bundles are nontrivial for $\alpha \in \Delta_y$. We will show that $\xi_G \cong \eta_G$, from which we can deduce \eqref{itm:regaut1} and \eqref{itm:regaut2}, as well as the uniqueness statement of the proposition.

First, observe that $\Delta_y$ is a set of positive simple roots for the root system $\Phi_y = \{\alpha \in \Phi \mid \alpha(y) = 0\}$. In particular, the homomorphism
\begin{align*}
T &\longrightarrow \mb{G}_m^{\Delta_y} = T_y \\
t &\longmapsto (\alpha(t))_{\alpha \in \Delta_y}
\end{align*}
is surjective, with kernel $K_y$ of dimension $l - |\Delta_y|$. So $T = \mrm{Aut}(\xi_T)$ acts transitively on the subset of points in $\prod_{\alpha \in \Delta_y} H^1(E, U_{-\alpha})$ with nonzero projection to each factor. So, acting by automorphisms of $\xi_T$ if necessary, we may assume that $\eta_{U_y} \times^{U_y} U_y/[U_y, U_y] \cong \xi_{U_y} \times^{U_y} U_y/[U_y, U_y]$. To prove $\eta_{U_y} \cong \xi_{U_y}$, and hence $\eta_G \cong \xi_G$, it will suffice to show that the diagram
\begin{equation} \label{eq:regautpullback}
\begin{tikzcd}
\B \mrm{Aut}(\eta_{U_y}) \ar[r, hook] \ar[d] & \bun_{U_y} \ar[d] \\
\B \mrm{Aut}(\eta_{U_y} \times^{U_y} U_y/[U_y, U_y]) \ar[r, hook] & \bun_{U_y/[U_y, U_y]}
\end{tikzcd}
\end{equation}
is a pullback.

Observe that, since $K_y$ acts trivially on $U_y$ by definition, it also acts trivially on $\bun_{U_y}$, so $\bun_{U_y}/T$ is a $K_y$-gerbe over $\bun_{U_y}/T_y$. Since $\bun_{U_y}/T$ embeds into $\bun_B^0$ as the fibre over $y \in Y$, we therefore have
\[ l = \dim \mrm{Aut}(\eta_G) \geq \dim \mrm{Aut}(\eta_B) \geq \dim \mrm{Aut}(\eta_{U_y}) + l - |\Delta_y|\]
and hence $\dim \mrm{Aut}(\eta_{U_y}) \leq |\Delta_y|$. But since the top and bottom arrows of \eqref{eq:regautpullback} are locally closed immersions of codimensions $\dim \mrm{Aut}(\eta_{U_y})$ and $\dim \mrm{Aut}(\eta_{U_y} \times^{U_y} U_y/[U_y, U_y]) = |\Delta_y|$ respectively and the right vertical morphism is smooth, we have that $\dim \mrm{Aut}(\eta_{U_y}) \geq |\Delta_y|$, and hence $\dim \mrm{Aut}(\eta_{U_y}) = |\Delta_y|$. So the locally closed immersion of $\B \mrm{Aut}(\eta_{U_y})$ into the pullback induced by \eqref{eq:regautpullback} has codimension $0$. But since $U_y$ and $U_y/[U_y, U_y]$ are unipotent, this is actually a closed immersion by \cite[Proposition 2.4.6]{davis19a}, hence an isomorphism since the pullback is smooth and connected. This completes the proof that $\eqref{itm:regaut4} \Rightarrow \eqref{itm:regaut1}$ and $\eqref{itm:regaut4} \Rightarrow \eqref{itm:regaut2}$.

We prove $\eqref{itm:regaut1} \Rightarrow \eqref{itm:regaut3}$ via the contrapositive. Assume that $\eqref{itm:regaut3}$ is false, and choose a degree $0$ reduction $\xi_B$ with with associated $T$-bundle $\xi_T$ corresponding to $y \in Y$ induced from $\xi_{T U_y}$, and $\alpha \in \Delta_y$ such that the associated $U_{-\alpha}$-bundle $\xi_{U_{-\alpha}}$ is trivial. The space $C^{s_\alpha}$ of sections of
\[ \xi_{T U_y} \times^{T U_y} B s_\alpha B/B = \xi_{T U_y} \times^{T U_y} R_u(B)/(R_u(B) \cap s_\alpha R_u(B) s_\alpha)\]
embeds as a locally closed subscheme of $\psi^{-1}(\xi_G)$. But the image of $U_{-\alpha}$ in $R_u(B)/(R_u(B) \cap s_\alpha R_u(B) s_\alpha)$ is $T U_y$-invariant, so gives a closed immersion
\[ \xi_{U_{-\alpha}} = E \times U_{-\alpha} \longhookrightarrow \xi_{T U_y} \times^{T U_y} R_u(B)/(R_u(B) \cap s_\alpha R_u(B) s_\alpha)\]
and hence a locally closed immersion $\mb{A}^1_k \hookrightarrow C^{s_\alpha} \hookrightarrow \psi^{-1}(\xi_G)$, from which we deduce $\dim \psi^{-1}(\xi_G) > 0$. So $\xi_G$ is not regular.

Finally, to prove that $\eqref{itm:regaut2} \Rightarrow \eqref{itm:regaut3}$, note that any reduction $\xi_{T U_y} \in \bun_{U_y}/T \subseteq \bun_B^0$ of a bundle $\xi_G$ with $\dim \mrm{Aut}(\xi_G) = l$ must satisfy $\dim \mrm{Aut}(\xi_{T U_y}) \leq l$. So
\[ \B \mrm{Aut}(\xi_{T U_y}) \longhookrightarrow \bun_{U_y}/T \]
is a locally closed immersion of codimension $\leq 0$, hence an open immersion. Since $\bun_{U_y}/T$ is irreducible, this open substack meets the open substack of points with nontrivial associated $U_{-\alpha}$-bundles for all $\alpha \in \Delta_y$, so $\xi_{T U_y}$ itself must have nontrivial associated $U_{-\alpha}$-bundle, and we are done.
\end{proof}

As the terminology suggests, the regular semistable and regular unstable $G$-bundles can be grouped together into a single open substack of $\bun_G$. In what follows, we define $\bun_G^{reg} \subseteq \bun_G$ to be the union over all open substacks $U \subseteq \bun_G$ such that the morphism
\begin{equation} \label{eq:regulardensity1}
\psi^{-1}(U) \longrightarrow U \times_{(\hat{Y}\sslash W)/\mb{G}_m} \Theta_Y^{-1}/\mb{G}_m
\end{equation}
is an isomorphism.

\begin{prop} \label{prop:regulardensity}
The open substack $\bun_G^{reg} \subseteq \bun_G$ is dense in every geometric fibre of $\chi \colon \bun_G \to (\hat{Y}\sslash W)/\mb{G}_m$, and big relative to $S$.
\end{prop}
\begin{proof}
Let $\{\alpha_1, \ldots, \alpha_n\} \subseteq \Delta$ denote the set of special roots, and let $Z_1, \ldots, Z_n$ be the corresponding regular slices of $\bun_{G, rig}$. Let $U \subseteq \bun_G$ be the preimage in $\bun_G$ of the union of the images of $Z_i \to \bun_{G, rig}/E$. Note that this is open since each $Z_i \to \bun_{G, rig}/E$ is smooth. By Proposition \ref{prop:regularsliceresolution} and Theorem \ref{thm:friedmanmorgansection}, it is clear that \eqref{eq:regulardensity1} is an isomorphism, so $U \subseteq \bun_G^{reg}$. Note that $U$ contains all regular unstable bundles by construction and that Propositions \ref{prop:fmsliceregular} and \ref{prop:regularautomorphisms} imply that $U$ also contains all regular semistable bundles, so the same is true for $\bun_G^{reg}$.

We first show that for every $x \in (\hat{Y}\sslash W)/\mb{G}_m$, $\chi^{-1}(x) \cap \bun_G^{reg}$ is dense in $\chi^{-1}(x)$. For $x$ not in the zero section of $\hat{Y}\sslash W \to S$, this is clear since Proposition \ref{prop:regularautomorphisms} implies that $\chi^{-1}(x) \cap \bun_G^{reg}$ is open and nonempty, and that $\chi^{-1}(x)$ is irreducible. For $x$ in the zero section, note that since $\hat{Y} \sslash W$ is regular, the inclusion $\{x\} \hookrightarrow \hat{Y}\sslash W$ is a local complete intersection morphism. So by Corollary \ref{cor:chiflat}, $\chi^{-1}(x)$ is a local complete intersection stack, hence of pure dimension. Since $\chi^{-1}(x)$ is the locus of unstable bundles on some fibre of $E \to S$, $\bun_G^{reg}$ meets every irreducible component of $\chi^{-1}(x)$ by construction, so $\bun_G^{reg} \cap \chi^{-1}(x)$ is dense in $\chi^{-1}(x)$ as claimed.

Finally, notice that $\bun_G^{ss, reg} \subseteq \bun_G^{reg}$, and the complement of $\bun_G^{ss}$ in $\bun_G$ has codimension at least $2$, so $\bun_G^{reg} \subseteq \bun_G$ is big by Proposition \ref{prop:regcodim}.
\end{proof}

\begin{cor} \label{cor:simultaneouslogresolution}
The elliptic Grothendieck-Springer resolution
\[
\begin{tikzcd}
\tbun_{G} \ar[r, "\psi"] \ar[d, "\tilde{\chi}"'] & \bun_{G} \ar[d, "\chi"] \\
\Theta_Y^{-1}/\mb{G}_m \ar[r, "q"] & (\hat{Y}\sslash W)/\mb{G}_m
\end{tikzcd}
\]
and its rigidification are simultaneous log resolutions in the sense of Definition \ref{defn:simultaneouslogresolution}.
\end{cor}
\begin{proof}
It is enough to prove the claim for the non-rigidified diagram: the statement for the rigidification follows immediately by descent along the gerbe $\bun_G \to \bun_{G, rig}$. For the non-rigidified diagram, condition \eqref{itm:simultaneouslogres1} of Definition \ref{defn:simultaneouslogresolution} holds by Propositions \ref{prop:kmproper} and \ref{prop:chiflatness} and Corollary \ref{cor:chiflat}, condition \eqref{itm:simultaneouslogres2} holds by Proposition \ref{prop:regulardensity}, and \eqref{itm:simultaneouslogres3} holds by Proposition \ref{prop:kmdivisor} \eqref{itm:kmdivisor1}, Proposition \ref{prop:kmblowdown}, and Corollary \ref{cor:chitildedivisor}.
\end{proof}

\bibliography{bibliography}

\end{document}